\documentclass[reqno,10pt,centertags]{amsart} 
\usepackage{amsmath,amsthm,amscd,amssymb,latexsym,upref,enumerate,hyperref}
\date{\today}

\numberwithin{equation}{section}

\input epsf
\usepackage{epsfig}
\usepackage[T2A,OT1]{fontenc}
\usepackage[ot2enc]{inputenc}
\usepackage[russian,english]{babel}
\usepackage[margin=3.0 cm]{geometry}
\usepackage{epic,eepicemu}
\usepackage[all]{xy}
\usepackage[center]{caption}

\usepackage{xcolor}

\newcommand{\bbD}{{\mathbb{D}}}

\newcommand{\bbR}{{\mathbb{R}}}

\newcommand{\bbZ}{{\mathbb{Z}}}
\newcommand{\bbC}{{\mathbb{C}}}

\newcommand{\bbT}{{\mathbb{T}}}

\newcommand{\bbE}{{\mathbb{E}}}

\newcommand{\cA}{{\mathcal{A}}}
\newcommand{\cB}{{\mathcal{B}}}

\newcommand{\cD}{{\mathcal{D}}}
\newcommand{\cE}{{\mathcal{E}}}
\newcommand{\cF}{{\mathcal{F}}}
\newcommand{\cG}{{\mathcal{G}}}
\newcommand{\cH}{{\mathcal{H}}}
\newcommand{\cJ}{{\mathcal{J}}}
\newcommand{\cK}{{\mathcal{K}}}
\newcommand{\cL}{{\mathcal{L}}}
\newcommand{\cM}{{\mathcal{M}}}
\newcommand{\cN}{{\mathcal{N}}}

\newcommand{\cP}{{\mathcal{P}}}

\newcommand{\cR}{{\mathcal{R}}}
\newcommand{\cS}{{\mathcal{S}}}
\newcommand{\cT}{{\mathcal{T}}}
\newcommand{\cU}{{\mathcal{U}}}
\newcommand{\cV}{{\mathcal{V}}}
\newcommand{\cW}{{\mathcal{W}}}
\newcommand{\cZ}{{\mathcal{Z}}}
\newcommand{\cX}{{\mathcal{X}}}
\newcommand{\cY}{{\mathcal{Y}}}

\newcommand{\bA}{{\mathbf{A}}}

\newcommand{\fA}{{\mathfrak{A}}}
\newcommand{\fB}{{\mathfrak{B}}}

\newcommand{\fD}{{\mathfrak{D}}}

\newcommand{\fa}{{\mathfrak{a}}}

\newcommand{\fc}{{\mathfrak{c}}}
\newcommand{\fd}{{\mathfrak{d}}}

\newcommand{\fj}{{\mathfrak{j}}}

\newcommand{\fs}{{\mathfrak{s}}}
\newcommand{\fr}{{\mathfrak{r}}}

\newcommand{\fv}{{\mathfrak{v}}}

\newcommand{\fz}{{\mathfrak{z}}}

\newcommand{\sE}{{\mathsf{E}}}
\newcommand{\sF}{{\mathsf{F}}}

\newcommand{\SL}{{\mathrm{SL}}}
\newcommand{\SU}{{\mathrm{SU}}}
\newcommand{\PSU}{{\mathrm{PSU}}}

\renewcommand{\a}{\alpha}
\renewcommand{\b}{\beta}
\newcommand{\g}{\gamma}
\newcommand{\e}{{\epsilon}}
\newcommand{\s}{{\sigma}}
\newcommand{\vk}{\varkappa}
\newcommand{\vt}{\vartheta}
\renewcommand{\l}{\lambda}

\newcommand{\z}{\zeta}
\def\u{\upsilon}
\def\U{\Upsilon}

\newcommand{\pd}{{\partial}}

\renewcommand{\Re}{\text{\rm Re}\,}
\renewcommand{\Im}{\text{\rm Im}\,}

\newcommand{\tr}{\text{\rm tr}\,}

\newcommand{\dist}{\text{\rm dist}}

\newcommand{\sgn}{\text{\rm sgn}}

\newcommand{\rank}{\text{\rm rank}}
\newcommand{\supp}{\text{\rm supp}}

\newcommand{\ac}{\text{\rm ac}}

\DeclareMathOperator{\spann}{span}

\newtheorem{theorem}{Theorem}[section]

\newtheorem{lemma}[theorem]{Lemma}
\newtheorem{proposition}[theorem]{Proposition}

\newtheorem{corollary}[theorem]{Corollary}

\newtheorem{conjecture}[theorem]{Conjecture}
\theoremstyle{definition}
\newtheorem{definition}[theorem]{Definition}

\newtheorem{remark}[theorem]{Remark}
\newtheorem{problem}[theorem]{Problem}

\newtheorem{example}[theorem]{Example}

\date{\today}

\title[Reflectionless canonical systems, II]
{Reflectionless canonical systems, II.  Almost periodicity and character-automorphic Fourier transforms}
\author{Roman Bessonov, Milivoje Luki\'c, and Peter  Yuditskii}

\address{
\begin{flushleft}
Roman Bessonov: bessonov@pdmi.ras.ru\\\vspace{0.1cm}
St.\,Petersburg State University\\  
Universitetskaya nab. 7-9, 199034 St.\,Petersburg, RUSSIA\\
\vspace{0.1cm}
St.\,Petersburg Department of Steklov Mathematical Institute\\ Russian Academy of Sciences\\
Fontanka 27, 191023 St.Petersburg,  RUSSIA
\end{flushleft}
}
\address{
\begin{flushleft}
Milivoje Luki\'c: milivoje.lukic@rice.edu\\\vspace{0.1cm}
Rice University, Department of Mathematics MS-136, 
Houston, TX 77251-1892, USA.
\end{flushleft}
}
\address{
\begin{flushleft}
Peter Yuditskii: Petro.Yudytskiy@jku.at\\\vspace{0.1cm}
Abteilung f\"ur Dynamische Systeme und Approximationstheorie, Johannes Kepler Universit\"at Linz, A-4040 Linz, Austria
\end{flushleft}
}

\begin{document}

\begin{abstract}
We develop a comprehensive theory of reflectionless canonical systems with an arbitrary Dirichlet-regular Widom spectrum with the Direct Cauchy Theorem property. This generalizes, to an infinite gap setting, the constructions of finite gap quasiperiodic (algebro-geometric) solutions of stationary integrable hierarchies. Instead of theta functions on a compact Riemann surface, the construction is based on reproducing kernels of character-automorphic Hardy spaces in Widom domains with respect to Martin measure. We also construct unitary character-automorphic Fourier transforms which generalize the Paley--Wiener theorem. Finally, we find the correct notion of almost periodicity which holds for canonical system parameters in Arov gauge, and we prove generically optimal results for almost periodicity for Potapov--de Branges gauge, and Dirac operators.
\end{abstract}

\maketitle

\tableofcontents

\section{Introduction}

For one-dimensional Schr\"odinger operators with spectrum $\sE$, and for other well-studied classes of self-adjoint and unitary operators including Dirac, Jacobi, and CMV operators, the reflectionless property is a certain pseudocontinuation relation between two Weyl functions which encode the two half-line restrictions of the operator. This was originally observed as a property of periodic operators and finite gap quasiperiodic operators, and has since become ubiquitous in spectral theory;  by Kotani theory \cite{Kot82}, the reflectionless property is a general feature of ergodic operators with zero Lyapunov exponent on the spectrum. By Remling \cite{Rem11}, it is a general property of right limits of operators with absolutely continuous spectrum.

The inverse spectral theory of reflectionless operators was originally considered for finite gap spectra, in the algebraic language associated with compact Riemann surfaces (double covers of $\bbC \setminus \sE$) \cite{DubMatNov,GesHol}. This theory was applied by finite gap approximation to the periodic case \cite{McKvanM75,McKTru76,MarOst75,MarOst87} and some almost periodic cases \cite{Chul81,PT1,PT2,Craig,Egorova,GesKriTes,JM,Lev,LevSav}. 

The finite gap construction was generalized by Sodin--Yuditskii \cite{SY97} to the more general setting of bounded Dirichlet-regular Widom sets $\sE$ with the DCT property (the definitions of these properties will be given below).  The core of the approach are intrinsic Fourier series representations of the character-automorphic Hardy spaces on the domain $\hat \bbC \setminus \sE$. The corresponding basis is formed using the Complex Green function and the reproducing kernels with respect to infinity, which is an internal point of the domain. The corresponding Fourier representations transform the multiplication operator by independent variable into Jacobi matrices.
The triumph of the theory is the almost periodicity of coefficients of Jacobi matrices, which follows from continuity of explicit representations involving trace formulas and a representation of translation as a linear flow with respect to character.  The theory was also applied to Schr\"odinger operators with semibounded spectra of finite gap length, by the standard finite gap approximation approach \cite{SY95}.

In this paper, we construct almost periodic parameters for spectral data on arbitrary Dirichlet-regular Widom set $\sE \subset \bbR$ with DCT, without any gap moment conditions or semiboundedness.  In contrast to the construction in \cite{SY97}, we have to build Fourier integrals instead of Fourier series representations.
Infinity still plays the role of the distinguished point,  but $\sE$ is an unbounded set, so $\infty$ is a \textit{boundary point} of the domain $\Omega = \bbC \setminus \sE$. In particular, the Complex Martin function must substitute the Complex Green function in this new construction. A passage from discrete systems to continuous ones always presents essential obstacles related to differentiability, but in the current setting it was not originally clear what is this ``almost periodic object" which corresponds to the chosen spectral data, and especially in which sense it is ``almost periodic".

We will show that the correct setting is provided by canonical systems in Arov gauge, which uses normalizations at a point $z_0$ in the upper half plane (we fix $z_0=i$ in our presentation), which is always an interior point of the domain $\Omega = \bbC \setminus \sE$. We also apply this theory to other well known gauges, namely, canonical systems in Potapov--de Branges gauge (see \cite{dB} and recent works \cite{Remling, BD, Romanov}), and Dirac operators \cite{LevSar}: we will explain that these other gauges don't always give almost periodic data, and give sufficient conditions for almost periodicity which are generically optimal. Note also that our approach doesn't use finite gap approximation: everything is constructed directly for the domain $\Omega$.

Our results can further be motivated through Paley--Wiener theory, the multiplicative theory of $j$-contractive matrix functions \cite{P60,EP73,GoKr}, and through a general perspective on nesting Weyl disks for one-dimensional operators. The first motivation doesn't even require spectral theory. Recall that the standard Hardy space $\cH^2(\bbC_+)$ can be viewed as a closed subspace of $L^2(\bbR)$ by passing to boundary values, and recall the following Paley--Wiener theorem: $\cH^2(\bbC_+)$ is the image of $L^2((0,\infty))$ in the Fourier transform.
According to de Branges, this theorem was the origin of his theory, see \cite[Preface]{dB}. We generalize the Paley--Wiener theorem to a character-automorphic setting with the domain $\Omega = \bbC \setminus \sE$. This requires several constructions.

Let $\sE$ be an unbounded proper closed subset of $\bbR$ such that $\Omega$ is Dirichlet regular.  The \emph{symmetric Martin function at $\infty$} is a positive harmonic function $M$ on $\Omega$ with the symmetry $M(\bar z) = M(z)$ which vanishes continuously on $\sE$; it is determined uniquely up to normalization \cite{Anc79,Ben80}. The limit $\lim_{y\to \infty} M(iy) / y$ is convergent, and it can be zero or strictly positive. This gives an important dichotomy: $\Omega$ is said to be of Akhiezer-Levin (A-L) type  \cite{AkhiezerLevin} if 
\begin{equation}\label{ALdefn}
\lim_{y\to\infty} \frac{M(iy)} y > 0.
\end{equation}
In the A-L case, $M(z)$ is also called the Phragm\'en--Lindel\"of function by Koosis \cite{Koosis}.  Among finite gap sets $\sE$ (an algebraic setting is possible), it holds precisely for those which are unbounded both above and below, i.e., those where $\infty$ corresponds to two different accessible boundary points/prime ends \cite[Section VI.3]{GarMar}, \cite[Section 2.4]{Pom92}. In the general case, the A-L condition measures that distinction for the minimal Martin boundary of the domain: $\infty$ corresponds to two minimal Martin boundary points if \eqref{ALdefn} holds and a single point if \eqref{ALdefn} fails, see Section \ref{secmb}.

The symmetric Martin function $M$ extends to a subharmonic function on $\bbC$, so its distributional Laplacian is a positive measure, called the Martin measure, $\vartheta = \frac 1{2\pi} \Delta M$.
If $\sE$ is a Widom set, $\vartheta$ is mutually absolutely continuous with Lebesgue measure on $\sE$.

Locally on $\Omega$, $M = \Im \Theta$ for some analytic function $\Theta$. Since $\Omega$ is multiply connected, $\Theta$ is multi-valued: its analytic continuation $\Theta \circ \gamma$ along a closed loop $\gamma \in \pi_1(\Omega)$ obeys 
\begin{equation}\label{MartinCharacter}
\Theta \circ \gamma = \Theta + \eta(\gamma), \qquad \forall \gamma \in \pi_1(\Omega)
\end{equation}
where $\eta : \pi_1(\Omega) \to \bbR$ is an additive character, i.e., $\eta(\gamma_1 \gamma_2) = \eta(\gamma_1) + \eta(\gamma_2)$.  Note that $\Theta(z)$ is defined up to an affine transform $\Theta(z)\mapsto a\Theta(z)+b$, $a>0, b\in\bbR$. Since different normalizations can be natural in different settings, we prefer to not fix a normalization and write $\Theta(i) = \theta_r + i \theta_i$. Of course, an affine change of $\Theta$ also affects $M$ and $\eta$.

We also work with multi-valued meromorphic functions $f$ on $\Omega = \bbC \setminus \sE$ such that $\lvert f \rvert$ is single-valued. Such functions $f$ are character-automorphic, i.e., there exists a character (additive map) $\alpha: \pi_1(\Omega) \to \bbR / \bbZ$ such that
\begin{equation}\label{7jul111}
f\circ\gamma=e^{2\pi i\a(\gamma)} f,\qquad \forall \gamma \in \pi_1(\Omega).
\end{equation}
All statements about multi-valued functions on $\Omega$ can also be expressed  in terms of lifts to the universal cover $\bbD$ via the uniformization $\Omega \simeq \bbD / \Gamma$,  $\Gamma \cong \pi_1(\Omega)$; in particular, we say that $f$ has bounded characteristic if its lift $F$ to $\bbD$ has bounded characteristic, i.e., $F = F_1 / F_2$ for some $F_1, F_2 \in H^\infty(\bbD)$. If, in addition, $F_2$ is outer, we say that $f$ is of Smirnov class.

Since $\Omega$ is a Denjoy domain, functions on $\Omega$ accept an antilinear involution $(\dots)_\sharp$ defined by 
\begin{equation}\label{fsharpdefinition}
f_\sharp(z) = \overline{f(\bar z)}.
\end{equation}
This involution doesn't change the character. 
We will also use another involution, related to the notion of pseudocontinuation: if $f$ has bounded characteristic, we denote by $f_\flat$ a function of bounded characteristic such that the nontangential boundary values from above and below obey
\begin{equation} \label{fflatdefinition}
f_\flat(\xi \pm i0) = f( \xi \mp i0), \qquad \text{a.e. }\xi \in \sE.
\end{equation}
The pseudocontinuation is very far from being a general property of functions of bounded characteristic, and we will discuss this later.

We denote the character group by $\pi_1(\Omega)^*$. For any character $\alpha \in \pi_1(\Omega)^*$, we define a \emph{character-automorphic Hardy space with respect to Martin measure}, denoted $\cH^2_\Omega(\alpha)$ or simply $\cH^2(\alpha)$, as the set of Smirnov class functions $f$ with character $\alpha$ with the norm
\begin{equation}\label{8oct5}
\lVert f \rVert_{\cH^2_\Omega(\alpha)}^2 =  \int_\sE ( \lvert f(\xi+i0) \rvert^2 +  \lvert f(\xi- i 0) \rvert^2 ) d\vartheta(\xi)  < \infty.
\end{equation}
Passing from $f$ to its boundary values gives an isometric embedding $\cH_\Omega^2(\alpha) \subset L^2(\sE, d\vartheta)^2$.

If $\sE$ is a Widom set,  $\cH^2(\alpha)$ is a nontrivial reproducing kernel Hilbert space for any $\alpha$ (see Section~\ref{sectionHardySpaces}). In particular, we will use the $L^2$-normalized reproducing kernel at $i$, denoted $K^\alpha$, which obeys:
\begin{equation}\label{KKsharpdefn}
\langle f, K^\alpha \rangle = \frac{f(i)}{K^\alpha(i)}, \quad \forall f \in \cH_\Omega^2(\alpha).
\end{equation}

With two sampling functions  $\fr(\a):=-\log K^\a(i)$ and  $\fs(\a):=K_\sharp^{\a}(i)/K^{\a}(i)$, for a fixed $\a\in\pi_1(\Omega)^*$ we associate two measures on $\bbR$
\begin{equation}\label{mualphadefinition}
\mu^\alpha ((\ell_1, \ell_2]) = (\ell_2 - \ell_1) \Im \Theta(i) +\fr(\a-\eta\ell_2)-\fr(\a-\eta \ell_1),
\end{equation}
\begin{equation}\label{mualphadefinition100}
\mu_1^\alpha ((\ell_1, \ell_2]) = 
\frac{\fs(\a-\eta\ell_1)-\fs(\a-\eta\ell_2)}{2}+\int_{\ell_1}^{\ell_2} \fs(\a-\eta l)d\mu^\a(l).
\end{equation}
These are the almost periodic parameters solving our inverse spectral problem:

\begin{theorem} \label{theorem11}
For any unbounded closed proper subset $\sE \subset \bbR$ which is Dirichlet-regular, obeys the Widom condition and DCT, for any $\alpha \in \pi_1(\Omega)^*$:
\begin{enumerate}[(a)]
\item $\mu^\alpha$ is a positive measure on $\bbR$;
\item The complex measure $\mu_1^\alpha$ is absolutely continuous with respect to $\mu^\alpha$ and its  Radon--Nikodym derivative $\fa^\alpha$, defined by
\begin{equation}\label{aalphadefinition}
d\mu_1^\a = \fa^\a\, d\mu^\a,
\end{equation}
obeys 
  $\lvert \fa^\alpha(\ell) \rvert \le 1$ for $\mu^\alpha$-a.e.\ $\ell \in \bbR$;
\item the measures $\mu^\alpha, \mu_1^\a$ are almost periodic in the sense that for every piecewise continuous compactly supported test function $h$,  the functions $g(\ell) = \int h(l + \ell) d\mu(l)$, $g_1(\ell) = \int h(l + \ell) d\mu_1(l)$ are almost periodic with frequency vector $\eta$ (for any sequence $\ell_n \to \infty$ such that $\eta \ell_n \to 0$ in $\pi_1(\Omega)^*$, $g(\cdot + \ell_n) \to g$ and  $g_1(\cdot + \ell_n) \to g_1$ uniformly on $\bbR$).
\end{enumerate}
\end{theorem}

Most of the earlier results establish uniform almost periodicity of Schr\"odinger and Dirac operator data, with $L^\infty$ bounds for associated potentials. On the other hand the theory of  one-dimensional Schr\"odinger \textit{periodic} operators with $L^2$ potentials \cite{MarOst75, MarOst87} looks very similar, as a certain branch of the same general inverse spectral theory. Moreover the mentioned recent growth of interest in a unified approach to all such operators via Potapov--de Branges  canonical systems motivates an extension of the concept of almost periodicity to this general setting. One approach could possibly be in terms of \textit{resolvent almost periodicity}, in which one would use a concept of compactness of shifts of resolvents in the operator norm. But as we will see, say when A-L fails, the measures of two canonical systems of the same isospectral class are possibly mutually singular. This looks like a very strong obstacle on the way to developing this approach: the corresponding isospectral operators can not be treated as operators acting in the same space.  Our approach, which seems to be a certain breakthrough in the area, is based on  the concept of \emph{almost periodic measures}.

Almost periodicity of measures is commonly described by convolution with some class of test functions \cite{ArgGil74,ArgGil90}; in particular, strong almost periodicity of measures uses $h \in C_c(\bbR)$. Note that our conclusion is strictly stronger than strong almost periodicity and also includes, e.g., characteristic functions of intervals, $h = \chi_{(0,\ell]}$, for any $\ell > 0$.  

Since $|\fa^\a(\ell)|\le 1$, we can define the nonnegative matrices
\[
A^\alpha(\ell) =\begin{pmatrix} 1 & - \overline{\fa^\alpha(\ell)} \\ -\fa^\alpha(\ell) & 1 \end{pmatrix}
\]
and introduce the Hilbert space $\overline{ \sqrt{A^\alpha} L^2(\bbR, \bbC^2, d\mu^\alpha)}$, with closure taken in $L^2(\bbR, \bbC^2, d\mu^\alpha)$. We also use a Complex Green function $\Phi$ with a zero at $i$ and denote by $\beta_\Phi$ its character.  This is the promised generalization of the Paley--Wiener theorem:

\begin{theorem} \label{theorem12}
For any unbounded closed proper subset $\sE \subset \bbR$ which is Dirichlet-regular, obeys the Widom condition and DCT, for any $\alpha \in \pi_1(\Omega)^*$, the map $\cF^\alpha$ defined by
\begin{equation}\label{Falphadefinition}
(\cF^\alpha \hat f )(z) = \frac{z+i}{\sqrt 2 \Phi_\sharp(z)} \int e^{i\Theta(z) \ell} \begin{pmatrix} K_\sharp^{\alpha- \eta \ell}(z) & K^{\alpha- \eta \ell}(z) \end{pmatrix} \sqrt{A^\alpha(\ell)} \hat f (\ell) d\mu^\alpha(\ell)
\end{equation}
for compactly supported $\hat f $ extends by continuity to a unitary operator
\begin{equation}\label{5jul1b}
\cF^{\alpha}: \overline{ \sqrt{A^\alpha} L^2(\bbR, \bbC^2, d\mu^\alpha)} \to L^2(\sE, d\vartheta)^2.
\end{equation}
For any $\ell \in\bbR$, $\cF^\alpha$ maps $\overline{ \sqrt{A^\alpha} L^2([\ell,\infty), \bbC^2, d\mu^\alpha)}$ bijectively to $e^{i\ell\Theta} \cH^2_\Omega(\alpha - \beta_\Phi - \eta \ell)$.
\end{theorem}

The case $\ell=0$ pertains directly to the space $\cH^2(\alpha - \beta_\Phi)$, but the $\ell$-dependence shows that $\cF^\alpha$ conjugates translation to a linear flow in $\ell$. The spaces $\cH^2(\alpha - \beta_\Phi) \ominus e^{i\ell\Theta} \cH^2_\Omega(\alpha - \beta_\Phi - \eta \ell)$ will play an important role in the proofs.

Note that $\lvert \fa^\alpha \rvert = 1$ if and only if $\rank A^\alpha = 1$.  Although $\overline{ \sqrt{A^\alpha} L^2(\bbR, \bbC^2, d\mu^\alpha)}$ is, by construction, a subset of $L^2(\bbR, \bbC^2, d\mu^\alpha)$ and consists of vector-valued functions, if $\rank A^\alpha = 1$, this effectively flattens the vector values to scalars.  Thus, the rank determines whether our almost periodic Hilbert space model contains vector-valued functions (like it does for Dirac operators) or scalar-valued functions (like it does for Schr\"odinger operators). We study this dichotomy and prove that the rank $1$ case happens uniformly exactly for sets $\sE$ which do \emph{not} obey the A--L condition \eqref{ALdefn}. In particular, the rank $2$ case does not occur for semibounded sets. 

The most common sufficient criterion for the A-L condition is the finite logarithmic gap length condition,
\begin{equation}\label{21apr4}
  \int_{\bbR\setminus\sE}\frac{|x|d x}{1+x^2}<\infty.
\end{equation}
When \eqref{21apr4} holds, the Fourier transform can be redefined with the domain $L^2(\bbR,\bbC^2)$ by using the Complex Martin function and renormalized boundary limits of reproducing kernels.
\begin{lemma} \label{lemma13}
If \eqref{21apr4} holds, then the following limits exist for all $\ell$,
\begin{equation}\label{3aug201}
L^\a_\pm(z,\ell)=\lim_{y\to+\infty}\frac{\mp k^{\a-\b_{\Phi}-\eta \ell}(z,\pm iy)}{k^{\a-\eta\ell}(\pm i,\pm iy)}.
\end{equation}
These functions have pseudocontinuations, i.e., the functions $L^\a_{\pm,\flat}(z,\ell)$ exist. For the matrix
\begin{equation}\label{3aug201b}
\cL^\a(z,\ell)=
\begin{pmatrix}
L^\a_{-,\flat}(z,\ell)&L^\a_{+,\flat}(z,\ell)\\
L^\a_-(z,\ell)& L_+^\a(z,\ell)
\end{pmatrix},
\end{equation}
the following limit exists,
\begin{equation}\label{3aug202}
\fd^\a(\ell)=\lim_{y\to+\infty}\det\cL^\a(iy,\ell).
\end{equation}
Moreover, all the limits are almost periodic in $\ell$.
\end{lemma}

In this special case,  our Fourier integral reduces to a much more familiar form, which could be interpreted as a limit case from a discrete or a finite gap version.

\begin{theorem}\label{th3aug201}
If \eqref{21apr4} holds, the map
$$
(\tilde \cF^\a \hat g)(z)=\int_0^\infty \frac{e^{i\Theta(z)\ell}}{\sqrt{\fd^\a(\ell)}}
\begin{pmatrix}
L^\a_-(z,\ell)& L_+^\a(z,\ell)
\end{pmatrix}\hat g(\ell) d\ell,\quad \hat g\in L^2([0,\infty),\bbC^2).
$$
defines a unitary Fourier transform  acting from $L^2([0,\infty),\bbC^2)$ to $\cH^2(\a-\b_{\Phi})$. 
\end{theorem}

The operator  $\cF^{\alpha}$ is precisely an ``eigenfunction expansion" for a reflectionless canonical system in Arov gauge.  To explain this, we use $j$-contractive matrix functions. 

\begin{definition}
Let $j$ be a $2\times 2$ matrix such that $j = j^* = j^{-1}$. An entire $2\times 2$ matrix valued function $\cA(z)$ is called $j$-inner if it obeys $j-\cA(z) j \cA(z)^* \ge 0$ for $z \in \bbC_+$ and $j - \cA(z) j \cA(z)^* = 0$ for $z \in \bbR$. 
\end{definition}

\begin{definition}
A family of matrix functions $\cA(z,\ell)$ parametrized by a real parameter $\ell$ is called $j$-monotonic if $\cA(z,\ell_1)^{-1} \cA(z,\ell_2)$ is $j$-inner whenever $\ell_1 < \ell_2$.
\end{definition}
To a spectral theorist, these notions are of interest because they describe common properties of transfer matrices; in particular, $j$-monotonicity describes the nesting property of Weyl disks
\begin{equation}\label{WeylDiskDefn}
D(z,\ell) = \{ w \mid \begin{pmatrix} w & 1 \end{pmatrix} \cA(z,\ell) j \cA(z,\ell)^* \begin{pmatrix} w & 1 \end{pmatrix}^* \ge 0 \},
\end{equation}
namely, $D(z,\ell_2) \subset D(z,\ell_1)$ if $\ell_1 < \ell_2$ and $z \in \bbC_+$. This was first observed in the setting of Schr\"odinger operators by Weyl \cite{Weyl}, with  $j = (\begin{smallmatrix} 0 & -i \\ i  & 0 \end{smallmatrix})$.
Note that by conjugating by a Cayley transform we can make the switch to
\begin{equation} \label{jmatrixchoice}
j = \begin{pmatrix} -1 & 0 \\ 0 & 1 \end{pmatrix}.
\end{equation}
We will always use this choice of $j$; note in particular that for $\cA(z,0) = I$, \eqref{WeylDiskDefn} gives $D(z,0) = \overline{\bbD}$, so our Weyl disks will be subsets of $\overline{\bbD}$. 

We will always work with matrix functions which are continuous in $\ell$ and obey $\det \cA(z,\ell) = 1$ for all $z, \ell$. In particular, the values $\cA(z,\ell)$ for $z \in \bbR$ will belong to the group of $2\times 2$ matrices
\[
\SU(1,1) = \{ U \mid  U j U^* = j  \text{ and }\det U = 1 \}.
\]
It follows directly from the definition \eqref{WeylDiskDefn} that Weyl disks are not affected by right multiplication of the transfer matrix by $U(\ell) \in \SU(1,1)$. Borrowing terminology from Yang--Mills theory, we say that the chain of transfer matrices possesses a \textit{gauge freedom} with the gauge group $\SU(1,1)$, while the corresponding Weyl disks are \textit{observables}, i.e. gauge independent. In particular, any $j$-monotonic family can be uniquely brought into the following gauge:
\begin{definition}
A $j$-monotonic family $\cA(z,\ell)$ is in \emph{Arov gauge (A-gauge)} if $\cA(i, \ell)$ is lower triangular with positive diagonal entries.
 \end{definition}
 The Arov gauge arose naturally in the description of the set of the unitary extensions of isometric operators \cite{AGr83} and was used by the author regularly, see e.g. \cite[Theorem 7.57]{AD08}.

For a $j$-monotonic family of transfer matrices in Arov gauge we use a special notation $\fA(z,\ell)$. If it obeys the initial condition $\fA(z,0) = I$ for all $z$, it is the solution of  a \emph{canonical system in Arov gauge},
\begin{equation}\label{canonicalArov1}
\fA(z,\ell) j = j +  \int_0^\ell \fA(z,l) \left( i z \left(\begin{smallmatrix}
1 & - \overline{\fa(l)} \\
- \fa(l) & 1
\end{smallmatrix}\right)
- \left(\begin{smallmatrix}
0 &  \overline{\fa(l)} \\
- \fa(l) & 0
\end{smallmatrix}\right) \right) \,d\mu(l),
\end{equation}
where $\mu$ is a positive continuous measure and $\fa \in L^\infty(d\mu)$ with $\lVert \fa \rVert_\infty \le 1$. 
The equation \eqref{canonicalArov1} is an initial value problem written in integral form: the solution $\fA(z,\ell)$ is entire in $z$ for each $\ell$ and absolutely continuous with respect to $\mu$ as a function of $\ell$. We find the integral form more natural because $\mu$ is allowed to contain a singular continuous component. 

The pair $(\mu,\fa)$ is the set of parameters of the canonical system in Arov gauge. In particular, $\fa$ can be treated as ``continuous Verblunsky coefficients", and  represented via boundary value of the spectral Schur function at infinity. We will also review key properties of canonical systems in Arov gauge in Remark~\ref{remarkArovGauge1}; in a companion paper \cite{BLY1}, we give a thorough presentation.

Due to the nesting property of Weyl disks, as $\ell \to \infty$ they shrink to a disk or a point; this is the famous limit circle/limit point dichotomy. Moreover, if the intersection is a point for one $z\in \bbC_+$, it is a point for all $z \in \bbC_+$. In the limit point case, the intersection of Weyl disks generates the spectral function $s_+(z)$ of the canonical system by
\[
\{ s_+(z) \} = \bigcap_{\ell \in (0,\infty) } D(z,\ell),
\]
which is a Schur function in the sense that it is an analytic map $s_+ : \bbC_+ \to \overline{\bbD}$.

As an analog of de Branges' uniqueness theorem \cite{dB}, any Schur function $s_+ : \bbC_+ \to \overline{\bbD}$ is the spectral function of a canonical system in Arov gauge, which is unique up to a monotone reparametrization of the parameter $l$. 

In the formulation for Schur functions, the reflectionless property with spectrum $\sE$ is:
\begin{definition}\label{defnSEnew}
The pair of Schur functions $(s_+, s_-)$ is a reflectionless pair with spectrum $\sE$ if $s_\pm$ extend to meromorphic single-valued functions on $\Omega$ with the properties:
\begin{enumerate}[(i)]
\item the symmetry property $\overline{s_\pm(\bar z)} =1/s_\pm(z)$ for $z \in \Omega$; 
\item the reflectionless property
\begin{equation}\label{reflpropertydefn}
 \overline{s_+(\xi+ i0)} = s_-(\xi + i0) \quad \text{a.e. }\xi \in \sE;
\end{equation}
\item $1-s_+(z)s_-(z)$ does not vanish in $\bbR\setminus \sE$.
\end{enumerate}
We denote by $\cS(\sE)$ the set of functions $s_+$ which are part of such a pair.
\end{definition}

The set $\cS(\sE)$ is not compact; when reflectionless theory is applied to some family of operators, there is at least one normalization condition natural to that family, which also compactifies the set. In Arov gauge, the natural normalization condition is:
\begin{definition}
We denote by $\cS_A(\sE)$ the set of $s_+ \in \cS(\sE)$ for which the corresponding $s_-$ obeys $s_-(i) = 0$.
\end{definition}

In \cite{BLY1}, we prove a version of Remling's theorem in A-gauge and use it to conclude:
\begin{theorem}[{\cite{BLY1}}]
Assume that for all $L > 0$, the functions $\mu((\ell, \ell+L])$, $\int_\ell^{\ell + L} \fa(l) d\mu(l)$ are uniformly almost periodic functions of $\ell$. Then the canonical system in A-gauge \eqref{canonicalArov1} is reflectionless on its a.c.\ spectrum $\{ \xi \mid \lvert s_+(\xi + i0) \rvert < 1 \} \cup \{ \xi \mid \lvert s_-(\xi + i0) \rvert < 1 \}$, i.e., it obeys $\overline{s_+(\xi+i0)} = s_-(\xi +i0)$ a.e. on this set.
 \end{theorem}
In the current paper, we are working in the inverse direction. We prove that our construction provides all reflectionless canonical systems in Arov gauge, with a natural parametrization of the line corresponding to $M$-type  (exponential type with respect to the Martin function):

\begin{theorem} \label{theorem110}
Let $\sE \subset \bbR$ be a Dirichlet-regular Widom set with DCT. Parametrized by $(\alpha,\tau) \in \pi_1(\Omega)^* \times \bbT$, the parameters $\mu = \mu^\alpha$ and $\fa = \tau \fa^\alpha$ describe all reflectionless canonical systems in Arov gauge with spectrum $\sE$, with the parametrization of the line such that for all $\ell>0$,
\begin{equation}\label{transfermatrixMartinexptype1}
\lim_{y \to \infty} \frac{ \log \lVert \fA(iy,\ell) \rVert} {M(iy)} = \ell.
\end{equation}
\end{theorem}

In particular, $\tau \in \bbT$ is an integral of motion (it is constant along the translation flow), and the class $\cS_A(\sE)$ is parametrized by the compact torus $\pi_1(\Omega)^* \times \bbT$.

A lot of research on canonical systems has been written in what we call Potapov--de Branges gauge (PdB-gauge) \cite{P60,dB,Remling}, which is normalized at $z=0$ by the condition that $\fB(0,\ell) = I$ for all $\ell$. Since $\cA(0,\ell) \in \SU(1,1)$, any $j$-monotonic family $\cA(z,\ell)$ can be transformed into PdB gauge by defining $\fB(z,\ell)=\cA(z,\ell)\cA(0,\ell)^{-1}$,  and canonical systems in PdB gauge can be written in the form of integral equations as
\begin{equation}\label{22may3}
\fB(z,\ell)j=j+iz\int_0^\ell\fB(z,l) H(l) d\mu(l), \qquad H(l)\ge 0, \ \tr (H j)=0.
\end{equation}
We show that this doesn't always give almost periodic data, and give sufficient conditions for almost periodicity:

\begin{theorem} \label{th2m}
Let $0\in\sE =\bbR\setminus \bigcup_{j\in \bbZ} (a_j, b_j)$ be  a Dirichlet-regular Widom set with DCT, such that
\begin{equation}\label{22may5}
\int_{[-1,1]\setminus \sE}\frac{dx}{|x|}<\infty.
\end{equation}
Without loss of generality, we fix a gap $(a_0,b_0)$ such that $b_0<0$ and denote $\sE_*=\sE\cap[b_0,0]$. Let $\omega(z,\sE_*)$ be the harmonic measure of $\sE_*$  at $z\in\Omega$ and let $(c_*)_j\in(a_j,b_j)$, $j\ge 1$ be its critical points. Assume that
\begin{equation}\label{22may4}
\sum_{j\ge 1}|\omega((c_*)_j,\sE_*)-\omega(a_j, \sE_*)|<\infty.
\end{equation}
Then the matrix measure $H(\ell) d\mu(\ell)$ corresponding to the canonical system \eqref{22may3} with spectral function $s_+ \in \cS(\sE)$ is almost periodic.
\end{theorem}

\begin{remark}\label{rem25oct20}
The simplest way to violate \eqref{22may5} and  \eqref{22may4} is to consider a set generated by geometric progressions: choose $\rho > 1$ and $\rho b_0^- < a_0^- < b_0^- < 0 < a_0^+ < b_0^+ < \rho a_0^+$. Let 
\[
\sE = \bbR \setminus \cup_{j\in \bbZ} ( (\rho^j a_0^-, \rho^j b_0^-) \cup (\rho^j a_0^+, \rho^j b_0^+) ).
\]
At least in the generic case (non algebraic numbers in a certain sense), the measure $H(\ell)d\mu(\ell)$ is not almost periodic.  Of course, in this case, by shifting the spectral parameter $z\mapsto z-x_*$, a PdB-type gauge with respect to some $x_* \in \sE \setminus (\{0\}\cup_j\{ \rho^j a^\pm_0, \rho^j b^\pm_0\})$, would give an almost periodic representation by Theorem \ref{th2m}.  In the generic case the conditions  \eqref{22may5}, \eqref{22may4} are necessary and sufficient for almost periodicity.  
\end{remark}

We also consider Dirac operators. Transfer matrices for Dirac operators obey the Dirac equation
\begin{equation}\label{21may1}
\pd_t\fD(z,t)j=\fD(z,t)(iz I-Q(t)), \quad Q(t)^*=-Q(t),\quad \tr(Qj)=0
\end{equation}
with the initial data $\fD(z,0)=I$. Note that one of \textit{canonical forms} is fixed by an extra condition $\tr Q(t)=0$ \cite{LevSar}. The solution $\fD(z,t)$ is once again a $j$-monotonic family, and the corresponding Schur function obeys
$\lim_{y\to\infty} s_+(iy)  = 0$  \cite{CG},
which we view as a normalization at infinity.

\begin{theorem} \label{th3m}
Let $\sE \subset \bbR$ be a Dirichlet-regular Widom set with DCT such that
\begin{equation}\label{22may7}
\int_{\bbR\setminus \sE}dx =\sum_{j}(b_j-a_j)<\infty.
\end{equation}
Then for any $s_+ \in \cS(\sE)$, the limit $ \lim_{y\to\infty} s_+(iy)$ exists in $\bbD$; moreover, if $\lim_{y\to\infty} s_+(iy)  = 0$,  then $s_+$ is the spectral function of a  classical Dirac differential equation \eqref{21may1} with a uniformly almost periodic potential $Q(\ell)$.
\end{theorem}

\begin{remark}
Like in the case of PdB gauge, the condition \eqref{22may7} is exact for generic sets. 
\end{remark}

Our work relies extensively on the function theory on Riemann surfaces.  H. Widom,  starting from \cite{W}, found a natural bound \cite{WActa, WAnn} for domains which allow complete family of multivalued Hardy spaces. Such domains are now called Widom domains. Hayashi and Hasumi  \cite{H83} found the DCT condition which makes true the counterpart of the Beurling theorem on invariant subspaces  in the Hardy spaces on Widom domains  and, equivalently, continuity of reproducing kernels with respect to the characters. In Section~\ref{sectionHardySpaces}, we will give a systematic presentation of this theory, from the perspective needed in this paper.

In Section~\ref{sectionAbelMap}, we consider bijections between reflectionless pairs of Schur functions, their corresponding  divisors, and elements of the enlarged character group $\pi_1(\Omega)^* \times \bbT$ related to them by a generalized Abel map.

In Section~\ref{sectionCanonical}, we use the notion of unitary node to construct the $j$-contractive families $\fA(z,\ell)$ starting from the Hardy space with a given character.

In Section~\ref{sectionFourier}, we construct the Fourier integrals and prove their unitarity.

In Section~\ref{sectionAP}, we consider almost periodicity of the constructed parameters in different gauges.

\bigskip
\noindent
\textbf{Acknowledgment.} We would like to thank Misha Sodin and Benjamin Eichinger for useful discussions. The work of R.B. in Sections 2.3, 3.1, 4.1, 5.3, 6.3 is supported by grant RScF 19-11-00058 of the Russian Science Foundation. In the rest of the paper,
M.L. was supported in part by NSF grant DMS--1700179 and 
P.Y. was supported by the Austrian Science Fund FWF, project no: P32885-N.

\section{Preliminary: Hardy spaces and reproducing kernels} \label{sectionHardySpaces}

\subsection{Elements of potential theory in Widom Denjoy domains}

Let $\sE \subsetneq \bbR$ be a closed, unbounded set. Let $\Omega = \bbC \setminus \sE$ denote the corresponding Denjoy domain, and note that $\infty \in \partial\Omega$.  Let us denote by $(a_j,b_j)$ the maximal intervals in $\bbR \setminus \sE$. 
If $\sE$ has a finite number of gaps, the subject we are going to discuss is related to the famous finite gap almost periodic differential operators of the second order.
So, our main interest is in the case when $\sE$ has infinitely many gaps
which we index by $j\in\bbZ$ and write the set $\sE$ in the form
\[
\sE =  \bbR \setminus \bigcup_{j\in \bbZ} (a_j, b_j).
\]
We will fix a gap $(a_0,b_0)$ and a point $\xi_* \in (a_0, b_0)$. This will be used to fix some normalizations.
  
We assume that $\sE$ has positive capacity and that $\Omega$ is a Dirichlet regular domain, i.e., every boundary point is a regular endpoint in the sense of potential theory. 
For any $z_0 \in \Omega$, we denote by $G(z,z_0)=G_\Omega(z,z_0)$ the Green function in the domain $\Omega$ with the logarithmic pole at $z_0$, and Dirichlet regularity means that $G(z,z_0)$ is continuous in $\Omega$ and vanishes on the boundary (including infinity). 
The \textit{complex Green function} $\Phi_{z_0}$ is defined by 
\[
\lvert \Phi_{z_0}(z) \rvert = e^{- G_\Omega(z,z_0)}, \qquad \Phi_{z_0}(\xi_*) > 0.
\]
The function $\Phi_{z_0}$ is character automorphic; it can also be characterized by the fact that its lift is a Blaschke product on $\bbD$ with zeros at the points $\zeta \in \Lambda^{-1}(\{z_0\})$.

We will deviate from the above phase normalization in one important special case: we will consider the complex Green function $\Phi = \Phi_i$ with the normalization 
\begin{equation}\label{BlaschkeNormalization}
\Phi(-i) > 0
\end{equation}
and the function $\Phi_\sharp$ defined by \eqref{fsharpdefinition}, which is a complex Green function with zero at $-i$. The functions $\Phi$ and $\Phi_\sharp$ have the same character $\beta_\Phi$. We will also consider the ratio
\begin{equation}\label{30mara}
v(z) = \frac{\Phi(z)}{\Phi_\sharp(z)} = e^{ic_*} \frac{z- i}{z + i}.
\end{equation}

Let $\tilde c_j$ be the collection of critical points of the Green function, $\nabla G(\tilde c_j, \xi_*)=0$. We assume throughout this text that $\Omega$ is of  Widom type, that is,
\begin{equation}\label{15jul1}
\sum_{j\in \bbZ} G(\tilde c_j,\xi_*) < \infty.
\end{equation}

The M\"obius transformations corresponding to $U \in \SU(1,1)$ are precisely the automorphisms of the unit disk $\bbD$. We denote by $\Lambda : \bbD / \Gamma \to \Omega$ a uniformization of $\Omega$, where $\Gamma$ is a discrete subgroup of $\SU(1,1)$ acting on $\bbD$ in the sense of M\"obius transformations. In particular $\Lambda$ is surjective and $\Lambda(\zeta_1) = \Lambda(\zeta_2)$ if and only if $\zeta_2 = \gamma(\zeta_1)$ for some $\gamma \in \Gamma$.
 Note that the homotopy group of $\Omega$ is $\pi_1(\Omega) \cong \Gamma$. 
 
Statements about multi-valued functions $f$ on $\Omega$ such that $\lvert f\rvert$ is single-valued can be written as statements about their single-valued lifts $F = f \circ \Lambda$.  In particular, \eqref{7jul111} can be restated as
\begin{equation}\label{7jul1}
F \circ \gamma = e^{2\pi i \alpha(\gamma)} F,  \qquad \forall \gamma \in \Gamma,
\end{equation}
where $F \circ \gamma$ simply denotes composition of functions.

By symmetry, we can fix the uniformization $\Lambda$ so that the diameter $(-1,1) \subset \bbD$ is mapped to the gap $(a_0,b_0)$ and $0$ is mapped to $\xi_*$. This normalization obeys the symmetry $\Lambda(\zeta) = \overline{\Lambda(\bar\zeta)}$. This choice affects the involution $(\dots)\sharp$ defined in \eqref{fsharpdefinition}: it is more precise to define this involution by saying that the lifts of $f$, $f_\sharp$ are related by
\[
F_\sharp(\zeta) = \overline{ F(\bar \zeta) }, \qquad \forall \zeta \in \bbD.
\]

We will often use an important criterion of Sodin--Yuditskii \cite{SY97}:
  \begin{theorem}[{\cite[Theorem D]{SY97}}]\label{thsy97}
 Let $\sE$ be a Dirichlet regular Widom set, and let $f$ be a meromorphic Herglotz function on $\Omega$ with $\overline{f(\bar z) }= f(z)$. If poles of $f$ satisfy the condition
\begin{equation} \label{poleconditionTheoremD}
\sum_{\substack{\lambda: f(\lambda) = \infty \\ \lambda \neq \xi_*} } G(\lambda,\xi_*)<\infty,
\end{equation}
then $f$ is of bounded characteristic with no singular inner factor (i.e. its inner factor is a quotient of Blaschke products).
 \end{theorem} 
Note that \eqref{poleconditionTheoremD} holds automatically if $f$ has at most one pole in each gap of $\sE$.

The Martin function $M$ has one critical point in each gap, which we denote by $c_j \in (a_j, b_j)$. For Widom domains, the Widom function
\[
\cW(z) = \prod_j \Phi_{c_j}(z)
\]
is well defined and nontrivial. Denote its character by $\beta_\cW$. Note that with our normalization, $\cW_\sharp = \cW$, and $\Theta'$ is a function of bounded characteristic with inner part  $\cW$, see Theorem \ref{thsy97}. 

The Dirichlet problem on $\Omega$ can be solved by using the uniformization and pushing the Lebesgue measure on $\bbT$ to the harmonic measure $\omega(z,F)$ on the Martin boundary $\partial\Omega$, which consists (up to a zero measure set) of two copies of $\sE$. For a Widom set $\sE$, harmonic measure is mutually a.c. with Lebesgue measure on each copy of $\sE$. For instance, this solves the following problem: for a set $F\subset \sE$, $\omega(z,F_+)$ is the harmonic function in $\Omega$ whose boundary values are given a.e.\ by
$$
\omega(\xi+i0,F_+)= \chi_{F_+}(\xi), \qquad \omega(\xi - i0,F_+)=0.
$$
Similarly, Martin measure is naturally defined on this double cover of $\sE$: the boundary values $\Theta(\xi\pm i0)$ obey $\frac 1{2\pi} d\Theta(\xi + i0) = - \frac 1{2\pi} d\Theta(\xi-i0) = d\vartheta$ \cite{EY12}. By combining these measures we obtain the Martin measure on $\partial\Omega$, denoted $\frac 1{2\pi}d\Theta$.

With respect to these measures, we have the standard Lebesgue spaces $L^p_{\partial\Omega}(d\omega)$ and 
\[
L^p_{\partial\Omega}= L^p_{\partial\Omega}\left( \frac 1{2\pi} d\Theta\right) \equiv L^p(\sE, d\vartheta)^2;
\]
compare \eqref{8oct5}.  Functions $f \in \cN(\Omega)$ have nontangential boundary values from above and below, denoted $f(\xi\pm i0)$ for $\xi \in \sE$, and we will use $L^p$ conditions on the boundary to $f$.

\subsection{Hardy spaces with respect to harmonic measure and Martin measure and reproducing kernels}
Character-automorphic Hardy spaces $H^p(\a) = H^p_\Omega(\a)$ can be defined in several equivalent ways \cite{H83}; one of the definitions uses the universal covering and the standard Hardy spaces $H^p(\bbD)$:
\begin{definition}\label{h2def3} $H^p(\a)$ is the set of character-automorphic functions $f$ with character $\alpha$ whose lift $F = f \circ \Lambda$ is an element of $H^p(\bbD)$, with the inherited norm. 
 \end{definition}

By passing to a universal covering and using the Smirnov maximum principle \cite{NIK}, an equivalent (alternative) definition is: 

\begin{definition}\label{h2def2}
$H^p(\alpha)$ is the space of character automorphic functions on $\Omega$ with character $\alpha$ which are in Smirnov class $\cN_+(\Omega)$ and whose boundary values are in $L^p_{\partial\Omega}(d\omega)$.
 \end{definition}

Definition~\ref{h2def2} makes clear that these character-automorphic Hardy spaces are with respect to harmonic measure for the internal point $\xi_*$ of the domain. In our setting, it is more natural to work with respect to Martin measure, since that measure  plays a crucial role in the spectral theory of ergodic and almost periodic differential equations, as the so-called integrated density of states. In particular, Definition~\ref{h2def2} motivates the definition of Hardy spaces with respect to Martin measure, made in the introduction; we will now show their relations to the spaces $H^2(\alpha)$ defined with respect to harmonic measure.

\begin{lemma}\label{lemma15jul201}
There is a character automorphic outer function $\psi$ with character $\beta_\psi$ such that $L^2_{\partial\Omega}(d\Theta)=\psi L^2_{\partial\Omega}(d\omega)$ and $\cH^2(\alpha+\beta_\psi) = \psi H^2(\alpha)$ for any character $\alpha$.
\end{lemma}

\begin{proof}
To simplify notation, without loss of generality, in this  proof we assume that $\xi_*=0$. In a Widom domain, $\omega(dx)$ is absolutely continuous w.r.t.\ the Lebesgue measure $dx$ \cite{EY12},
\[
2\pi i d\omega(x)=f_\omega(x) dx, \qquad f_\omega(z)= \frac{1}{z\sqrt{(1-z/a_0)(1-z/b_0)}}\prod_{j\not= 0}\frac{1-z/\tilde c_j}{\sqrt{(1-z/a_j)(1- z/b_j)}}.
\]
Due to Theorem \ref{thsy97} the product $f_\omega$ is of bounded characteristic with no singular inner part. Its inner part is the convergent Blaschke product determined by reading off its zeros, so we obtain its outer part as 
\[
\Psi_{\omega}(z)^2=\frac{\Phi_0(z)}{z\sqrt{(1-z/a_0)(1-z/b_0)}}\prod_{j\not= 0}\frac{1-z/\tilde c_j}{\Phi_{\tilde c_j}(z)\sqrt{(1-z/a_j)(1- z/b_j)}}
\]
and on $\sE$ we have $\omega(dx)=\frac 1 {2\pi }|\Psi_{\omega}(x)|^2 dx$. Likewise, assuming that $ c_0 \neq 0$, we have
\[
2\pi i d\Theta(x)= f_\Theta(x) dx =  i\lvert \Psi_{\Theta}(x) \rvert^2 dx, \qquad 
\Psi_{\Theta}(z)^2 = 
 C\prod_{j\not= 0}\frac{1-z/c_j}{\Phi_{c_j}(z)\sqrt{(1-z/a_j)(1- z/b_j)}}.
\]
Thus, $d\omega = \lvert \psi \rvert^2 d\Theta$ with the outer function $\psi = \Psi_\omega / \Psi_\Theta$. Since $\psi \cN_+(\Omega) = \cN_+(\Omega)$, the claim follows.
\end{proof}

In a Widom domain, $H^2(\a)$ is nontrivial for any $\alpha$, and it has a reproducing kernel inherited from the universal covering and $H^2(\bbD)$. By Lemma~\ref{lemma15jul201}, $\cH^2(\a)$ inherits these properties:

\begin{proposition}
For a Widom Denjoy domain  $\Omega$ the Hardy space $\cH^2(\alpha)$ is nontrivial for every $\alpha \in \pi_1(\Omega)^*$. This is a reproducing kernel Hilbert space, i.e., for each $z_0\in \Omega$ there exists $k_{z_0}^\alpha \in \cH^2(\alpha)$ such that for all $f\in \cH^2(\alpha)$,
\[
f(z_0) = \langle f, k_{z_0}^\alpha \rangle.
\]
In particular, $\langle k_{z_0}^\alpha, k_{z_0}^\alpha \rangle = k_{z_0}^\alpha(z_0) > 0$. We also write $k^\alpha(z,z_0) = k^\alpha_{z_0}(z)$.
\end{proposition}

\begin{remark}
It seems natural to give an alternative definition for $\cH^2(\a)$ in the spirit of Definition \ref{h2def3}, by considering subspaces of $H^2(\bbC_+)$ which are character automorphic w.r.t.\ a discrete subgroup of the group of $\SL(2,\bbR)$. However, this is possible \textit{only} in A-L domains \cite{KY18}.
\end{remark}

\subsection{Extensions of symmetric operators and their Cayley transforms. Resolvent representation for the reproducing kernel}

Consider the multiplication operator by the independent variable $z$ in $\cH^2(\a)$, as an unbounded operator with the domain
\[
\fD_z=\left\{\frac{\Phi}{z-i} f: \ f\in \cH^2(\a-\b_\Phi)\right\}.
\]
Since the Direct Cauchy Theorem holds in $\Omega$, see subsection \ref{sDCT}, criterion  (b), $\fD_z$ is dense in $\cH^2(\a)$, since $\frac{\Phi}{z-i}$ is an outer function. 

By \eqref{30mara}, we can consider multiplication by $\overline{v(z)}$ its Cayley transform. Let
\begin{align*}
\fD_{\bar v}=\text{clos} \{(z-i) f:\ f\in \fD_{z}\}=\Phi \cH^2(\a-\b_\Phi)\\
\Delta_{\bar v}=\text{clos} \{(z+i) f:\ f\in \fD_{z}\}=\Phi_\sharp \cH^2(\a-\b_\Phi)
\end{align*}
Since $\overline{v(z)}$ is unimodular for $ z\in\partial \Omega$, multiplication by 
$\overline{v(z)}$ acts isometrically from $\fD_{\bar v}$ to $\Delta_{\bar v}$. The defect spaces are one dimensional,
\begin{equation}\label{20mar2}
\cH^2(\a)=\{K^\a \}\oplus \fD_{\bar v}=\{K_\sharp^\a \}\oplus\Delta_{\bar v}.
\end{equation}
Thus, this isometry has a one-parameter family of unitary extensions $\hat U_\tau : \cH^2(\alpha) \to \cH^2(\alpha)$, which are of the form
\begin{equation}\label{15jul5}
\hat U_\tau =\tau K^\a_\sharp \langle\,\cdot\, , K^\a \rangle+\bar v\cdot P_{\fD_{\bar v}},\quad \tau\in \bbT.
\end{equation}

\begin{proposition}\label{prop15jul1} The resolvent (Titchmarsh--Weyl) function  of the unitary extension $\hat U_\tau:\cH^2(\a)\to \cH^2(\a)$ is of the form 
\begin{equation}\label{13jul1}
m(z):=i\left\langle \frac{I+v(z) \hat U_\tau}{I-v(z) \hat U_\tau}K^\a, K^\a \right\rangle=i\frac{1+ \tau v(z) s^\a_+(z)}{1-\tau v(z) s^\a_+(z)},
\quad  z\in \bbC_+,
\end{equation}
where
\begin{equation}\label{18sep201}
s_+^\a(z) = \frac{K^\a_\sharp(z)}{K^\a(z)}.
\end{equation}
\end{proposition}

\begin{proof}
For any $z_0 \in \bbC_+$, the definition of $m(z_0)$ implies
\[
\frac{-im(z_0) + 1}2 = \langle (I-v(z_0) \hat U_\tau)^{-1} K^\alpha, K^\alpha\rangle,
\]
so by \eqref{20mar2} there exists $f\in \Phi\cH^2(\a-\b_\Phi)$ such that
$$
(I-v(z_0) \hat U_\tau)^{-1} K^\a = f+\frac{-i m(z_0)+1} 2 K^\a.
$$
Applying $I - v(z_0) \hat U_\tau$ and using \eqref{15jul5} gives
$$
K^\a = f-\bar v v(z_0) f+\frac{-i m(z_0)+1} 2 (K^\a - v(z_0) \tau K_\sharp^\a ).
$$
Multiplying by $v$ gives
\[
v K^\a = (v -  v(z_0)) f+ v \frac{-i m(z_0)+1} 2 (K^\a -v(z_0) \tau K_\sharp^\a).
\]
Both sides of the equality are functions in $\frac 1{\Phi_\sharp} \cH^2(\alpha+\beta_\Phi)$, so we can evaluate them at $z=z_0$ and obtain
\[
2K^\a(z_0)= (-i m(z_0)+1) (K^\a(z_0)-v(z_0) \tau K_\sharp^\a(z_0)).
\]
Now solving for $m(z_0)$ gives \eqref{13jul1}, since $z_0 \in \bbC_+$ is arbitrary.
\end{proof}

In the context of the multiplication operator by $z$ in $\cH^2(\a)$, we obtain a kind of resolvent  representation for the reproducing kernel.

\begin{lemma} \label{lemmareproducingkernelformula}

Reproducing kernels obey the identity
\begin{align}\nonumber
k^\alpha(z,z_0)
&=i\frac{(z+i)\overline{(z_0+i)}K^\alpha(z)\overline{K^\alpha(z_0)}-(z-i)\overline{(z_0-i)} K_\sharp^\alpha(z)
\overline{K_\sharp^\alpha(z_0)}}{2(z-\bar z_0)}\\ \label{30jun2}
&=\frac{K^\alpha(z)\overline{K^\alpha(z_0)}-v(z)\overline{v(z_0)} K_\sharp^\alpha(z)
\overline{K_\sharp^\alpha(z_0)}}{1-v(z)\overline{v(z_0)}}
\\ \nonumber
&=K^\alpha(z)\overline{K^\alpha(z_0)}\frac{1-v(z)s_+^\a(z)\overline{v(z_0)s_+^\a(z_0)}}{1-v(z)\overline{v(z_0)}}.
\end{align}
\end{lemma}

\begin{proof}
For all $g\in \fD_z$,  $\langle k^\alpha_{z_0}, (z-z_0)g \rangle=0$, so writing $g = \sqrt{\frac{\Phi}{z-i}\frac{\Phi_\sharp}{z+i}} f$, we have
\[
\left\langle\frac{z-\bar z_0}{\sqrt{\Phi\Phi_\sharp (z^2+1)}} k^\alpha_{z_0}, f\right\rangle=0 \qquad \forall f \in \cH^2(\a-\b_\Phi).
\]
On the other hand, for suitable $C_1, C_2$, 
\[
\frac{z-\bar z_0}{\sqrt{\Phi\Phi_\sharp (z^2+1)}} k^\alpha_{z_0}-C_1\frac{K^\a}{\Phi}-C_2\frac{K^\a_\sharp}{\Phi_\sharp} \in \cH^2(\a - \b_\Phi)
\]
so
\[
(z-\bar z_0)k^\alpha_{z_0}=C_1\sqrt{(z^2+1)\frac{\Phi_\sharp}{\Phi}}{K^\a}
+C_2\sqrt{(z^2+1)\frac{\Phi}{\Phi_\sharp}}{K^\a_\sharp}.
\]
Due to \eqref{30mara} we get
\[
(z-\bar z_0)k^\alpha_{z_0}=C_1e^{-ic_*/2}(z+i){K^\a}
+C_2e^{ic_*/2}(z-i){K^\a_\sharp}.
\]
By setting $z=\pm i$, we compute these constants
\[
C_1 e^{-ic_*/2}K^\a(i)=\frac{i-\bar z_0}{2i}k^\a(i,z_0),\quad C_2 e^{ic_*/2}K^\a_\sharp(-i)=\frac{-i-\bar z_0}{-2i}k^\a(-i,z_0).
\]
Since $K^\a(\pm i)=\sqrt{k^\a(\pm i,\pm i)}$, with a trivial algebraic manipulation
\begin{equation}\label{2jul1}
1 - v(z) \overline{v(z_0)}  = 1 - \frac{ z-i}{z+i} \frac{\bar z_0 + i}{\bar z_0 - i} = - \frac{ 2i(z - \bar z_0)}{(z+i) \overline{( z_0 + i)} },
\end{equation}
we get \eqref{30jun2}.
\end{proof}

\begin{remark}
As is well known, the Titchmarsh--Weyl function $m(z)$ has positive imaginary part on $\bbC_+$, so by \eqref{13jul1}, $\tau s_+^\a(z)$ is a Schur function (analytic map from $\bbC_+$ to $\overline{\bbD}$). This is also evident from \eqref{30jun2} and positivity of reproducing kernels.
\end{remark}

\subsection{Reflectionless property, pseudocontinuation, and DCT. Reproducing kernels and Wronskian identity}\label{sDCT}

From the point of view of function theory, the reflectionless property is closely related to the notion of pseudocontinuation. For a function $F$ of bounded characteristic in $\bbD$, we say that a function of bounded characteristic $G$ in $\hat\bbC \setminus \bbD$ is the pseudocontinuation of $F$ if
\[
\lim_{r\uparrow 1} F(r \zeta) = \lim_{r\downarrow 1} {G(r \zeta)}, \qquad \text{a.e. }\zeta\in \bbT.
\]
By the substitution $G(\zeta) = \overline{ F_*(1/\bar \zeta) }$, existence of a pseudocontinuation can be expressed entirely in terms of functions on $\bbD$: $F \in \cN(\bbD)$ has a pseudocontinuation if and only if there exists $F_* \in \cN(\bbD)$ such that
\[
\lim_{r\uparrow 1} F(r \zeta ) = \lim_{r\uparrow 1} \overline{ F_*(r \zeta)}, \qquad  \text{a.e. }\zeta\in \bbT.
\]
Applying these notions to lifts of character-automorphic functions on $\Omega$ leads to a notion of pseudocontinuation on $\Omega$ and an important involution:

\begin{definition} We say that $f\in \cN(\Omega)$ has a pseudocontinuation if 
there exists $f_* \in \cN(\Omega)$ such that
\begin{equation*}\label{fstarinvolution}
f_*(z) = \overline{f(z)}\quad \text{for a.e.} \ z\in\partial \Omega.
\end{equation*}
We point out that if $\alpha_f$ is the character of $f$ then the character of $f_*$ is $\alpha_{f_*}=-\alpha_f$. 
\end{definition}

For Denjoy domains, combining this involution with the involution $(\dots)_\sharp$,  we obtain the linear involution $f \mapsto f_\flat$ from the introduction,
\begin{equation}\label{16jul2}
f_{\flat}(z)= (f_*)_\sharp(z) = \overline{f_*(\overline{z})},\quad z\in \Omega.
\end{equation}
This is well defined for an arbitrary $f$ which has a pseudocontinuation, and $\alpha_{f_\flat}=-\alpha_f$. Note that on the boundary of the domain we have \eqref{fflatdefinition}.

\begin{example}
If $\Delta \in \cN(\Omega)$ is an inner function, $\Delta_* = 1/\Delta$ so
\[
\Delta_\flat(z) = \frac 1{\overline{\Delta(\bar z)}}.
\]
In particular, 
$(\Phi)_\flat = \frac 1{\Phi_\sharp}$, $(\Phi_\sharp)_\flat = \frac 1{\Phi}$, and $v_\flat = v$.
\end{example}

Let $f\in \cH_\Omega^1(\beta_\cW)$. Then $\frac{f(z)}{\cW(z)} d\Theta(z)$ is a single-valued differential in $\Omega$, moreover
$f \frac{\Theta'}\cW \in \cN_+(\Omega)$. 

\begin{definition}
A Widom domain $\Omega$ obeys DCT if  for all $f\in \cH_\Omega^1(\beta_\cW)$,
\[
\oint_{\partial\Omega} \frac{f(z)}{\cW(z)} \,d\Theta(z) = \oint_{\partial\Omega} \frac{f(z)\Theta'(z)}{\cW(z)} \,dz = 0.
\]
\end{definition}

DCT is an abbreviation for ``Direct Cauchy Theorem". However, it is actually a property discovered by Hayashi and Hasumi \cite{H83}, which holds for some Widom sets and fails for others. In this paper, we will work with sets for which   DCT holds.

A statement equivalent to DCT is given in the following theorem.
\begin{theorem}\label{thmDCTifandonlyif}
For a regular Widom domain  $\Omega$, the DCT property holds if and only if
\begin{equation}\label{18jul1}
 L^2_{\partial\Omega}=\cH^2(\alpha) \oplus \cW \overline{ \cH^2(\beta_\cW - \alpha) }
\end{equation}
for every $\alpha\in \pi_1(\Omega)^{*}$, where $\overline{ \cH^2(\beta_\cW - \alpha) }$ denotes the set of functions conjugated to $ \cH^2(\beta_\cW - \alpha) $.
\end{theorem}

\begin{proof}
We point out that  for a.e. $z\in \partial\Omega$
\[
\frac{\Phi_0(z)}{\cW_\omega(z)}\overline{\Psi_{\omega}(z)}=\Psi_{\omega}(z),\quad \cW_\omega=\prod_{j\not=0}\Phi_{\tilde c_j}(z)
\]
and
\[
\frac{1}{\cW_\Theta(z)}\overline{\Psi_{\Theta}(z)}=\Psi_{\Theta}(z),\quad \cW_\Theta=\prod_{j\in\bbZ}\Phi_{c_j}(z).
\]
By \cite{SY97} for $g\in L^2_{\partial\Omega}(d\omega)\ominus H^2(\alpha)$ we have
\begin{equation}\label{8jul1}
\frac{\cW_{\omega}}{\Phi_0}\overline{g}\in H_\Omega^2(-\alpha+\beta_{\cW_\omega}-\beta_{\Phi_0}).
\end{equation}
Combining all these,
we obtain
\[
{\cW_{\Theta}}\overline{f}\in \cH^2(-\alpha-\beta_{\psi}+\beta_{\Phi_{\Theta}})
\]
for $f=\psi g\in L^2_{\partial\Omega}(d\vt)\ominus \cH^2(\alpha+\beta_{\psi})$.
Indeed, see \eqref{8jul1},
\[
\cW_\Theta\overline{\psi g}=\psi \cW_\Theta\frac{\bar\psi}{\psi}\overline{ g}=\psi \cW_\Theta\frac{\cW_{\omega}}{\Phi_0}
 \frac{1}{\cW_\Theta}\overline{ g}=\psi 
\left(\frac{\cW_{\omega}}{\Phi_0}\overline{g}\right)\in \cH^2(-\alpha-\beta_{\psi}+\beta_{\Phi_{\Theta}}).
\]
It remains to show that the decomposition \eqref{18jul1} implies DCT. Let $f\in \cH^1(\b_{\cW})$. Consider its inner-outer factorization  $f=f_{i} f_o$ and define  $g_1=\sqrt{f_o}$, $g_2= f_i\sqrt{f_o}$. If $\a$ is the character of $g_1$, then $g_1\in \cH^2(\a)$, respectively, $g_2\in\cH^2(\b_{\cW}-\a)$.
Therefore, $\cW\overline{g_2}\in \cW\overline{\cH^2(\b_{\cW}-\a)}=L^2_{\partial\Omega}(d\vartheta) \ominus \cH^2(\a)$. We get
\[
\frac 1{2\pi}\oint_{\partial\Omega} f\frac{ d\Theta}{\cW}=\frac 1 {2\pi} \oint_{\partial\Omega} g_1 g_2 \frac{ d\Theta}{\cW}=\langle g_1, \cW \overline{g_2} \rangle=0. \qedhere
\]
\end{proof}

There are two more important characteristic properties for DCT.

\begin{itemize}
\item[(a)] DCT holds if and only if $k^\a(z,z)$ is  continuous on $\pi_1(\Omega)^*$.
\item[(b)] Let  $\cM\subset \cH^2(\a)$ and $w \cM\subset \cM$ for an arbitrary $w\in H_\Omega^\infty$. DCT holds if and only if for an arbitrary such $\cM$  there exists an inner function $\Delta$ such that
$$
\cM=\Delta \cH^2(\a-\b_\Delta),\quad \Delta\circ\g= e^{2\pi i\b_\Delta(\g)}\Delta.
$$
\end{itemize}

The following property is closely related to \eqref{18jul1}, and, in fact, is also characteristic for DCT.
\begin{corollary}\label{l218jul}
Denote
\[
\tau_* = e^{i\varphi_*}, \qquad \varphi_* = - \arg \left( \frac{\Theta'(i)}{\cW(i)} \frac{i}{\Phi'(i)}\right),
\]
noting that this phase is independent of $\alpha$. Denote
\begin{equation}\label{29sept203}
\tilde \alpha = \beta_\cW + \beta_\Phi - \alpha.
\end{equation}
Then
\begin{equation}\label{11jul89}
(K^{\tilde\alpha})_* = \tau_* \frac{K^{\alpha}}{ \Phi \cW}, 
\quad
(K_\sharp^{\tilde\alpha})_* = \bar \tau_* \frac{K_\sharp^{\alpha}}{ \Phi_\sharp \cW}. 
\end{equation}
Moreover, for any $\alpha$,
\begin{equation}\label{19mar2}
K^{\alpha}(i) K^{\tilde \alpha}(i) = \left\lvert \frac{\cW(i)\Phi' (i)}{\Theta'(i)} \right\rvert.
\end{equation}
\end{corollary}

\begin{proof} 
Since for any $f \in \cH^2_\Omega(\alpha - \beta_\Phi)$,
\[
\left\langle f, \frac{K^{\alpha}}{\Phi }  \right\rangle = \langle \Phi f, K^{\alpha}  \rangle = 0,
\]
it follows that $\frac{K^{\alpha}}{\Phi} \perp \cH^2_\Omega(\alpha-\b_\Phi)$, so by Theorem~\ref{thmDCTifandonlyif}, for some $g \in \cH_\Omega^2(\tilde\alpha)$,  $\frac{K^{\alpha}}{\Phi }  = \cW \bar g$ in the sense of equality in $L^2_{\partial\Omega}$. For any $f \in \cH_\Omega^2(\tilde\alpha)$,
\[
\langle f, g \rangle =\frac 1{2\pi} \int_{\partial\Omega} f \bar g \, d\Theta = \frac 1{2\pi}  \int_{\partial\Omega} \frac f \cW  \frac{ K^{\alpha }}{\Phi} \, d\Theta.
\]
Since $\Phi$ has a simple zero at $i$ and no other zeros, by the direct Cauchy theorem,
\[
\langle f, g \rangle = i \frac{f(i)}{\cW(i)} \frac{K^{\alpha}(i)}{\Phi'(i)} \Theta'(i).
\]
This implies that $g = \bar C K^{\alpha}(i) K^{\tilde\alpha}(i) K^{\tilde\alpha}$ where $C = \frac{\Theta'(i)}{\cW(i)} \frac{i}{\Phi'(i)}$. Note that $C$ is independent of character. Thus, in the sense of equality in $L^2_{\partial\Omega}$,
\begin{equation}\label{19mar1}
\frac{K^{\alpha} }{\Phi} = C  K^{\alpha}(i)  K^{\tilde \alpha}(i)  \cW \overline{ K^{\tilde \alpha} }.
\end{equation}
By the normalization $\|K^{\a} \|=\| K^{\tilde\a} \|=1$, comparing $L^2_{\partial\Omega}$-norms of both sides of \eqref{19mar1} implies $\lvert C \rvert K^{\alpha}(i)  K^{\tilde\alpha}(i) = 1$. This implies \eqref{19mar2}, and since $\arg C = - \varphi_*$, allows to rewrite \eqref{19mar1} as  the first relation in \eqref{11jul89}. The second relation follows by the involution $(\dots)_\sharp$.
\end{proof}

\begin{corollary}
By Corollary \ref{l218jul},
\begin{equation}\label{11jul77}
( K^{\alpha})_\flat = \bar \tau_* \frac{K_\sharp^{\tilde\alpha}}{\cW \Phi_\sharp},\quad
( K_\sharp^{\alpha})_\flat = \tau_* \frac{K^{\tilde\alpha}}{\cW \Phi}.
\end{equation}
\end{corollary}

We now define the matrix function
\begin{equation}\label{7jul3a}
\cT_\a=\begin{pmatrix}
\tau_* \Phi_\sharp K^{\tilde\alpha} & \bar \tau_* \Phi K_\sharp^{\tilde\alpha} \\
 K_\sharp^{\alpha} & K^{\alpha} 
\end{pmatrix}
\end{equation}
which will play an essential role in what follows. First, using the involutions and the resolvent representation for the reproducing kernels, we derive the following ``Wronskian identity".

\begin{lemma}\label{lemmaTadet}
$\det \cT_\a$ is an outer function independent of $\alpha$ and given by
\begin{equation}\label{7jul3}
\det \cT_\a
= - i \frac{v'}v \frac{\cW}{\Theta'} \Phi \Phi_\sharp = 2 \frac{\cW}{\Theta'} \frac{\Phi}{z-i} \frac{\Phi_\sharp}{z+i}.
\end{equation}
\end{lemma}

\begin{proof}
Using Lemma~\ref{lemmareproducingkernelformula}, for any $f \in \cH^2(\alpha)$,
\begin{align*}
f(z_0) 
& =  \frac 1 {2\pi}\int_{\partial\Omega}  f \overline{k_{z_0}^\alpha} \, d\Theta \\
& = \frac 1{2\pi} \oint_{\partial\Omega} f \frac{v \tau_*  \frac{K^{\tilde\alpha}}{\cW \Phi} K^\alpha(z_0) - v(z_0) \bar \tau_* \frac{ K_\sharp^{\tilde\alpha}}{\cW \Phi_\sharp} K_\sharp^\alpha(z_0) }{v - v(z_0)}  \Theta' \,dz.
\end{align*}
Since this has a simple pole at $z_0$ and no other singularities, computing the residue at $z_0$ and using DCT gives
\[
f(z_0) =  i f(z_0) \frac{v(z_0) \tau_* \frac{K^{\tilde\alpha}(z_0)}{\cW(z_0) \Phi(z_0)} K^\alpha(z_0) - v(z_0) \bar \tau_* \frac{ K_\sharp^{\tilde\alpha}(z_0)}{\cW(z_0) \Phi_\sharp(z_0)} K_\sharp^\alpha(z_0) }{v'(z_0)}  \Theta'(z_0).
\]
In this formula, $z_0$ was arbitrary, so we can regard this as equality of functions,
\[
\tau_* \frac{K^{\tilde\alpha} }{ \Phi } K^\alpha  - \bar \tau_* \frac{ K_\sharp^{\tilde\alpha} }{ \Phi_\sharp } K_\sharp^\alpha  = \frac {1}{i} \frac{v'}{v} \frac{\cW}{\Theta'},
\]
and \eqref{7jul3} follows by elementary manipulations. By the second equality in \eqref{7jul3}, $\det \cT_\a$ is an outer function.
\end{proof}

We add a few related matrix identities. For $a\in\bbD$ define
$$
\cV(a)=\frac 1 \rho\begin{pmatrix} 1 & - \bar a \\ - a & 1
 \end{pmatrix}\in \SU(1,1),\quad
\rho:=\sqrt{1-|a|^2}.
$$
The following lemma  is essentially one step in the classical Schur algorithm \cite{PY06}. 

\begin{lemma}\label{l18jul1}
For any $\a$,
\begin{equation}\label{26may1}
\begin{pmatrix} K_\sharp^{\alpha} &  K^{\alpha}
 \end{pmatrix} 
\cV(s_+^\a(i))
 =
 \begin{pmatrix}\Phi K_\sharp^{\alpha-\b_\Phi} & \Phi_\sharp K^{\alpha-\b_\Phi}
 \end{pmatrix}.
\end{equation}
Consequently the Schur functions are related by
\begin{equation}\label{15jul6}
v(z) s_+^{\a-\b_\Phi}(z) = \frac{s_+^{\a}(z)-s_+^{\a}(i)}{1-s_+^{\a}(z)\overline{s_+^{\a}(i)}}.
\end{equation}
\end{lemma}

\begin{proof}
The $2$-dimensional space $\cH^2(\alpha) \ominus \Phi  \Phi_\sharp \cH^2(\alpha - 2 \beta_\Phi)$ has an orthonormal basis $K^{\alpha}, 
\Phi K_\sharp^{\alpha-\b_\Phi}$ and contains the normalized vector $K_\sharp^\a$, so that vector can be expressed in the form
\begin{equation}\label{1jul11}
K_\sharp^{\alpha} = a K^{\alpha } + \rho \Phi K_\sharp^{\alpha-\b_\Phi}.
\end{equation}
By normalization, $\lvert a \rvert^2 + \lvert \rho \rvert^2 = 1$, and by taking the inner product with $\Phi K_\sharp^{\alpha-\b_\Phi}$,
\[
\rho = \langle \Phi K_\sharp^{\alpha-\b_\Phi}, K_\sharp^{\alpha} \rangle =\frac{\Phi (-i) K_\sharp^{\alpha-\b_\Phi}(-i)}{K_\sharp^{\alpha}(-i)} > 0
\]
so $\rho = \sqrt{1- \lvert a\rvert^2}$. Evaluating \eqref{1jul11} at $i$ gives $a = K_\sharp^{\alpha}(i) /  K^{\alpha}(i)$, so
\[
\begin{pmatrix} K_\sharp^{\alpha} &  K^{\alpha}
 \end{pmatrix} 
 \begin{pmatrix} 1 \\ -a 
 \end{pmatrix}
\frac 1{\sqrt{1-a\bar a}}
 =
\Phi K_\sharp^{\alpha-\b_\Phi}.
\]
Applying the antilinear involution $(\dots)_\sharp$ gives
\[
\begin{pmatrix} K_\sharp^{\alpha} &  K^{\alpha}
 \end{pmatrix} 
 \begin{pmatrix} - \bar a \\ 1 
 \end{pmatrix}
\frac 1{\sqrt{1-a\bar a}}
 =
\Phi_\sharp K^{\alpha-\b_\Phi}.
\]
Combining the two equalities in matrix form gives \eqref{26may1}, whereas taking their ratio gives \eqref{15jul6}.
\end{proof}

The following corollary is a matrix form of the relation \eqref{26may1}.

\begin{corollary}\label{l18jul1new}
For any $\alpha$,
\[
  \begin{pmatrix} \tau_*  K^{\b_\cW-\a}  & \bar \tau_* K_\sharp^{\beta_\cW - \alpha} \\ \Phi K_\sharp^{\a-\b_\Phi} & \Phi_\sharp K^{\a-\b_\Phi}
 \end{pmatrix} = \cT_{\alpha} \cV(s_+^\a(i)).
\]
\end{corollary}

\begin{proof}
The second row of this statement is precisely \eqref{26may1}. The first row follows from \eqref{26may1} after applying the involution $(\dots)_\flat$ and multiplying by $\cW \Phi \Phi_\sharp$.
\end{proof}

Also, using \eqref{26may1}, we get a complementary to \eqref{30jun2} representation for the reproducing kernel.

\begin{lemma}
\begin{equation}\label{28jun1c}
 \Phi_\sharp(z) \overline{\Phi_\sharp(z_0)}  k^{\alpha-\b_\Phi}(z,z_0) = \frac{ K^\a(z) \overline{K^\a(z_0)} - K_\sharp^\a(z) \overline{K_\sharp^\a(z_0)}}{1-v(z) \overline{v(z_0)}}.
\end{equation}
\end{lemma}

\begin{proof}
By writing
\[
\begin{pmatrix} K_\sharp^{\alpha }(z) & K^{\alpha }(z)
\end{pmatrix} \cV(s_+^\a(i)) = \Phi_\sharp(z) \begin{pmatrix} v(z) K_\sharp^{\alpha-\b_\Phi}(z) & K^{\alpha-\b_\Phi}(z)
\end{pmatrix},
\]
using this for $z$ and $z_0$ implies, since $\cV(s_+^\a(i))$ is $j$-unitary, that
\begin{align*}
& \begin{pmatrix} K_\sharp^{\alpha }(z) & K^{\alpha }(z)
\end{pmatrix} j \begin{pmatrix} K_\sharp^{\alpha }(z_0) & K^{\alpha }(z_0)
\end{pmatrix}^*  \\
& = \Phi_\sharp(z)\overline{\Phi_\sharp(z_0)} \begin{pmatrix} v(z) K_\sharp^{\alpha-\b_\Phi}(z) & K^{\alpha-\b_\Phi}(z)
\end{pmatrix} j \begin{pmatrix} v(z_0) K_\sharp^{\alpha-\b_\Phi}(z_0) & K^{\alpha-\b_\Phi}(z_0)
\end{pmatrix}^*.
\end{align*}
Applying \eqref{30jun2} with $\a-\b_\Phi$ instead of $\a$ concludes the proof.
\end{proof}

The previous lemma is closely related to one entry of the matrix product $\cT_\a(z) j(\cT_\a(z_0))^*$; in fact, we can compute all entries of this product.

\begin{lemma} \label{lemma29jun3}
\begin{equation}\label{28jun2}
i\frac{\cT_\a(z)  j  \cT_{\a}(z_0)^*}{z-\bar z_0} = 2  \frac{\Phi_\sharp(z)}{z + i}\overline{\frac{\Phi_\sharp(z_0)}{z_0+i}}
\begin{pmatrix}
- k^{\tilde\a}(z,z_0) &  \cW(z) (k^{\a-\b_\Phi}_{z_0})_\flat (z) \\
 -\cW(z) (k^{\tilde\a}_{z_0})_\flat (z)  & k^{\a - \b_\Phi}(z,z_0)
\end{pmatrix}
\end{equation}
\end{lemma}

\begin{proof}
We prove the equality of matrices entry by entry.
Equality of the $(2,2)$-entry in \eqref{28jun2} follows from \eqref{28jun1c} and the algebraic manipulations \eqref{2jul1}.
Equality of the $(1,2)$-entry in \eqref{28jun2} follows from the equality of the $(2,2)$-entry by applying the involution $(\dots)_\flat$ and multiplying by $\cW \Phi\Phi_\sharp$, since $ \begin{pmatrix} 1 & 0 \end{pmatrix} \cT_\a  = \cW \Phi \Phi_\sharp \left( \begin{pmatrix} 0 & 1 \end{pmatrix} \cT_\a \right)_\flat$. Similarly, equality of the $(1,1)$-entry follows from  \eqref{2jul1} and \eqref{30jun2} with $\tilde\a$ instead of $\a$, and equality of the $(2,1)$-entry follows by applying the involution $(\dots)_\flat$ and multiplying by $\cW \Phi\Phi_\sharp$.
\end{proof}

\section{Reflectionless pairs of Schur functions: classes $\cS(\sE)$ and $\cS_A(\sE)$ and their parametrization} \label{sectionAbelMap}

The reflectionless property is defined in terms of half-line Schur functions, but is a property of a whole line system/operator, and many consequences of the reflectionless property are best seen from the perspective of whole line resolvents. We define the resolvent function
 \begin{equation}\label{21apr1b}
R(z)= i \frac{(1- s_+(z))(1- s_-(z))}{1-s_+(z) s_-(z)}.
\end{equation}
For instance, the spectral interpretation of Definition~\ref{defnSEnew}.(iii) is that a corresponding whole-line operator does not have spectrum outside of $\sE$.

\begin{lemma} \label{lemmaresolventfunctionbasic}
If $s_+ \in \cS(\sE)$,
 then  $R$ is a Herglotz function, analytic in $\bbC \setminus \sE$, with the symmetry $R_\sharp = R$. Moreover, $\lim_{\epsilon \downarrow 0} \arg R(\xi + i \epsilon) = \frac \pi 2$ for Lebesgue-a.e. $\xi \in \sE$.
\end{lemma}

\begin{proof} By Cayley transforms we obtain the Herglotz functions
\begin{equation}\label{16jul3b}
m_\pm(z)=i\frac{1+  s_\pm(z)}{1- s_\pm(z)}
\end{equation}
which obey $(m_\pm)_\sharp = m_\pm$. A direct calculation gives $R = - 2/(m_+ + m_-)$, so $R$ is Herglotz, meromorphic on $\bbC \setminus \sE$, and $R_\sharp = R$. Since $1-s_+s_-$ is nonzero on $\bbR \setminus \sE$, $R$ has no poles there.  Since $s_- = (s_+)_*$, a calculation gives $\lim_{\epsilon \downarrow 0} R(\xi + i \epsilon) \in i \bbR$ for Lebesgue-a.e. $\xi \in \sE$. Since that limit is a.e.\ nonzero and $R$ is Herglotz, the normal limit of the argument is $\pi/2$.
\end{proof}

\begin{corollary}
If $s_+, s_-$ are a reflectionless pair of Schur functions with a Widom spectrum $\sE$, then $s_\pm \in \cN(\Omega)$.
\end{corollary}
\begin{proof}
Since $R$ has no poles on $\bbR \setminus \sE$, $m_\pm$ have at most one pole in each gap, so applying Theorem \ref{thsy97} shows that $m_\pm \in \cN(\Omega)$ and therefore $s_\pm \in \cN(\Omega)$.
\end{proof}

\begin{remark} 
The reflectionless property often appears in the literature in the form for Titchmarsh--Weyl $m$-functions; Definition~\ref{defnSEnew} converts to that form with the substitutions \eqref{16jul3b}. In particular, by calculations like those above, the nonvanishing of $1-s_+(z)s_-(z)$ implies analyticity in $\Omega$ of both the symmetric combinations
\[
-\frac{1}{m_+(z)+m_-(z)}, \quad \frac{m_+(z)m_-(z)}{m_+(z)+m_-(z)}
\]
which appear in whole-line Titchmarsh--Weyl $M$-matrix functions.
\end{remark}

\subsection{Schur spectral functions and unitary nodes. The map $\pi_1(\Omega)^{*}\times\bbT\to\cS_A(\sE)$}

In the sense of spectral theory, the Schur class of functions, and its generalization to matrix (operator)-valued functions, is associated to the concept of \textit{unitary nodes}. Let $E_1, E_2$ be complex Euclidean spaces. We say that $\cE(z)$ belongs to the Schur class $\cS(E_1,E_2)$ if 
it is a linear operator-valued function, $\cE(z):E_1\to E_2$ for a fixed $z\in\bbC_+$, analytic in $z$, and $\|\cE(z)\|\le 1$ for all $z\in\bbC_+$.

By passing to a matrix representation with respect to orthonormal bases of $E_1, E_2$, the operator-valued function $\cE$ gives a matrix-valued function. Of course, a change of basis would correspond to a change of matrix-valued function, as we will see in examples below. Respectively, this or that way of basis fixing leads to this or that gauge normalization condition for transfer matrices.

\begin{definition}
Let $H$ be a Hilbert space.
By a \textit{unitary node} we mean a unitary operator $U$ acting from $E_1 \oplus H$ to $E_2 \oplus H$. $H$ is  called the \textit{state space} and $E_1, E_2$ are called \textit{coefficient spaces}. The operator function
\begin{align*}
S(z,U)=&P_{E_2}(I_{H\oplus E_2}-v(z) U P_H)^{-1}U|_{E_1}\\
=&P_{E_2}U(I_{H\oplus E_1}-v(z) P_H U)^{-1}|_{E_1}
\end{align*}
is called the \textit{characteristic function} of the unitary node. Here $P_K$ and $P_{E_2}$ are the orthogonal projections onto the corresponding subspaces.
\end{definition}

\begin{theorem}
The characteristic function $S(z,U)$ of a unitary node $U:H\oplus E_1\to H\oplus E_2$ belongs to the class $\cS(E_1,E_2)$. Vice versa, if $S\in \cS(E_1,E_2)$, then there exists a unitary node such that $S(z)=S(z,U)$.
\end{theorem}

Essentially, this theorem is a certain point of view on the Nagy-Foias theory \cite{NF}; for a ``bird's eye view" on the subject see \cite{SzN}.  For the reader's convenience in the preprint form we prove this theorem in the appendix.

Recall that the function $v(z)$ was defined in \eqref{30mara}.

\begin{proposition}\label{prop17jul1}
Multiplication by $\bar v$ in the decomposition
\begin{equation}\label{15jun1}
\bar v: \left\{\frac{K_\sharp^{\a} }{\Phi_\sharp}\right\}\oplus \cH^2(\alpha-\beta_\Phi)\to \left\{\frac{K^{\a}}{\Phi}\right\} \oplus \cH^2(\alpha-\beta_\Phi)
\end{equation}
forms a unitary node with the state space $H=\cH^2(\a-\beta_\Phi)$.
Its characteristic  function is
$$
s^\a_+(z) = \frac{K_\sharp^{\a}(z)}{K^{\a}(z)}.
$$
\end{proposition}

\begin{proof}[Proof of  Proposition  \ref{prop17jul1}] 
Multiplication by $\bar v = 1/v$ is a unitary operator from $\frac 1{\Phi_\sharp} \cH^2(\alpha) = \left\{\frac{K_\sharp^{\a} }{\Phi_\sharp}\right\}\oplus \cH^2(\alpha-\beta_\Phi)$ to $\frac 1{\Phi} \cH^2(\alpha) = \left\{\frac{K^{\a} }{\Phi}\right\}\oplus \cH^2(\alpha-\beta_\Phi)$.

Fix $z_0 \in \bbC_+$. To compute the value of the characteristic function $S(z_0)$, we write
\[
P_{E_2} (I - v(z_0) U P_H)^{-1} U\vert_{E_1} \frac{K_\sharp^{\alpha}}{\Phi_\sharp} = \frac{K^{\alpha}}{\Phi}    S(z_0)
\]
so for some $g \in \cH_\Omega^2(\alpha-\beta_\Phi)$,
\[
(I - v(z_0) U P_H)^{-1} U\vert_{E_1}  \frac{K_\sharp^{\alpha}}{\Phi_\sharp} = g + \frac{K^{\alpha}}{\Phi}  S(z_0)
\]
and therefore
\[
U  \frac{K_\sharp^{\alpha}}{\Phi_\sharp} = g - v(z_0) U g + \frac{K^{\alpha}}{\Phi}  S(z_0).
\]
Multiplying by $v$ gives
\[
 \frac{K_\sharp^{\alpha}}{\Phi_\sharp} = (v - v(z_0)) g + v f \frac{K^{\alpha}}{\Phi}  S(z_0).
\]
Evaluating at $z = z_0$ gives
\[
 \frac{K_\sharp^{\alpha}(z_0)}{\Phi_\sharp(z_0)} = v(z_0) \frac{K^{\alpha}(z_0)}{\Phi(z_0)} S(z_0).
\]
Since $z_0 \in \bbC_+$ was arbitrary, solving for $S$ completes the proof.
\end{proof}

\begin{remark}
In Prop.~\ref{prop17jul1}, we implicitly took the basis vector $\frac{K^\a_\sharp}{\Phi_\sharp}$ for the coefficient space $E_1$ and $\frac{K^\a}{\Phi}$ for $E_2$. If we had multiplied those basis vectors by some unimodular constants, we would have obtained characteristic functions of the form
\[
s_+^{\a,\tau}(z) =\tau s_+^{\a}(z) = \tau \frac{K^\a_\sharp(z)}{K^\a(z)}, \qquad (\alpha,\tau) \in \pi_1(\Omega)^* \times \bbT.
\]
\end{remark}

The functions obtained here are closely related to Prop.~\ref{prop15jul1}. However, the reader should notice the difference between the Cayley transforms, which are operators on a single Hilbert space, and unitary nodes, which have the same state space. The unitary operator $\hat U_\tau$ and the unitary node \eqref{15jun1} are related by the following commutative diagram
$$
\begin{CD}
\left\{\frac{K^\a_{\sharp}}{\Phi_\sharp}\right\}\oplus \cH^2(\a-\b_\Phi) @>>{\bar v}>
\left\{ \frac{K^\a}{\Phi}\right\}\oplus \cH^2(\a-\b_\Phi)  \\
@VV{\hat \Phi_\tau}V @VV{\Phi}V\\
\cH^2(\a)=\{K^\a\}\oplus \fD_{\bar v}@>>{\hat U_\tau}> \cH^2(\a)=\{K^\a_\sharp\}\oplus\Delta_{\bar v}\\
\end{CD}
$$
where 
$$
\hat \Phi_\tau |_{\cH^2(\a-\b_\Phi)}=\Phi \quad\text{and}\quad \hat\Phi_\tau : \tau \frac{K^\a_\sharp}{\Phi_{\sharp}} \mapsto K^\a.
$$

It follows from earlier observations that the functions $s_+^{\a,\tau}$ are in this class. Let us denote
\begin{equation}\label{8sep1}
s^{\alpha,\tau}_-(z) = \bar \tau s^\alpha_-(z), \qquad s^\alpha_-(z)=
\tau_*^{-2}\frac{\Phi(z) K_\sharp^{\tilde\a}(z)}{\Phi_\sharp(z) K^{\tilde\alpha}(z)}.
\end{equation}

\begin{corollary} \label{cor8sep4}
For any $(\alpha,\tau) \in \pi_1(\Omega)^* \times \bbT$,  the pair  $(s_+^{\a,\tau}, s_-^{\a,\tau})$ is a reflectionless pair of Schur functions and $s_+^{\a,\tau} \in \cS_A(\sE)$.
\end{corollary}

\begin{proof}
The property $(s_+^{\a,\tau})_\sharp = 1/s_+^{\a,\tau}$ follows from the definition and $(s_+^{\a,\tau})_* = s_-^{\a,\tau}$ from \eqref{11jul89}. Moreover, $1 - s_+ s_-$ has no zeros in $\bbR \setminus \sE$ by the Wronskian identity, Lemma~\ref{lemmaTadet}.
\end{proof}

\subsection{Divisors $\cD(\sE)$. The map $\cS_A(\sE)\to \cD(\sE)$}

We begin with the resolvent function $R(z)$ defined as in \eqref{21apr1b} from $s_+ \in \cS_A(\sE)$.

\begin{lemma}
For each gap $(a_j, b_j)$, there exist $x_j \in [a_j, b_j]$ such that
\begin{equation}\label{25jun1}
R(z)=i|1- s_+(i)|
e^{\int_{\bbR}\left(\frac{1}{\xi-z}-\frac{\xi}{1+\xi^2}\right)\chi(\xi)d\xi},
\end{equation}
where
\begin{equation}\label{25un2}
\chi(\xi) = \begin{cases}
1/2,&x\in(a_j,x_j)\\
0,& x\in\sE\\
-1/2, &x\in(x_j,b_j)
\end{cases}
\end{equation}
\end{lemma}

\begin{proof}
From the Herglotz representation of $R$, it follows that it is strictly increasing on each gap $(a_j, b_j)$, so there exist $x_j \in [a_j,b_j]$ such that $R$ is negative on $(a_j,x_j)$ and positive on $(x_j,b_j)$. It follows that the boundary values $\frac 1\pi \arg R(\xi + i0)$ for $\xi \in \bbR$ are given by \eqref{25un2}. Since $s_+\in \cS_A(\sE)$, \eqref{21apr1b} implies that $\lvert R(i) \rvert = \lvert 1 -s_+(i) \rvert$, the exponential Herglotz representation of $R$ gives \eqref{25jun1}. 
\end{proof}

Assume that $x_j \in (a_j, b_j)$ for some $j$. Then $x_j$ is a zero of $R$, so by \eqref{21apr1b}, at least one of the functions $s_\pm$ is equal to $1$ at $x_j$. They cannot both be equal to $1$, since $s_+ s_- \neq 1$ on $\Omega$. Define the sign $\e_j \in \{\pm \}$ so that $s_{\e_j}(x_j) = 1$. Of course, if $x_j = a_j$ or $x_j = b_j$, it does not correspond to a zero on $\Omega$, nevertheless by continuity of $s_\pm(z)$ at the ends of the gap we have $1=\overline{s_+(x_j)}=s_-(x_j)$ in the sense of nontangential limit at the gap endpoint $x_j \in \{ a_j, b_j \}$.

\begin{definition} 
We define $\cD(\sE)$ as the product, with the product topology,
\begin{equation}\label{divisor25jun}
\cD(\sE)= \prod_{j\in\bbZ} I_j,
\end{equation}
where each $I_j$ is a double cover of the corresponding gap with edges identified and endowed with a  topology of a circle,
\[
I_j = \{ (x_j, \e_j) \mid x_j \in [a_j,b_j] \times \{+1,-1\} \} /_{\substack{(a_j,+) \sim (a_j,-)\\(b_j,+) \sim (b_j,-)}}.
\]
\end{definition}

{\em The above construction describes the map $\cS_A(\sE) \to \cD(\sE)$} given by
\begin{equation}\label{15jul202}
s_+ \mapsto D=\{(x_j,\e_j)\}_{j\in\bbZ}.
\end{equation}

For a better understanding of our further steps assume that indeed $s_+ = s_+^{\a,\tau}$. Using the Wronskian identity \eqref{7jul3}, we obtain
\begin{align}\label{21apr3}
R^{\a,\tau}(z)=&
i\frac{(1-\tau s^\a_+(z))(1-\bar \tau s^\a_-(z))}{1-s^\a_+(z) s^\a_-(z)}\nonumber\\
=& \vk^{\a,\tau}(z) \vk^{\a,\tau}_*(z)
\frac{i(z-i)(z+i)\Theta'}{2 \Phi \Phi_\sharp},
\end{align}
where 
\[
\vk^{\a,\tau}(z)=K^\a(z)(1- s_+^{\a,\tau}(z))=K^\a(z) - \tau K_\sharp^\a(z), \qquad (\a,\tau)\in\pi_1(\Omega)^* \times\bbT, 
\]
and respectively 
\[
\Phi\Phi_\sharp \cW(z)(\vk^{\a,\tau})_*(z)=\tau_* \Phi_\sharp K^{\tilde \a} - \bar \tau \bar\tau_* \Phi K_\sharp^{\tilde\a}.
\]
Thus, the resulting factorization \eqref{21apr3} of $R^{\a,\tau}$ leads to the \textit{symmetric combinations of reproducing kernels} 
$\vk^{\a,\tau}$. The following theorem shows that this symmetric reproducing kernel $\vk^{\a,\tau}$ can be expressed in terms of the divisor $D$ assigned to $s^{\a,\tau}_+$.

\begin{theorem}
The symmetric reproducing kernel $\vk^{\a,\tau}$ in terms of the divisor $D\in \cD(\sE)$ possesses the following multiplicative representation
\begin{equation}\label{15jul201}
\vk^{\a,\tau}(z)=C
\left\{ \frac{\Phi (z) }{z - i} \frac{\Phi_\sharp (z) }{z + i}  \prod_j\frac{z-x_j}{\sqrt{1+x_j^2}\Phi_{x_j}(z)}\frac{\sqrt{1+c_j^2}\Phi_{c_j}(z)}{z-c_j}\right\}^{\frac 1 2}\prod_j\Phi_{x_j}(z)^{\frac{1+\e_j}{2}}
\end{equation}
Respectively,
the ratio
\begin{equation}\label{28apr2}
\Delta^{\a,\tau}=\frac{ \tau_*  \Phi_\sharp K^{\tilde \alpha} -\bar \tau \bar \tau_* \Phi K_\sharp^{\tilde \alpha}}{ K^{\alpha} - \tau K_\sharp^{\alpha}}
\end{equation}
is a Blaschke product, given explicitly in terms of the divisor \eqref{15jul202} by
\begin{equation}\label{28apr2b}
\Delta^{\a,\tau}=
\prod_{j} \Phi_{x_j}(z)^{-\e_j}.
\end{equation}

\end{theorem}
\begin{proof} We already demonstrated similar arguments in proving Lemma \ref{lemma15jul201}.
By Theorem \ref{thsy97}, the function $R^{\a,\tau}$ doesn't have a singular inner factor. Therefore we can separate the inner part of $\vk^{\a,\tau}$ as the Blaschke product, which is the second factor in \eqref{15jul201}. 
Note now that  both factors 
$\vk^{\a,\tau}$ and
$\Phi\Phi_{\sharp}\cW\vk^{\a,\tau}_*$
 have the same absolute value on $\partial\Omega$.  So the outer part of each of them is the square root  of the outer part of their product,
which  (up to a multiplicative constant)  is given as the product in brackets.  Since the outer parts coincide and the zeros of numerator are exactly complementary to the denominator, we have \eqref{28apr2b}.
\end{proof}

\subsection{The Abel map $\cD(\sE)\to \pi_1(\Omega)^{*}\times\bbT$}\label{sectabm}

We now generalize the above correspondence, i.e., starting from an arbitrary divisor $D=\{(x_j,\e_j)\}_{j\in\bbZ}  \in \cD(\sE)$, we consider the product
\begin{equation}\label{15jul201copy}
\vk_D(z)=C\left\{  \frac{\Phi (z) }{z - i} \frac{\Phi_\sharp (z) }{z + i}  \prod_j\frac{z-x_j}{\sqrt{1+x_j^2}\Phi_{x_j}(z)}\frac{\sqrt{1+c_j^2}\Phi_{c_j}(z)}{z-c_j}\right\}^{\frac 1 2}\prod_j\Phi_{x_j}(z)^{\frac{1+\e_j}{2}}
\end{equation}
Note that if $x_j \in \{a_j, b_j\}$ then $\Phi_{x_j} \equiv 1$ and the value of $\epsilon_j$ is irrelevant. We denote the character of the product $\vk=\vk_D$ by $\alpha=\a(D)$.

\begin{lemma}\label{15sepl13}
$\vk$ is a linear combination of $K^\alpha$ and $K^\alpha_\sharp$. With the right choice of $C$ in \eqref{15jul201copy}, $\vk$ is of the form $\vk = K^\a - \tau K^\a_\sharp$ with $\tau \in \bbT$.
\end{lemma}

\begin{proof}
Assume $f \in \cH^2(\a)$ is such that $f(i) =0$ and $f(-i) = 0$. Decompose $f = \Phi \Phi_\sharp g$ where $g \in \cH^2(\a-2\b_\Phi)$, and denote by $\vk = C \vk_o \vk_i$ the inner-outer decomposition \eqref{15jul201copy}. By comparing $\vk$ and $\overline{\vk}$ on $\partial\Omega$, we obtain
\begin{equation}\label{17oct1}
\overline{\vk} = \frac{\overline{C}}{C} \vk \frac 1{\Phi \Phi_\sharp \cW} \prod_j  \Phi_{x_j}^{-\e_j} = \frac{\overline{C}}{C} \vk_o \frac 1{\Phi \Phi_\sharp \cW} \prod_j  \Phi_{x_j}^{\frac{1-\e_j}2} \qquad \text{a.e.\ on }\partial\Omega.
\end{equation}
In addition to relating the boundary values, since the character of $\overline{\vk}$ is $-\a$, this implies that the character of $\vk_o \prod_j  \Phi_{x_j}^{\frac{1-\e_j}2}$ is $2\b_\Phi + \b_\cW -\a$. Thus, $g \in \cH^2(\a -2\b_\Phi)$ and $\vk_o  \prod_j  \Phi_{x_j}^{\frac{1-\e_j}2} \in \cH^2(2\b_\Phi + \b_\cW -\a)$, so by DCT, we compute
\[
\langle f, \vk \rangle = \int f \overline{\vk} d\Theta = \oint_{\partial\Omega}  g \vk_o \prod_j  \Phi_{x_j}^{\frac{1-\e_j}2} \frac{d\Theta}{\cW} = 0.
\] 
Thus, $\vk$ is orthogonal to $\Phi \Phi_\sharp \cH^2(\a - 2 \b_\Phi)$, so $\vk = C_1 K^\a + C_2 K^\a_\sharp$ for some $C_1, C_2$. Moreover, the representation $\vk$ implies that $\vk_\sharp$ is a multiple of $\vk$, so $\lvert C_1 \rvert = \lvert C_2 \rvert$. Thus, $\vk$ can be normalized so that $\vk = K^\a - \tau K^\a_\sharp$ with $\tau \in \bbT$.
 \end{proof}
 
 {\em This procedure gives us an Abel map $\pi: D\mapsto (\a, \tau)$,  $\pi: \cD(\sE) \to \pi_1(\Omega)^* \times \bbT$}. To provide explicit formulas for this map,  we denote by $\gamma_k$ for $k\neq 0$ the generators of $\pi_1(\Omega)$ so that $\gamma_k$ intersects $\bbR \setminus \sE$ "upward" through $\xi_*$ and "downward" through the gap $(a_k, b_k)$. In other words, the contour $\gamma_k$ has winding number $1$ if $b_k < a_0$ and winding number $-1$ if $a_k < b_0$. Denote by $\sE_k$ the part of $\sE$ between the $0$-th and $k$-th gaps. Finally, we denote
\begin{equation}\label{12oct1}
\cA_{\g_k}(D)=  \sum_{j} \frac{ \e_j}2 ( \omega(x_j, \sE_k) -\omega(a_j,\sE_k)) \mod \bbZ
\end{equation}

\begin{lemma} The Abel map $\pi: \cD(\sE) \to \pi_1(\Omega)^* \times \bbT$ is continuous and given by the following explicit formulas: 
\begin{equation}\label{4sep1}
\a(\g_k)=\b_{\Phi}(\g_k)+\cA_{\g_k}(D)-\cA_{\g_k}(D_c), \quad \text{with}\ D_c=\{(c_j, -1)\}
\end{equation}
for any $k$, and
\begin{equation}\label{17oct5}
\tau = - \overline{\tau_*}^2 \frac{\vk_{D_*}(i)}{\vk_{D_*}(-i)}.
\end{equation}
\end{lemma}

\begin{proof}
For fixed $k$, fix domains $\Pi_k^\pm$ in $\hat\bbC$ bounded by simple Jordan curves $\gamma_k^\pm$ in $\Omega \setminus \{\xi_* \}$ such that $\Pi_k^- \cap \sE =  \sE_k$, $\Pi_k^+ \cap \sE = \sE \setminus \sE_k$, and $\xi_* \notin \Pi_k^\pm$. Note that  $\omega(z,\sE_k)$  is zero on $\Pi_k^- \cap \sE$ and  $1 - \omega(z,\sE_k)$ is zero on $\Pi_k^+ \cap \sE$.  By comparing the harmonic functions $\omega(z,\sE_k)$ and $1 - \omega(z,\sE_k)$ with $G(z,\xi_*)$ on the compact images $\gamma_k^\pm$ and applying the maximum principle on the domains $\Pi_k^\pm \setminus \sE$, we conclude  the existence of $C_k$ such that
\[
\omega(z,\sE_k) \le C_k G(z,\xi_*) \quad \forall z \in \Pi_k^-, \quad 1- \omega(z,\sE_k) \le C_k G(z,\xi_*) \quad \forall z \in \Pi_k^+.
\]
Combined with the Widom condition \eqref{15jul1} this implies that the series \eqref{12oct1} is absolutely summable uniformly in $D$, so $\cA_{\gamma_k}(D)$ is a continuous function of $D$.

For fixed $z_0 \in \Omega$, consider the outer function $\frac{z-z_0}{\Phi_{z_0}}$; its boundary values have absolute value $\lvert z -z_0\rvert$ a.e., so we obtain the representation
\[
 \log \frac{ \lvert z-z_0  \rvert}{\lvert \Phi_{z_0}(z) \rvert} = \int \log\lvert x-z_0 \rvert  \, \omega(dx,z).
\]
Using $G(z,z_0) = G(z_0,z)$ we can switch the roles of $z, z_0$ and then pass to harmonic conjugates to obtain
\begin{equation}\label{7oct5}
\arg \frac{  z-z_0 }{ \Phi_{z_0}(z) } = \int \arg (x-z)  \, \omega(dx,z_0)+ C
\end{equation}
as an equality of multi-valued harmonic functions up to an additive constant $C$ independent of $z$. In particular, they have the same additive characters,  and their additive jumps along the closed loop $\gamma_k$ are equal to $2\pi \omega(z_0,\sE_k)$.  It follows that the character of a product $\prod_j\Phi_{x_j}(z)^{\frac{\e_j}{2}}$ is $\cA_{\gamma_k}(D)$.

The product \eqref{15jul201copy} can be regrouped as
\begin{equation}\label{12oct2}
\vk_D(z) =C
\left\{  \frac{\Phi (z) }{z - i} \frac{\Phi_\sharp (z) }{z + i} \right\}^{\frac 1 2}\left\{  \prod_j\frac{(z-x_j)\sqrt{1+c_j^2}}{(z-c_j)\sqrt{1+x_j^2}}\right\}^{\frac 1 2}\prod_j\Phi_{x_j}(z)^{\frac{\e_j}{2}}\prod_j\Phi_{c_j}(z)^{\frac 12}
\end{equation}
Note that the second bracket is a meromorphic, single-valued function on $\Omega$ and we can assume without loss of generality that $\gamma_k$ does not contain any points in the intervals $[c_j,x_j]$. Thus, combining the characters of all the factors in \eqref{12oct2} gives \eqref{4sep1}.

To compute $\tau$, we use $\vk_*$. Since \eqref{11jul89} gives
\[
\Phi \Phi_\sharp \cW  \vk_* = \tau_* \Phi_\sharp K^{\tilde\a} - \overline{\tau_* \tau} \Phi K_\sharp^{\tilde\a},
\]
we see that $\overline{\tau_*}  (\Phi \Phi_\sharp \cW  \vk_* )(i) = \Phi_\sharp(i) K^{\tilde\a}(i) > 0$ and similarly $\tau_* \tau  (\Phi \Phi_\sharp \cW  \vk_* )(-i) < 0$. By \eqref{17oct1}, $ \Phi \Phi_\sharp \cW  \vk_*= C_1\vk_{D_*}$, where $D_*=\{(x_j,-\e_j)\}$. Thus, $\tau_*^2 \tau \vk_{D_*}(-i) / \vk_{D_*}(i) < 0$. Since $\lvert \vk_{D_*} \rvert$ is symmetric, this implies \eqref{17oct5}.
\end{proof}

In particular, all other components of the Abel map correspond to closed curves, but $\tau$ corresponds to a jump in argument from $-i$ to $i$ through the gap $(a_0,b_0)$. Changing the normalization to a different gap would correspond to a change of $\tau$ by another component of the Abel map.

\subsection{Parametrization of the class $\cS_A(\sE)$: proof of the uniqueness theorem }\label{sect5u}
We have described the construction of three maps
\begin{equation}\label{11oct1}
\pi_1(\Omega)^* \times \bbT \to \cS_A(\sE) \to  \cD(\sE) \overset{\pi}{\to} \pi_1(\Omega)^* \times \bbT.
\end{equation}

\begin{theorem} \label{theoremhomeomorphisms}
For a Dirichlet regular Widom set which obeys DCT, the three maps in \eqref{11oct1} are homeomorphisms.
\end{theorem}

\begin{proof}
We already know that the maps are continuous and that their composition in the order \eqref{11oct1} is the identity map. 
 It remains to prove some injectivity statements.

For $s_+\in \cS_A(\sE)$, let
$$
s_+\mapsto D\mapsto(\a,\tau), \quad -s_+\mapsto D_1\mapsto(\a_1,\tau_1).
$$
Then, combining definitions of all these maps we get
$$
\frac{1+s_+}{1-s_+}\frac{1+s_-}{1-s_-}=
\frac{(1+s_+)(1+ s_-)}{1-s_+ s_-} \frac{1-s_+s_-} {(1- s_+)(1- s_-)}
=\frac{\vk^{\a_1,\tau_1}\vk_*^{\a_1,\tau_1}}
{\vk^{\a,\tau}\vk_*^{\a,\tau}}.
$$
Since $i(1+s_+)/(1-s_+)$ is a Herglotz function, we can use once again Theorem \ref{thsy97}. Thus, having in mind that
 $s_-(i) = 0$, we obtain
$$
\frac{1+s_+}{1-s_+}=\frac{\vk^{\a_1,\tau_1}}
{\vk^{\a,\tau}}.
$$
Since this function is single-valued in $\Omega$, $\a_1=\a$. 
Using Lemma \ref{15sepl13},
for $\tau_2=\tau_1/\tau$ we get
$$
\frac{1+s_+}{1-s_+}=\frac{1-\tau_2 s_{+}^{\a,\tau}}{1-s_{+}^{\a,\tau}}=\frac{1+\tau_2}{2}+\frac{1-\tau_2}{2}\frac{1+s_{+}^{\a,\tau}}{1-s_{+}^{\a,\tau}}.
$$
Since $s_+(z)\in\bbT$ for $z$ in an arbitrary gap $(a_j,b_j)$, we get that the LHS is pure imaginary valued.  We have that in the RHS the real part vanishes, i.e.,
$$
\frac{1+\Re\tau_2}{2}-i\frac{\Im\tau_2 }{2}\frac{1+s_{+}^{\a,\tau}(z)}{1-s_{+}^{\a,\tau}(z)}=0, \quad z\in(a_j,b_j).
$$
Since $s_+^{\a,\tau}$ is not a constant, we get $\Im\tau_2=0$ and $\Re\tau _2=-1$. Finally, we obtain $s_+(z)=s^{\a,\tau}_+(z)$.

Finally, the Abel map is injective: if $\pi(D_1) = \pi(D_2) = (\a, \tau)$, then $D_1, D_2$ give the same product $\vk = \vk^{\a,\tau}$. However, $\vk$, $\vk_*$ determine $R$ by \eqref{21apr3} which uniquely determines the $x_j$ as zeros of $R$ and the $\epsilon_j$ according to whether $x_j$ is a zero of $\vk$ or $\vk_*$.
\end{proof}

In particular, the  Abel map is a homeomorphism. This allows us in what follows to repeatedly use the same trick: having a continuous $X(D)$ function on $\cD(\sE)$ by a superposition with the inverse $\pi^{-1}$ we get a continuous function $Y(\a,\tau)=X(\pi^{-1}(\a,\tau))$ on a compact abelian group $\pi_1(\Omega)^*\times\bbT$. In this way we obtain a so called sampling function, so that $y(t)=Y(\a-\eta t,e^{2i\theta t}\tau)$, proving almost periodicity of $y(t)$.

For any $s_+ \in \cS(\sE)$, we can reduce to the case $s_-(i) = 0$ by acting on $s_\pm$ with some automorphism of $\bbD$, uniquely up to the stabilizer subgroup of $0$. This observation implies that: 

\begin{corollary} \label{corSE}
The class $\cS(\sE)$ is parametrized by the noncompact space $\pi_1(\Omega)^* \times \PSU(1,1)$. 
\end{corollary}

\section{Reflectionless canonical systems via the chain of invariant subspaces $e^{i\ell\Theta}\cH^2(\a-\eta\ell)$} \label{sectionCanonical}

\subsection{Unitary nodes with the co-invariant $\cK_\Delta(\a)$ as the state space} 

Just as Prop.~\ref{prop17jul1}  reflects the spectral theory of a differential operator on a half axis, the next construction is related to a restriction of a differential operator on an interval. We will start from a quite general construction. Let $\Delta$ be an inner character automorphic function and denote its character by $\beta_\Delta$. Later we will specialize to the case
$\Delta=e^{i\ell \Theta}$, $\ell>0$. 

A general description of the functions in $\cH^2(\a)$ which have a pseudocontinuation is given by the following lemma.

\begin{lemma}\label{lemma29jun1}
Let $\Delta$ be an inner character automorphic function with the character $\beta_\Delta$. Denote
\[
\cK_\Delta(\alpha)=\cH^2(\alpha)\ominus \Delta \cH^2(\alpha-\beta_\Delta).
\]
Then $f\in \cK_\Delta(\alpha)$ implies that $f$ has a pseudocontinuation, and moreover
\begin{equation}\label{10jul1}
f_*(z)=\frac{g(z)}{\Delta(z)\cW(z)}
\end{equation}
for some $g \in \cK_\Delta(\b_\Delta + \b_\cW - \a)$.
\end{lemma}

\begin{proof}
Since $\overline{\Delta}f$ is orthogonal to $\cH^2(\alpha-\beta_\Delta)$, it follows by Theorem~\ref{thmDCTifandonlyif} that 
$g := \cW\Delta\overline{f}\in \cH^2(-\alpha+\beta_\Delta+\beta_{\cW})$. Moreover, $f \in \cH^2(\a)$ implies $\cW \bar f \perp \cH^2(\b_\cW - \a)$, so $g = \cW \Delta \bar f \perp \Delta \cH^2(\b_\cW - \a)$. We finally conclude that $g \in \cK_\Delta(\b_\Delta + \b_\cW - \a)$.
\end{proof}

\begin{lemma}
Multiplication by $\bar v$ is a unitary node with the state space $\cK_\Delta(\a - \b_\Phi)$ and two dimensional coefficient spaces: 
\begin{equation}\label{13jul99}
U_\alpha : \left\{ \frac{K_\sharp^{\alpha}}{\Phi_\sharp} \right\} \oplus \cK_\Delta(\alpha -\beta_\Phi) \oplus \{ \Delta K^{\alpha -\beta_\Phi - \beta_\Delta} \} \to \left\{ \frac{K^{\alpha}}{\Phi } \right\} \oplus \cK_\Delta(\alpha-\beta_\Phi)  \oplus \{ \Delta K_\sharp^{\alpha - \beta_\Phi - \beta_\Delta} \}.
\end{equation}
Moreover, if we denote
\begin{align*}
e_1 = \frac{ K_\sharp^{\alpha}}{\Phi_\sharp}, & \qquad e_2  = \Delta K^{\alpha - \beta_\Phi - \beta_\Delta}, \\
f_1  = \Delta K_\sharp^{\alpha - \beta_\Phi - \beta_\Delta}, & \qquad f_2  = \frac{K^{\alpha}}{\Phi}
\end{align*}
the characteristic function $S$ for the unitary node \eqref{13jul99}, written in the basis $(e_1,e_2)$ for the coefficient space $E_1$ and the basis $(f_1, f_2)$ for the coefficient space $E_2$, obeys
\begin{equation}\label{11jul3}
\begin{pmatrix}  (e_1)_\flat  & (e_2)_\flat \\ e_1 & e_2 \end{pmatrix}  = \begin{pmatrix}  (f_1)_\flat v & (f_2)_\flat v \\ f_1 v & f_2 v \end{pmatrix} S.
\end{equation}
\end{lemma}

\begin{proof}
On the space $L^2_{\partial\Omega}$, multiplication by $\overline{ v(z)} = 1 /v(z)$ is unitary. Since
\[
f\in \frac 1{\Phi_\sharp} \cH^2(\alpha) \iff \frac fv \in \frac 1{\Phi} \cH^2(\alpha)
\]
and
\[
f \in \Delta \Phi \cH^2(\alpha-\beta_\Delta-2\beta_\Phi) \iff  \frac fv \in \Delta\Phi_\sharp \cH^2(\alpha-\beta_\Delta-2\beta_\Phi),
\]
multiplication by $\bar v$ is a unitary map from $\frac 1{\Phi} \cH^2(\alpha) \ominus \Delta \Phi_\sharp \cH^2(\alpha-\beta_\Delta-2\beta_\Phi)$ to $\frac 1{\Phi_\sharp} \cH^2(\alpha) \ominus \Delta \Phi \cH^2(\alpha-\beta_\Delta-2\beta_\Phi)$. Decomposing these spaces, we get a unitary node with the state space $\cK_\Delta(\a-\beta_\Phi)$ and two dimensional coefficient spaces \eqref{13jul99}.

To compute $S(z_0)$, we use the fact that for any $c_1, c_2 \in \bbC$, 
\[
P_{E_2} (I - v(z_0) U_\alpha  P_{\cK_\Delta(\alpha-\beta_\Phi)})^{-1} U_\alpha \vert_{E_1} \begin{pmatrix} e_1 &  e_2 \end{pmatrix} \begin{pmatrix} c_1 \\ c_2 \end{pmatrix}= \begin{pmatrix} f_1 & f_2  \end{pmatrix} S(z_0)\begin{pmatrix} c_1 \\ c_2 \end{pmatrix}
\]
so for some $g \in \cK_\Delta(\alpha-\beta_\Phi)$,
\[
(I - v(z_0) U_\alpha  P_{\cK_\Delta(\alpha-b_*)})^{-1} U_\alpha \vert_{E_1} \begin{pmatrix} e_1 & e_2 \end{pmatrix} \begin{pmatrix} c_1 \\ c_2 \end{pmatrix}= g +  \begin{pmatrix} f_1 & f_2  \end{pmatrix} S(z_0) \begin{pmatrix} c_1 \\ c_2 \end{pmatrix}.
\]
Applying $I - v(z_0) U_\alpha  P_{\cK_\Delta(\alpha -\beta_\Phi)}$ and then multiplying by $v$ gives
\begin{equation}\label{7jul6}
\begin{pmatrix} e_1 & e_2 \end{pmatrix} \begin{pmatrix} c_1 \\ c_2 \end{pmatrix}  = (v - v(z_0)) g + v \begin{pmatrix} f_1 & f_2 \end{pmatrix} S(z_0) \begin{pmatrix} c_1 \\ c_2 \end{pmatrix}.
\end{equation}
Since all functions in \eqref{7jul6} have pseudocontinuations, applying the linear involution $(\dots)_\flat$ and using $v_\flat = v$ gives
\begin{equation}\label{7jul6flat}
\begin{pmatrix} (e_1)_\flat & (e_2)_\flat \end{pmatrix} \begin{pmatrix}  c_1 \\  c_2 \end{pmatrix}  = ( v - {v(z_0)} ) g_\flat + v \begin{pmatrix} (f_1)_\flat & (f_2)_\flat \end{pmatrix} {S(z_0)} \begin{pmatrix}  c_1 \\  c_2 \end{pmatrix}.
\end{equation}
Evaluating \eqref{7jul6} and \eqref{7jul6flat} at $z = z_0$, the unknown functions $g$, $g_\flat$ vanish from the equations and we obtain
\begin{align*}
\begin{pmatrix} e_1(z_0) & e_2(z_0) \end{pmatrix} \begin{pmatrix} c_1 \\ c_2 \end{pmatrix}  & =  v(z_0) \begin{pmatrix} f_1(z_0) &  f_2(z_0) \end{pmatrix} S(z_0) \begin{pmatrix} c_1 \\ c_2 \end{pmatrix} \\
\begin{pmatrix} (e_1)_\flat(z_0)  & (e_2)_\flat(z_0)  \end{pmatrix}  \begin{pmatrix} c_1 \\ c_2 \end{pmatrix} & = v(z_0) \begin{pmatrix} (f_1)_\flat(z_0)  & (f_2)_\flat(z_0)  \end{pmatrix} {S(z_0)} \begin{pmatrix} c_1 \\ c_2 \end{pmatrix}.
\end{align*}
Since $c_1, c_2$ are arbitrary, \eqref{11jul3} holds at $z_0 \in \bbC_+$. Since $z_0 \in \bbC_+$ is arbitrary, this concludes the proof.
\end{proof}

At this point let us compute
\begin{align*}
(e_1)_\flat & =  \frac{\tau_* K^{\beta_\cW +\beta_\Phi - \alpha}}{\cW}, \qquad (e_2)_\flat =  \frac{\bar\tau_*  \Delta_\flat K_\sharp^{\beta_\Delta + 2 \beta_\Phi + \beta_\cW - \alpha}}{\cW \Phi_\sharp } \\
(f_1)_\flat & =   \frac{\tau_* \Delta_\flat K^{\beta_\Delta + 2 \beta_\Phi + \beta_\cW - \alpha} }{\cW \Phi }, \qquad
(f_2)_\flat =   \frac{\bar\tau_* K_\sharp^{\beta_\cW + \beta_\Phi - \alpha} }{\cW}
\end{align*}

\subsection{Potapov--Ginzburg transform and transfer matrices corresponding to $\Delta$} 
We will now study the transfer matrix $ \fA_\Delta$ defined by
\begin{equation}\label{19jun1}
\cT_{\a}  \fA_{\Delta} = \begin{pmatrix} \Delta_\flat & 0 \\ 0 & \Delta \end{pmatrix} \cT_{\alpha -\beta_\Delta}
\end{equation}
with the matrix $\cT_\a$ defined in \eqref{7jul3a}. We will see that $\fA_\Delta$ is closely related to the unitary node \eqref{13jul99}. 

\begin{lemma} \label{lemma15oct1}
$\fA_\Delta$ is a well-defined meromorphic function on $\Omega$ and can only have poles at zeros of  $\Delta_\sharp$. 
If $\Delta = \Delta_\sharp$, then $\det \fA_\Delta = 1$. Moreover, if $\Delta(i) > 0$, $\fA_\Delta(i)$ is lower triangular with strictly positive diagonal entries.
\end{lemma}

\begin{proof}
Since $\det \cT_\a$ is outer and independent of $\a$, $\fA_\Delta$ is well-defined meromorphic by \eqref{19jun1}, poles can only come from $\Delta_\flat = 1 / \Delta_\sharp$, and $\det \fA_\Delta = \Delta_\flat \Delta$. In particular, if $\Delta = \Delta_\sharp$, then $\det \fA_\Delta = 1$.  From \eqref{7jul3a}, the matrix $\cT_\a(i)$ is lower triangular and $(\cT_\a(i))_{22} > 0$. If $\Delta(i) > 0$, a calculation shows that $\fA_\Delta(i)$ is lower triangular and $(\fA_\Delta(i))_{22} > 0$. By $\det \fA_\Delta = 1$, $(\fA_\Delta(i))_{11} > 0$.
\end{proof}

Straightforward calculations show that
\begin{align*}
\begin{pmatrix} (e_1)_\flat  & v (f_2)_\flat \\ e_1 & v f_2  \end{pmatrix} & = \begin{pmatrix} \frac 1{\cW \Phi_\sharp} & 0 \\  0 & \frac 1{\Phi_\sharp} \end{pmatrix}
\cT_{\alpha} \\
\begin{pmatrix} v (f_1)_\flat  & (e_2)_\flat \\ v f_1 & e_2  \end{pmatrix}  &
 = \begin{pmatrix} \frac 1{\cW \Phi_\sharp} & 0 \\  0 & \frac 1{\Phi_\sharp} \end{pmatrix}  \begin{pmatrix} \Delta_\flat & 0 \\ 0 & \Delta \end{pmatrix}
 \cT_{\a-\b_\Delta} \cV(s_+^{\a-\b_\Delta}(i))
 \end{align*}
(the last step uses Corollary~\ref{l18jul1new}) so the matrix $\cA_\Delta$ defined by
\begin{equation}\label{19jun3}
\cA_\Delta = \fA_\Delta  \cV(s_+^{\a-\b_\Delta}(i))^{-1}
\end{equation}
obeys
\begin{equation}\label{11jul2}
\begin{pmatrix} (e_1)_\flat  & v(f_2)_\flat \\  e_1 & vf_2  \end{pmatrix}\cA_\Delta(z)  = \begin{pmatrix}  v(f_1)_\flat  & (e_2)_\flat \\ v f_1 & e_2 \end{pmatrix}.
\end{equation}
Comparing \eqref{11jul2} with \eqref{11jul3}, we see that $\cA_\Delta$ and $S$ are related precisely by the  Potapov-Ginzburg transform. In general, the  Potapov-Ginzburg transform compactifies the class of $j$-contractive matrix functions by relating a $j$-contractive matrix function $\cA$ to a contractive matrix function $S$ so that
\begin{equation}\label{27apr1}
P+\cA Q=(Q+\cA P)S, 
\quad P=\begin{pmatrix} 1&0\\0&0\end{pmatrix},\ 
 Q=\begin{pmatrix} 0&0\\0&1\end{pmatrix},\ j=-P+Q
\end{equation}
In our case, applying the Potapov--Ginzburg transform to \eqref{11jul3} separates the terms containing $\Delta$ from those that don't contain $\Delta$.

Explicitly,
\begin{equation}\label{31jul202}
  S(z)=\begin{pmatrix} s_{11}(z)&s_{12}(z)\\s_{21}(z)&s_{22}(z)
\end{pmatrix}
=\begin{pmatrix}a_{11}(z)&0\\a_{21}(z)&1
\end{pmatrix}^{-1}\begin{pmatrix}1&a_{12}(z)\\0&a_{22}(z)
\end{pmatrix}
\end{equation}

\begin{lemma}\label{lemma27may1}
$\fA_\Delta$ is $j$-contractive for $z \in \bbC_+$.
\end{lemma}
\begin{proof}
On $\bbC_+$, away from the discrete set of poles of functions in \eqref{11jul2},
\begin{equation} \label{27may1}
 j-\cA j\cA^* = Q+\cA P\cA^*-P-\cA Q \cA^* = (Q+\cA P)(I-S S^*)(Q+\cA P)^*  \ge 0
\end{equation}
so $\cA_\Delta$ is $j$-contractive. Since $ \cV(s_+^{\a-\b_\Delta}(i))$ is $j$-unitary, by \eqref{19jun3}, $\fA_\Delta$ is also $j$-contractive on $\bbC_+$.
\end{proof}

\begin{lemma} \label{lemma16oct1}
The boundary values on $\sE$ from above and below coincide,
\begin{equation}\label{19jun2}
 \fA_\Delta(\xi+i0) = \fA_\Delta(\xi-i0), \qquad \text{a.e. }\xi \in \sE.
\end{equation}
\end{lemma}

\begin{proof}
$\cA_\Delta(z)$ has nontangential boundary values a.e., moreover
$$
\begin{pmatrix}  (e_1)_\flat  & v(f_2)_\flat \\ e_1 & v f_2 \end{pmatrix}(\xi\pm i0)\cA_\Delta(\xi\pm i0)  = \begin{pmatrix} v(f_1)_\flat  & (e_2)_\flat \\ v f_1 & e_2  \end{pmatrix} (\xi\pm i0).
$$
By the definition of the $\flat$-involution we have
$$
\begin{pmatrix}  e_1 & v f_2 \\  (e_1)_\flat  & v (f_2)_\flat 
\end{pmatrix}(\xi+ i0)\cA_\Delta(\xi -i0)  = \begin{pmatrix} v f_1 & e_2 \\ v (f_1)_\flat  & (e_2)_\flat \end{pmatrix} (\xi+i0).
$$
Multiplying both parts by $\begin{pmatrix} 0&1\\ 1& 0\end{pmatrix}$ and using \eqref{19jun3} gives \eqref{19jun2}.
\end{proof}

\subsection{From transfer matrices to reflectionless canonical systems in Arov gauge}\label{sec53}

We now specialize to the case $\Delta(z) = e^{i \ell \Theta(z)}$, with $\ell$ as a parameter. Let 
\begin{equation}\label{16jul2020}
\Theta(i)=\theta_r+i\theta_i,\quad \theta_r\in\bbR,\ \ \theta_i>0.
\end{equation}
We generalize the definition \eqref{7jul3a} and define for $(\a,\tau)\in \pi_1(\Omega)^*\times\bbT$,
\begin{equation}\label{7jul203a}
\cT_{\a,\tau}=\begin{pmatrix}
\tau_* \Phi_\sharp K^{\tilde\alpha} & \bar\tau_* \bar \tau\Phi K_\sharp^{\tilde\alpha} \\
 \tau K_\sharp^{\alpha} & K^{\alpha} 
 \end{pmatrix} = \cU_\tau^{-1}\cT_\a\cU_\tau 
\end{equation}
This reduces to \eqref{7jul3a} by conjugation with a diagonal unitary and $j$-unitary matrix $\cU_\tau$,
\begin{equation}\label{17oct7}
\cT_{\a,\tau}= \cU_\tau^{-1}\cT_\a\cU_\tau, \qquad \cU_\tau=\begin{pmatrix}\tau^{1/2}&0\\0&\tau^{-1/2}.
\end{pmatrix}.
\end{equation}
Sometimes it is convenient to pass to the $\SL(2,\bbC)$ normalization of this matrix, i.e. to $\Pi_{\a,\tau}(z) =(\det \cT_{\a,\tau}(z))^{-1/2} \cT_{\a,\tau}(z)$. 
Due to \eqref{7jul3}, $\det\cT_{\a,\tau}(z)$ does not depend of $(\a,\tau)$. Since
\begin{equation}\label{20jun1}
\cT_{\alpha,\tau} =
\begin{pmatrix}
\tau_* \Phi_\sharp K^{\tilde\alpha} & 0 \\
0 & K^\alpha
\end{pmatrix}
\begin{pmatrix}
1 & \bar\tau s_-^\alpha \\
\tau s_+^\alpha & 1
\end{pmatrix},
\end{equation}
 $\Pi_{\alpha,\tau}$ can also be written in the form
\begin{equation}\label{11may2}
\Pi_{\a,\tau}(z)=
\begin{pmatrix}
\sqrt{{\iota^\a(z)}}
&0\\0& 
(\sqrt{{\iota}^\a(z)})^{-1}
\end{pmatrix}
\frac{\begin{pmatrix}1& s_-^{\a,\tau}(z)\\
s_+^{\a,\tau}(z)& 1\end{pmatrix}}{\sqrt{1-s^{\a,\tau}_+(z) s^{\a,\tau}_-(z)}}.
\end{equation}
where
\begin{equation}\label{11may1}
\iota^\a=\frac{ \tau_* \Phi_\sharp  K^{\tilde \alpha} }{ K^{\alpha}}.
\end{equation}
We point out that $|K^{\tilde \alpha}(z) |=| K^{\alpha}(z)|$ on $\partial\Omega$. Therefore $\iota^\a$ is a meromorphic inner function. 

\medskip
With these notations we define the following family of matrices.
\begin{definition}\label{def43}
Let $(\alpha,\tau) \in \pi_1(\Omega)^*\times \bbT$.  We define the transfer matrix $\fA^{\alpha,\tau}(z,\ell)$ by the identity
\begin{align}\label{9mar202}
\fA^{\alpha,\tau}(z,\ell)=&\cT_{\alpha,\tau}(z)^{-1} \Lambda_{\Theta(z)-\theta_r}(\ell) \cT_{\alpha-\eta \ell, \tau}(z)\\
=&\Pi_{\alpha,\tau}(z)^{-1} \Lambda_{\Theta(z)-\theta_r}(\ell) \Pi_{\alpha-\eta \ell, \tau}(z), \label{9mar203}
\end{align}
where
\[
\Lambda_{\theta}(\ell)= \begin{pmatrix} e^{- i\ell\theta} & 0 \\ 0 & e^{i\ell\theta} \end{pmatrix}.
\]
\end{definition}

Note that the additive correction of $\Theta(z)$ by $\theta_r$ is required to obey $e^{-i\ell(\Theta(i) - \theta_r)} > 0$ and therefore for $\fA^{\a,\tau}(z,\ell)$ to obey the Arov gauge condition. 

Immediately from the definition we obtain the chain rule
\begin{equation}\label{chainruleAgauge}
\fA^{\alpha,\tau}(z, \ell_1+ \ell_2)=\fA^{\alpha,\tau}(z, \ell_1)\fA^{\alpha-\eta \ell_1,\tau}(z,\ell_2),
\end{equation}
so Lemmas~\ref{lemma15oct1}, \ref{lemma27may1}, \ref{lemma16oct1} imply:

\begin{theorem} \label{theorem47}
\begin{enumerate}[(a)]
\item $\fA^{\alpha,\tau}(z,\ell)$ is holomorphic in $\Omega$.
\item For $\ell \ge 0$, $\fA^{\a,\tau}$ is $j$-contractive in $\bbC_+$  and $\det\fA^{\alpha,\tau}(z,\ell)=1$.
\item The boundary values on $\sE$ coincide,
\begin{equation}\label{18jul7}
\fA^{\alpha,\tau}(\xi+i0,\ell)=\fA^{\alpha,\tau}(\xi-i0,\ell), \quad \text{a.e.} \ \xi\in\sE.
\end{equation}
\item $\fA^{\alpha,\tau}(z,\ell)$ is jointly continuous with respect to $\a,\tau,\ell$, for an arbitrary $z \in\Omega$.
\item 
The family is $j$-monotonic with respect to $\ell$, i.e.,
$$
j-\fA^{\alpha,\tau}(z,\ell_2)j\fA^{\alpha,\tau}(z,\ell_2)^*\ge j-\fA^{\alpha,\tau}(z,\ell_1)j\fA^{\alpha,\tau}(z,\ell_1)^*\ge 0
$$
for $\ell_1< \ell_2$. 
\end{enumerate}
\end{theorem}

Now, we prove one of the most important properties. We show that all possible singularities of $\fA^{\a,\tau}(z,\ell)$ on $\sE$ are removable.

\begin{lemma}
For fixed $({\alpha,\tau})\in\pi_1(\Omega)^* \times \bbT$ and $\ell \in \bbR$, the matrix $\fA^{\alpha,\tau}(z,\ell)$ is entire.
\end{lemma}

\begin{proof}
Let  $\sE_n$ be a  subset of $\sE$,   $\sE_n=\sE\cap[b_{n_-},a_{n_+}]$, $b_{n_-}<a_{n_+}$.
Consider an arbitrary rectangle $Q$  whose vertical edges pass throw the gaps $(a_{n_-},b_{n_-})$ and $(a_{n_+},b_{n_+})$ respectively, $\sE_n\subset Q$. 

It is easy to see that $\Omega_{Q}=\Omega \cap Q$ is of Widom type and DCT holds in it. Indeed, if $\check\a$ is a character on $\pi_1(\Omega_Q)$
we can find a character $\a\in\pi_1(\Omega)^*$, so that $\a|_{\pi_1(\Omega_Q)}=\check \a$.  Since $H_\Omega^\infty(\a)$ contains a non-trivial function, this function in its restriction on $\Omega_Q$ provides a non-trivial function in $H_{\Omega_Q}^\infty(\a)$. That is, $\Omega_Q$ is of Widom type. Let $\check \a_n$ be a sequence of characters which converge to the trivial character in $\pi_1(\Omega_Q)^*$. Again, we can find a sequence 
$\a_n$ such that
$$
\a_n|_{\pi_1(\Omega_Q)}=\check \a_n \quad\text{and}\quad \a_n\to 0_{\pi_1(\Omega)^*}.
$$
Fix $z_0\in \Omega_Q$. By the DCT in $\Omega$ we have
$$
\lim_{n\to\infty}\sup\{|w(z_0)|:\ w\in H^\infty_\Omega(\a_n)\}=1.
$$
Moreover
$$
\lim_{n\to\infty}\sup\{|w(z_0)|:\ w\in H^\infty_{\Omega_Q}(\check\a_n)\}=1,
$$
and this is one of characteristic properties of DCT \cite[Theorem, p. 206]{H83}.

We can explicitly write
\begin{equation}\label{11jul99}
\fA^{\alpha,\tau}(z,\ell) =\frac 1{\det \cT_{\alpha}}
 \begin{pmatrix}
K^\alpha & - \bar\tau_* \bar \tau \Phi K_\sharp^{\tilde\alpha} \\
- \tau K_\sharp^\alpha  & \tau_* \Phi_\sharp K^{\tilde\alpha} 
\end{pmatrix}\Lambda_{\Theta(z)-\theta_r}(\ell)\cT_{\a-\eta \ell,\tau }(z).
\end{equation}
Since $e^{\pm i \ell \Theta} \in \cH^\infty_{\Omega_Q}(\pm \eta \ell)$ and $\frac{z-i}{\Phi},\frac{z+i}{\Phi_\sharp} \in \cH^\infty_{\Omega_Q}(-\beta_\Phi)$, \eqref{11jul99} implies $\fA^{\a,\tau} = \frac{\Theta'}{\cW} \cB$ where entries of the matrix $\cB$ are in the set $\cH^1_{\Omega_Q}(\beta_\cW)$. By DCT in $\Omega_Q$, for any $z_0 \in \Omega_Q$,
\begin{align*}
\fA^{\a,\tau}(z_0,\ell) & = \frac 1{2\pi i} \oint_{\partial\Omega_Q} \frac{\fA^{\a,\tau}(z,\ell)}{z-z_0} \,dz
\end{align*}
Due to \eqref{18jul7},
\[
\frac 1{2\pi i} \oint_{\sE_n \cap \partial\Omega_Q} \frac{\fA^{\a,\tau}(z,\ell)}{z-z_0} \,dz = 0,
\]
so for all $z_0 \in \Omega_Q$,
\begin{align*}
\fA^{\a,\tau}(z_0,\ell) & = \frac 1{2\pi i} \oint_{\partial Q} \frac{\fA^{\a,\tau}(z,\ell)}{z-z_0} \,dz
\end{align*}
with $\fA^{\a,\tau}(z,\ell)$ integrable on $\partial Q$. The right-hand side defines an analytic function in $Q$. Thus, all possible singularities of $\fA^{\a,\tau}(z,\ell)$, given by  \eqref{11jul99}, on the subset $\sE_n$ are removable.  Since $\sE_n$ is an arbitrary piece of $\sE$ in the finite part of the plane $\bbC$,  $\fA^{\a,\tau}(z,\ell)$ is entire.
\end{proof}

\begin{remark}
Let $1\le p\le\infty$ and $z_0\in\Omega$. Let $h_p^\a$ be the extremal function of the problem: find
$$
h_p^\a(z_0)=\sup\{ |g(z_0)|: g\in H^p(\a),\ \|g\|\le 1\}.
$$
In fact, the theorem on page 206 in \cite{H83} claims an equivalence of DCT and the condition
\begin{equation}\label{10oct1}
h_p^\a(z_0) \ \text{is continuous on $\pi_1(\Omega)^*$ for every $p$ with}\ 1\le p\le\infty.
\end{equation}
But \eqref{10oct1} is an easy consequence of continuity of $h_\infty^\a(z_0)$ in the vicinity of the origin $0_{\pi_1(\Omega)^*}$. Indeed, for an arbitrary $\a,\b\in
\pi_1(\Omega)^*$ we have
\[
 h_p^\b(z_0)h_\infty^{\a-\b}(z_0)\le h_p^\a(z_0) \quad\text{and}\quad h_p^\a(z_0)h_\infty^{\b-\a}(z_0)\le h_p^\b(z_0).
\]
Passing to the limits as $\a\to\b$ we have
\[
h_p^\b(z_0)\le \liminf_{\a\to\b} h^\a_p(z_0)\le 
\limsup_{\a\to\b} h^\a_p(z_0)\le h_p^\b(z_0),
\]
that is, \eqref{10oct1} holds.
\end{remark}

\begin{theorem}\label{theorem410}
The function $s^{\a,\tau}_+$ is the Schur function corresponding to the family of transfer matrices $\{ \fA^{\a,\tau}(z,\ell) \}_{\ell \in \bbR_+}$.
\end{theorem}

\begin{proof}
For $z \in \bbC_+$,
\begin{align*}
\begin{pmatrix} s_+^{\alpha,\tau}(z) & 1 \end{pmatrix} \fA^{\a,\tau}(z,\ell) 
& \simeq
\begin{pmatrix} 0 & 1 \end{pmatrix} \Lambda_{\Theta-\theta_r}(\ell) \cT_{\a-\eta \ell,\tau} (z)  \\
& \simeq
\begin{pmatrix} 0 & 1 \end{pmatrix} \cT_{\a-\eta \ell,\tau}(z)  \simeq
\begin{pmatrix} s_+^{\a-\eta \ell, \tau}(z) & 1 \end{pmatrix}.
\end{align*}
Since $s_+^{\a-\eta \ell,\tau}(z) \in \overline{\bbD}$, it follows that $s_+^{\a,\tau}(z)$ is in the Weyl disk for every $\ell > 0$.

Simultaneously, we can observe that $\tau$ is an ``integral of motion" for the translation flow in Arov gauge generated by this family, i.e., the flow 
$s_+^{\a,\tau}(z)\mapsto s_+^{\a-\eta \ell,\tau}(z)$.
\end{proof}

Now that we have constructed the $j$-monotonic family $\fA^{\a,\tau}$, we have to invoke general facts about canonical systems in A-gauge  (general proofs in A-gauge are available in \cite{BLY1}). We will need the following:
\begin{remark}[{\cite{BLY1}}] \label{remarkArovGauge1}
Let $\fA(z,\ell)$ be a $j$-monotonic family in A-gauge with $\fA(z,0) = I$ for all $z$.  Then:
\begin{enumerate}
\item 
$\fA$ is the solution of a canonical system in A-gauge \eqref{canonicalArov1}, which we also write in the form
\begin{equation}\label{canonicalArov2}
\fA(z,\ell) j = j + \int_0^\ell \fA(z,l) \left(  i z A(l) - B(l) \right) \,d\mu(l), \qquad A = \begin{pmatrix} 1 & - \bar \fa \\ -\fa & 1 \end{pmatrix}, \quad B = \begin{pmatrix} 0 &  \bar \fa \\ -\fa & 0 \end{pmatrix}.
\end{equation}

\item $\fA_{22}(i,\ell)$ is a decreasing function of $\ell$ and the positive measure $\mu$ is determined by $\mu(\ell) = -\log \fA_{22}(i,\ell)$.  The family is in the Weyl limit point case if and only if $\mu(\ell) \to \infty$ as $\ell \to \infty$. The parameters $\fa$ are determined by 
\[
A(\ell) + B(\ell) =  \begin{pmatrix} 1 & 0 \\ - 2 \fa(\ell) & 1\end{pmatrix} =  - (\fA(i,\ell))^{-1} \partial_\mu \fA(i,\ell) j, \qquad \mu\text{-a.e. }\ell.
\]
Here, like elsewhere, $\partial_\mu$ denotes Radon--Nikodym derivative: in particular, $\fA(i,\ell)$ is a.c. with respect to $\mu$ and $\lvert \fa(\ell) \rvert \le 1$ for $\mu$-a.e. $\ell$.
\item (Ricatti equation) The translation flow on canonical systems can be obtained by a familiar coefficient stripping process or, at the level of the transfer matrices, by considering for $\ell > 0$ the family $\{\fA(z,\ell)^{-1} \fA(z,l+\ell)  \}_{l \ge 0}$. Denoting the corresponding Schur functions by $s_+(z,\ell)$, their behavior is described by the Ricatti equation
\begin{equation}\label{RicattiEquation}
\partial_\mu s_+(z,\ell)  = \begin{pmatrix} s_+(z,\ell) & 1 \end{pmatrix} (-iz A(\ell) + B(\ell)) \begin{pmatrix} 1 \\ s_+(z,\ell) \end{pmatrix}.
\end{equation}
\item 
(continuous Verblunsky parameters as boundary values of Schur functions)
For $\mu$-a.e. $\ell \ge 0$, 
\begin{equation}\label{10jun2}
\lim_{\substack{z \to \infty \\ \arg z \in [\delta,\pi-\delta]}} s_+(z,\ell) =\frac{ \fa(\ell)}{1+\sqrt{1 - \lvert \fa(\ell)\rvert^2}}.
\end{equation}
\item Denoting by $\fc(\ell) \in \overline{\bbD}$ the right-hand side of \eqref{10jun2}, we have the mutually inverse formulas
\begin{equation}\label{9jul2}
\fa(\ell)=\frac{2\fc(\ell)}{1+|\fc(\ell)|^2}\quad\text{and}\quad\fc(\ell)=\frac{\fa(\ell)}{1+\sqrt{1 - \lvert\fa(\ell)\rvert^2}}.
\end{equation}
In particular,
\begin{equation}\label{9jul1}
\sqrt{A(\ell)}=\frac{1}{\sqrt{1+|\fc(\ell)|^2}}\begin{pmatrix}1&-\overline{\fc(\ell)}\\
-\fc(\ell)&1
\end{pmatrix}.
\end{equation} 
\item (Krein--de Branges formula \cite{dB,Krein97,BLY1}) 
The exponential type of the transfer matrix is 
\begin{equation}\label{9jul4}
\limsup_{y\to\infty}\frac{\log\|\fA(iy,\ell)\|}{y} = \int_0^\ell\sqrt{\det A(l)}\,d\mu(l) =  \int_0^\ell\sqrt{1 - \lvert \fa(l) \rvert^2} \,d\mu(l).
\end{equation}
\item (de Branges uniqueness theorem) For any Schur function $s_+:\bbC_+ \to \overline{\bbD}$, there is a half-line canonical system in A-gauge \eqref{canonicalArov1} with Schur function $s_+$, determined uniquely up to  reparametrizations  $\tilde \mu( \ell) = \mu(g(\ell))$,  $\tilde \fa(\ell) = \fa(g(\ell))$ 
with an increasing bijection $g:[0,\infty) \to [0,\infty)$.
\item A reflection of the $\ell$-axis gives the $j$-monotonic family
\[
\tilde{\fA}(z,\ell) = j_1 \fA(z,-\ell) j_1^{-1},  \qquad j_1 = \begin{pmatrix} 0 & 1 \\ 1 & 0 \end{pmatrix}
\]
which is not in A-gauge; it is upper triangular at $z=i$ instead of lower triangular. Its spectral function $s_-$ necessarily obeys $s_-(i) =0$. In fact, that is the only restriction: $s_-(z) = \frac{z-i}{z+i} \overleftarrow{s}_+(z)$ where $\overleftarrow{s}_+$ is the Schur function corresponding to the canonical system in A-gauge with reflected parameters $\overleftarrow{\mu}(\ell) = - \mu(-\ell)$,  
$\overleftarrow{\fa}(\ell) = \overline{\fa(-\ell)}$.
\end{enumerate}
\end{remark}

\begin{theorem}\label{th47}
$\fA^{\a,\tau}(z,\ell)$ solves the canonical system equation \eqref{canonicalArov1} with $\mu = \mu^\alpha$ given by \eqref{mualphadefinition} and $\fa = \fa^{\a,\tau} =  \tau \fa^{\alpha}$ where $\fa^\alpha$ is the Radon--Nikodym derivative 
 given by \eqref{aalphadefinition}.
\end{theorem}

\begin{proof}
We already established that $\fA^{\alpha,\tau}$ is a $j$-monotonic family in A-gauge. By a direct calculation,
\begin{equation}\label{15oct1}
\fA^{\alpha,\tau}_{22}(i,\ell) = \frac{K^{\a - \eta \ell}(i)}{K^{\a}(i)} e^{i\ell (\Theta(i)-\theta_r)}.
\end{equation}
The reproducing kernel depends continuously on the character due to DCT, so $\fA^{\alpha,\tau}_{22}(i,\ell)$ is continuous in $\ell$. Thus, $\fA^{\a,\tau}$ is the solution of a canonical system in A-gauge.
The measure has the distribution function
\[
\mu(\ell) = - \log \fA^{\alpha,\tau}_{22}(i,\ell) = \ell \theta_i - \log \frac{K^{\a - \eta \ell}(i)}{K^{\a}(i)} 
\]
which gives precisely the measure $\mu = \mu^\a$ independent of $\tau$ and given by \eqref{mualphadefinition}. 

By construction, coefficient stripping corresponds to a linear shift in character, so  $s^{\a,\tau}_+(z,\ell) = s_+^{\a - \eta \ell,\tau}(z)$. Thus, applying the Ricatti equation at $z=i$ and integrating gives
\[
s_+^{\a-\eta\ell, \tau}(i) - s_+^{\a, \tau}(i) = 2 \int_0^\ell s_+^{\a - \eta l,\tau}(i) d\mu^\a(l)  - 2 \int_0^\ell a(l) d\mu^\a(l)  
\]
Algebraic manipulations bring this to the form $\int_0^\ell a(l) d\mu^\a(l) = \tau \mu_1^\a((0,\ell])$ with $\mu_1^\a$ defined by  \eqref{mualphadefinition100}. Therefore $\mu_1^\a$ is absolutely continuous w.r.t. $\mu^\a$, $a = \tau \fa^\a$ with $\fa^\a$ given by \eqref{aalphadefinition}, and $\lvert \fa^\a \rvert \le 1$ a.e..
\end{proof}

In particular, this proves Theorem~\ref{theorem11}(a),(b).

\begin{remark}
We have already seen that $\mu^{\a,\tau}=\mu^\a$ is $\tau$-independent, 
 and by  \eqref{mualphadefinition}, we have that in average
\begin{equation}\label{28sept203}
\int_{\pi_1(\Omega)^*}\int_0^\ell d\mu^\a(l)d\a=\theta_i\ell.
\end{equation}
The additional parameter $\tau\in \bbT$ is needed to describe all reflectionless systems, but in many formulas its influence is simple and can be factored out. We will denote the canonical system parameters by
\begin{equation}\label{canonicalArovalphatau}
\fA^{\a,\tau}(z,\ell) j = j + \int_0^\ell \fA^{\a,\tau}(z,l) \left( i z A^{\a,\tau}(l) - B^{\a,\tau}(l) \right) \,d\mu^{\a,\tau}(l).
\end{equation}
Since the coefficient $\fa^{\a,\tau}$ depends of $\tau$ in a trivial way $\fa^{\a,\tau}=\tau \fa^{\a}$, we write
\begin{equation}\label{17jul201}
A^{\a,\tau}=\cU_\tau^{-1}A^\a\cU_\tau,\ B^{\a,\tau}=\cU_\tau^{-1}B^\a\cU_\tau,
\end{equation}
with $A^\a = A^{\a,1}$, $B^\a = B^{\a,1}$, with the diagonal unitary and $j$-unitary $\cU_\tau$ from \eqref{17oct7}.
\end{remark}

\subsection{$M$-type: growth of  transfer matrices  with respect to the Martin function } \label{secmb}

\subsubsection{Growth at $\infty$ of positive harmonic functions on $\Omega$}

 Borichev and Sodin \cite{BS01} proved the following lemma:

\begin{lemma}[Borichev--Sodin] \label{lemmaBorichevSodin}
Let $h$ be a positive harmonic function on $\bbC \setminus \sE$ such that $h(\bar z) = h(z)$. The function can be decomposed as $h(z) = C M(z) + \tilde h(z)$ where $C \ge 0$, $\tilde h$ is a positive harmonic function on $\bbC \setminus \sE$, and
\begin{equation}\label{7jul4}
\lim_{y\to\infty} \frac{\tilde h(iy)}{M(iy)} = 0.
\end{equation}
\end{lemma}

It has the following corollary:

\begin{corollary}
If $f$ is an outer function on $\Omega$, $\lvert f \rvert \ge 1$ and $\lvert f(\bar z) \rvert = \lvert f(z) \rvert$, then
\[
\lim_{y\to\infty} \frac{ \log \lvert f(iy) \rvert}{M(iy)} = 0.
\]
\end{corollary}

\begin{proof}
By Lemma~\ref{lemmaBorichevSodin}, there exists $C \ge 0$ such that
\[
C = \lim_{y\to\infty} \frac{ \log \lvert f(iy) \rvert}{M(iy)}
\]
and $\log \lvert f \rvert \ge C M$. This implies that
\begin{equation}\label{16jul98}
\frac 1{\lvert f(z) \rvert} \le \lvert e^{iC\Theta(z)} \rvert.
\end{equation}
If $C>0$, $e^{iC\Theta(z)}$ is a singular inner function and $1/f$ an outer function, so \eqref{16jul98} would give a contradiction.
\end{proof}

We will apply this corollary in order to compute the $M$-exponential type of the transfer matrix  $\fA^\alpha(z,x)$ of a reflectionless canonical system, see Lemma \ref{23l410}. In this section we provide a systematic approach to the Borichev-Sodin kind propositions,
see Theorem \ref{th13mar1} below.

We recall briefly the construction of the Martin boundary \cite{Mar41}. Consider for $w\in \Omega$ the Green function with a pole at $w$ normalized at some $z_* \in \Omega$,
\begin{equation}\label{22jun5}
M(z,w) = \frac{G(z,w)}{G(z_*,w)}.
\end{equation}
The Martin boundary $\partial^M \Omega$ is the collection of limits of sequences $M(z,w_n)$ for sequences of $w_n\in \Omega$ which eventually leave every compact subset of $\Omega$. The limits are considered in the sense of uniform convergence on compact subsets of $\Omega$; in particular, they are positive harmonic functions on $\Omega$. The construction of the Martin compactification extends the definition of $M(z,w)$ to $w \in \hat\Omega = \Omega \cup \partial^M \Omega$.

Let $\partial_1^M \Omega$ denote the subset of the Martin boundary consisting of minimal harmonic functions. Since $\Omega$ is a Denjoy domain, $\partial_1^M \Omega$ contains $1$ or $2$ points for each point of $\sE \cup \{\infty\}$ \cite{Ben80}. We denote that correspondence by $\cP : \partial_1^M \Omega \to \sE \cup \{\infty\}$.  Every positive harmonic function on $\Omega$ has a unique representation
\begin{equation}\label{22jun1}
h(z) = \int_{\partial_1^M \Omega} M(z,w) \,d\sigma_h(w)
\end{equation}
with a unique finite measure $\sigma_h$ on $\partial_1^M\Omega$.

\begin{lemma}
For any $\delta > 0$,
\begin{equation}\label{28sep1}
\sup_{\substack{\lvert z \rvert \ge 1 \\ \arg z \in [\delta,\pi-\delta]}} \sup_{w\in \partial^M_1 \Omega} \frac{M(z,w)}{M(z)} < \infty.
\end{equation}
\end{lemma}

\begin{proof}
Recall that all Martin functions are normalized at the same internal point $z_* \in \Omega$; however, by the Harnack principle, the desired conclusion \eqref{28sep1} is independent of the choice of $z_*$. For the proof, let us fix $z_* = 10i$, so that
\[
M(z,w) = \frac{G(z,w)}{G(10i,w)}, \qquad w \in \Omega.
\]
The key is the use of the boundary Harnack principle for Denjoy domains \cite{Anc79,GarSjo09}.  Let us use the notation $f \lesssim g$ if $f \le C g$ for some universal constant $C$ and $f \simeq g$ if $f \lesssim g$ and $g \lesssim f$. For instance, for any $\delta > 0$, by a Harnack chain with constant complex modulus of size depending only on $\delta$,
\[
\frac{M(z,w)}{M(z)} \simeq \frac{M(i\Im z, w)}{M(i\Im z)}
\]
so to prove the nontangential bound \eqref{28sep1}, it suffices to prove the normal bound for $z=iy$, $y\ge 1$. It will also be more convenient to apply an inversion and prove
\begin{equation}\label{3feb1}
\sup_{y \in (0,1]} \sup_{w \in \partial_1^M \Omega} \frac{M(iy,w)}{M(iy,0)} < \infty,
\end{equation}
where $0 \in \partial_1^M \Omega$.

By \cite[Thm. 3]{GarSjo09}, Denjoy domains obey the following boundary Harnack principle: if $r > 0$, for all $x, y\in \bbC_+$ such that $\lvert x \rvert < r$, $\lvert y \rvert < r$, and $t \ge 10 r$,
\[
\frac{G(x,it)}{G(x,2r i)} \simeq \frac{G(y,it)}{G(y,2r  i)}.
\]
We apply this to $r \in (0,1]$ and $t=10$ to conclude
\[
\frac{G(10i,x)}{G(2r  i,x)} \simeq \frac{G(10i,y)}{G(2r i,y)}, \qquad \forall x, y\in \bbC_+ : \lvert x \rvert < r, \lvert y \rvert < r.
\]
Letting $x\to r i$ and letting $y\to 0$ gives
\begin{equation}\label{4feb2}
M(2r  i,r  i) \simeq M(2r i,0), \quad r \in (0,1].
\end{equation}
For $x \in \partial_1^M \Omega$, by the Harnack principle applied in the domain $\Omega \setminus \{r i\}$,
\[
\frac{M(z,r i)}{M(z,x)} \simeq \frac{M(2r i,r i)}{M(2r i,x)}, \quad \lvert z - r i \rvert =  \frac r2.
\]
By the maximum principle, 
\[
M(z,r i) \lesssim M(z,x) \frac{M(2r i, r i)}{M(2r i,x)}, \quad \lvert z - r i \rvert \ge \frac r2
\]
since on the domain $\{z \in \hat\Omega \mid \lvert z - r i \rvert \ge r/2\}$, the left-hand side achieves its maximum on the circle $\lvert z - r i \rvert \ge r/2$ and the right-hand side achieves its minimum there. In particular, using $z=10i$, we conclude $M(2r i,x) \lesssim M(2r i,r i)$ for $r \in (0,1]$. Combining this with \eqref{4feb2} gives $M(2r i, x) \lesssim M(2r i,0)$ for $r \in (0,1]$, $x \in \partial_1^M \Omega$, which implies \eqref{3feb1}.
\end{proof}

\begin{lemma} \label{lemmapointwiseMartin1}
For any $w \in \partial^M_1\Omega$ with $\pi(w) \neq \infty$ and any $\delta > 0$,
\[
\lim_{\substack{z\to\infty \\ \arg z \in [\delta,\pi-\delta]}} \frac{M(z,w)}{M(z)} = 0.
\]
\end{lemma}

\begin{proof}
Since $\pi(w) \neq \infty$, the Martin function $M(z,w)$  is subharmonic in some neighborhood of $\infty$ and therefore bounded there. Meanwhile, the symmetric Martin function has a Riesz representation
\[
M(z) = M(\xi_*) + \int \log \left\lvert \frac{x-z}{x-\xi_*} \right\rvert \,d\vartheta(x)
\]
with $\xi_* \in \bbR \setminus\sE$ and $\supp \vartheta = \sE$. From this representation, by monotone convergence, $M(iy) \to \infty$ as $y\to \infty$. By the Harnack principle applied with Harnack chains of fixed size along an arc with fixed $\lvert z \rvert$, this implies that $M(z) \to \infty$ as $z \to\infty$ with $\arg z \in [\delta,\pi-\delta]$.
\end{proof}

It is a general fact about Denjoy domains that all elements of $\cP^{-1}(\{\infty\})$ can be obtained as limits of $M(z,w_n)$ for sequences $w_n = i y_n \to \pm i \infty$. Thus, let us denote the corresponding elements of the Martin boundary as $\pm i\infty$. In this notation, the Akhiezer--Levin case is precisely the case $+ i \infty  \neq - i \infty$.

From now on, let us assume that the normalization point $z_*$ from \eqref{22jun5} is in $\bbR \setminus \sE$. Then the symmetric Martin function is
\[
M(z) = \frac{M(z,+i\infty) + M(z,-i\infty)}2.
\]
Moreover, in the Akhiezer--Levin case, $M(z,-i\infty) = o(M(z,+i\infty))$ as $\lvert z \rvert \to \infty$ with $\arg z \in [\delta,\pi-\delta]$, so Lemma~\ref{lemmapointwiseMartin1} implies that (in both cases):

\begin{lemma} \label{lemmapointwiseMartin2}
For any $w \in \partial^M_1\Omega$ with $w \neq +i\infty$ and any $\delta > 0$,
\[
\lim_{\substack{z\to\infty \\ \arg z \in [\delta,\pi-\delta]}} \frac{M(z,w)}{M(z,+i\infty)} = 0.
\]
\end{lemma}

\begin{theorem}\label{th13mar1}
For any positive harmonic function $h$ on $\Omega$ and any $\delta > 0$,
\[
\lim_{\substack{\lvert z\rvert \to\infty \\ \arg (\pm z) \in [\delta,\pi-\delta]}} \frac{h(z)}{M(z,\pm i\infty)} = \sigma_h(\{\pm i\infty\}) = \inf_{z\in \Omega} \frac{h(z)}{M(z,\pm i\infty)}.
\]
\end{theorem}

\begin{proof}
The second equality is general Martin theory \cite[Chapter 9]{ArmGar01}. For the first, use \eqref{22jun1} to write
\[
\lim_{\substack{\lvert z\rvert \to\infty \\ \arg(\pm z) \in [\delta,\pi-\delta]}} \frac{h(z)}{M(z)} = \lim_{\substack{\lvert z\rvert \to\infty \\ \arg (\pm z) \in [\delta,\pi-\delta]}}\int_{\partial_1^M\Omega} \frac{M(z,w)}{M(z)} d\sigma_h(w).
\]
By \eqref{28sep1} and since $\sigma_h$ is a finite measure, the dominated convergence theorem can be applied with a constant majorant. Thus, by Lemma~\ref{lemmapointwiseMartin2}, this gives
\[
\lim_{\substack{\lvert z\rvert \to\infty \\ \arg (\pm z) \in [\delta,\pi-\delta]}} \frac{h(z)}{M(z)} =  \int_{\partial_1^M\Omega} \chi_{\{\pm i\infty\}}(w) d\sigma_h(w) = \sigma_h(\{\pm i\infty\}). \qedhere
\]
\end{proof}

This has a corollary for symmetric functions $h$ and the symmetric Martin function. The proof is immediate, by considering separately the A--L case and the non-A--L case:

\begin{corollary}\label{cor13mar1}
For any positive harmonic function $h$ on $\Omega$ which obeys $\overline{h(\bar z)} = h(z)$, 
\[
\lim_{\substack{\lvert z\rvert \to\infty \\ \arg (\pm z) \in [\delta,\pi-\delta]}} \frac{h(z)}{M(z)} =\inf_{z\in \Omega} \frac{h(z)}{M(z)}, \quad \forall \delta > 0.
\]
\end{corollary}

\subsubsection{Growth at $\infty$ of the transfer matrices}

\begin{theorem}\label{23l410}
For all $(\alpha,\tau)$, all $\ell > 0$, and all $\delta > 0$,
\begin{equation}\label{16jul99}
\lim_{\substack{z\to\infty \\ \arg z \in [\delta,\pi-\delta]}} \frac{\log \lVert \fA^{\a,\tau}(z,\ell) \rVert}{M(z)} = \ell.
\end{equation}
\end{theorem}

\begin{proof}
We use the representation \eqref{9mar203}.
Since $\det \Pi_{\a,\tau} = 1$ for all $(\a,\tau)$, it follows that  $\lVert \Pi_{\a,\tau} \rVert \ge 1$. 
Each entry $(\Pi_{\a,\tau })_{ij}$ can be majorized by an outer function $a_{ij}$ with $\lvert a_{ij} \rvert \ge 1$: it suffices to define $a_{ij}$ by its boundary values on $\partial\Omega$, $\log \lvert a_{ij} \rvert = \log_+ \lvert (\Pi_{\alpha,\tau})_{ij} \rvert$. Then consider the outer function $f$ defined by its boundary values $\log \lvert f \rvert = \max \{ \log \lvert a_{ij} \rvert \mid i,j \in \{1,2\} \}$ on $\partial\Omega$ (well defined because the pointwise maximum of integrable functions is integrable). Then $\lvert a_{ij}(z) \rvert \le \lvert f(z) \rvert$ for all $z\in \Omega$ so
\[
0 \le \log \lVert \Pi_{\a} \rVert \le \log \sum_{i,j=1}^2 \lvert a_{ij} \rvert \le \log (4 \lvert f \rvert)
\]
and since $f$ is outer and $\lvert f(\bar z) \rvert = \lvert f(z) \rvert$,
\[
\lim_{\substack{z\to\infty \\ \arg z \in [\delta,\pi-\delta]}} \frac{\log \lvert f(z) \rvert}{M(z)} = 0.
\]
Thus
$$
\lim_{\substack{z\to\infty \\ \arg z \in [\delta,\pi-\delta]}} \frac{\log \lVert \Pi_{\a,\tau}(z) \rVert}{M(z)} = 0.
$$
for all $(\a,\tau) \in\pi_1(\Omega)^*\times \bbT$.

Since $\Pi_{\a,\tau}$ is a $2\times 2$ matrix and $\det \Pi_{\a,\tau} = 1$, $\lVert \Pi_{\a,\tau}^{-1} \rVert = \lVert \Pi_{\a,\tau} \rVert$.
By submultiplicativity of operator norm,
\[
\lVert \fA^{\a,\tau} \rVert \le  \lVert \Pi_{\a,\tau}^{-1} \rVert \lVert \Lambda_{\Theta(z)-\theta_r}(\ell) \rVert  \lVert \Pi_{\alpha-\eta \ell,\tau }\rVert
\]
and
\[
 \lVert \Lambda_{\Theta(z)-\theta_r}(\ell) \rVert \le \lVert \Pi_{\a,\tau} \rVert \lVert \fA^{\a,\tau} \rVert  \lVert \Pi_{\alpha-\eta \ell,\tau}^{-1} \rVert
\]
so \eqref{16jul99} follows from
\[
\lim_{\substack{z\to\infty \\ \arg z \in [\delta,\pi-\delta]}} \frac{\log \lVert \Lambda_{\Theta(z)-\theta_r}(\ell) \rVert}{M(z)} = \ell. \qedhere
\]
\end{proof}

\begin{proof}[Proof of Theorem~\ref{theorem110}] 
By Theorem~\ref{theoremhomeomorphisms}, all reflectionless canonical systems with spectrum $\sE$ correspond to Schur functions $s_+^{\a,\tau}$, and by Theorem~\ref{theorem410} and de Branges' uniqueness theorem, they all correspond to $j$-monotonic families of the form $\fA^{\a,\tau}(z,\ell)$, up to reparametrization. By  \eqref{transfermatrixMartinexptype1} our construction obeys \eqref{16jul99}, so $\mu = \mu^\a$ and $\fa = \tau \fa^\a$.
\end{proof}

\section{Fourier transform} \label{sectionFourier}

In this section, we construct unitary Fourier transforms. The basic strategy is standard: we construct norm-preserving maps on dense sets and show that their continuous extensions are unitary (compare the construction in Appendix). 
 We start by working on the spaces $\cK_{e^{i\ell\Theta}}(\a)$ and compute inner products in order to obtain the norm-preserving properties. Eventually, the space $\cK_{e^{i\ell\Theta}}(\a)$ will correspond to the interval $[0,\ell]$ on the target Hilbert space, so working on this space is related to working on compactly supported functions.

\subsection{Reproducing kernels on $\cK_{e^{i\ell\Theta}}(\a)$ and involutions}

\begin{lemma}
If $\Delta$ is inner, $\Delta \overline{\Delta(z_0)} k^{\a-\b_\Delta}_{z_0}$ is the reproducing kernel for $\Delta \cH^2(\a - \b_\Delta)$.
\end{lemma}

\begin{proof}
For any $f \in \Delta \cH^2(\a-\b_\Delta)$,
\[
\left\langle f, \Delta \overline{\Delta(z_0)} k_{z_0}^{\a-\b_\Delta} \right\rangle =  \left\langle \frac f\Delta, \overline{\Delta(z_0)} k_{z_0}^{\a-\b_\Delta} \right\rangle =  \Delta(z_0) \left(\frac f\Delta\right)(z_0) = f(z_0). \qedhere 
\]
\end{proof}

\begin{lemma} \label{lemma30jun1}
If $\Delta$ is inner, the function $k_{z_0}^\a - \Delta \overline{\Delta(z_0)} k_{z_0}^{\a - \b_\Delta}$ is the reproducing kernel for $\cK_\Delta(\a)$.
\end{lemma}

\begin{proof}
The function $k_{z_0}^\a - \Delta \overline{\Delta(z_0)} k_{z_0}^{\a - \b_\Delta}$ is obviously in $\cH^2(\a)$. Moreover, for any $g \in \Delta\cH^2(\a-\b_\Delta)$, 
\[
\langle g, k_{z_0}^\a - \Delta \overline{\Delta(z_0)} k_{z_0}^{\a - \b_\Delta}\rangle = g(z_0) - \left\langle \frac{g}{\Delta}, \overline{\Delta(z_0)} k_{z_0}^{\a - \b_\Delta} \right\rangle = g(z_0) - \Delta(z_0) \left( \frac{g}{\Delta} \right)(z_0) = 0,
\]
so $k_{z_0}^\a - \Delta \overline{\Delta(z_0)} k_{z_0}^{\a - \b_\Delta} \in \cK_\Delta(\a)$. Finally, for any $f \in \cK_\Delta(\a)$,
\[
\langle f, k_{z_0}^\a - \Delta \overline{\Delta(z_0)} k_{z_0}^{\a - \b_\Delta} \rangle = \langle f, k_{z_0}^\a \rangle - \langle f, \Delta \overline{\Delta(z_0)} k_{z_0}^{\a - \b_\Delta} \rangle = f(z_0) - 0. \qedhere
\]
\end{proof}

Note, in particular, that evaluating this reproducing kernel at $z_0$ gives:

\begin{corollary} \label{repkernelinequality1}
If $\Delta$ is inner then $k^\a_{z_0}(z_0) > \lvert \Delta(z_0)\rvert^2 k^{\a-\b_\Delta}_{z_0}(z_0)$.
\end{corollary}

\begin{lemma} \label{lemma30jun2}
If $f \in \cK_\Delta(\a)$, then 
\[
f_\flat = \frac{h}{\Delta_\sharp \cW}
\]
for some $h \in \cK_{\Delta_\sharp}(\b_\Delta+ \b_\cW - \a)$.
\end{lemma}

\begin{proof}
By Lemma~\ref{lemma29jun1}, $f_* = \frac{g}{\Delta \cW}$ for some $g \in \cK_\Delta(\b_\Delta+ \b_\cW - \a)$. Applying the involution $(\dots)_\sharp$ and using $\cW_\sharp = \cW$ gives $f_\flat = \frac{g_\sharp}{\Delta_\sharp \cW}$. Note that  $g \in \cK_\Delta(\b_\Delta+ \b_\cW - \a)$ implies $h = g_\sharp \in \cK_{\Delta_\sharp}(\b_\Delta+ \b_\cW - \a)$.
\end{proof}

In particular, we will apply this to $\Delta = e^{i\ell \Theta}$. By Lemma~\ref{lemma30jun1}, the function
\[
V^{\a}_{z_0,\ell} = k_{z_0}^\a - e^{i\ell (\Theta - \overline{\Theta(z_0)})} k_{z_0}^{\a - \eta \ell}
\]
is the reproducing kernel in $\cK_{e^{i\ell \Theta}}(\a)$. We also define
\begin{align}
\tilde V^\a_{z_0,\ell} & = e^{i\ell (\Theta + \overline{\Theta(z_0)})} \cW (V^{\eta \ell + \b_\cW  - \a}_{z_0,\ell})_\flat  \label{1jul1} \\
& = \cW \left( e^{i\ell (\Theta + \overline{\Theta(z_0)})} (k_{z_0}^{\eta\ell+\b_\cW-\a})_\flat - (k_{z_0}^{\b_\cW - \a} )_\flat \right). \label{1jul1b}
\end{align}
This is a kind of dual reproducing kernel in $\cK_{e^{i\ell\Theta}}(\a)$:

\begin{lemma}\label{lemma1jul1}
The vector $\tilde V^\a_{z_0,\ell}$ is an element of $\cK_{e^{i\ell\Theta}}(\a)$ and
\begin{equation}\label{1jul4}
\cW(z_0) f_\flat(z_0) = \langle f, \tilde V^\a_{z_0,\ell} \rangle, \qquad \forall f \in \cK_{e^{i\ell\Theta}}(\a).
\end{equation}
\end{lemma}

\begin{proof}
By Lemma~\ref{lemma30jun2} with $\Delta = \Delta_\sharp = e^{i\ell\Theta}$,  $\tilde V^\a_{z_0,\ell} \in \cK_{e^{i\ell\Theta}}(\a)$. Moreover, for any $f \in \cK_{e^{i\ell\Theta}}(\a)$,
\[
\langle f , e^{i\ell\Theta} \cW  (V_{z_0,\ell}^{\eta\ell+\b_\cW-\a})_\flat \rangle = \langle e^{i\ell\Theta} \cW f_\flat , V_{z_0,\ell}^{\eta\ell+\b_\cW-\a} \rangle = e^{i\ell\Theta(z_0)} \cW(z_0) f_\flat(z_0)
\]
by the reproducing kernel property of $V_{z_0,\ell}^{\eta\ell + \b_\cW -\a}$. Dividing by $e^{i\ell\Theta(z_0)}$ gives \eqref{1jul4}. 
\end{proof}

We now compute inner products of reproducing kernels and dual reproducing kernels.

\begin{lemma} \label{lemma2jul5}
For any $z_1, z_2 \in \Omega$ and $\ell_1, \ell_2 > 0$ and with $\ell := \min(\ell_1,\ell_2)$,
\begin{align}
\left\langle V^\a_{z_1,\ell_1}, V^\a_{z_2, \ell_2} \right\rangle & = k^\a(z_2,z_1) - e^{i\ell (\Theta(z_2) - \overline{\Theta(z_1)})} k^{\a-\eta \ell} (z_2,z_1)  \label{30juna}  \\
\left\langle \tilde V^\a_{z_1,\ell_1}, V^\a_{z_2, \ell_2} \right\rangle & =  \cW(z_2) \left( e^{i\ell (\Theta(z_2)+\overline{\Theta(z_1)})}  (k_{z_1}^{\eta\ell +\b_\cW-\a})_\flat(z_2) -  (k_{z_1}^{\b_\cW-\a})_\flat (z_2) \right)  \label{30junb} \\
 \left\langle \tilde V^\a_{z_1,\ell_1}, \tilde V^\a_{z_2,\ell_2}\right\rangle & =  e^{-i\ell(\Theta(z_2) -\overline{\Theta(z_1)})} k_{z_1}^{\eta \ell + \b_\cW - \a} (z_2) -  k_{z_1}^{\b_\cW - \a} (z_2) \label{30junc}
\end{align}
\end{lemma}

\begin{proof}
To prove \eqref{30juna}, assume without loss of generality that $\ell_1 \le \ell_2$; the other case reduces to this by complex conjugation. If $\ell_1 \le \ell_2$, both functions are elements of $\cK_{e^{i \ell_2 \Theta}}(\a)$ and $V^\a_{z_2, \ell_2}$ is a reproducing kernel so
\[
\langle V^\a_{z_1,\ell_1}, V^\a_{z_2, \ell_2} \rangle = V^{\a}_{z_1,\ell_1}(z_2).
\]
Evaluating this by definition gives \eqref{30juna}.

To prove \eqref{30junb} if $\ell_1 \le \ell_2$, use $\tilde V^\a_{z_1,\ell_1} \in \cK_{e^{i\ell_1\Theta}}(\a) \subset \cK_{e^{i\ell_2\Theta}}(\a)$ to compute
\[
\langle \tilde V^\a_{z_1,\ell_1}, V^\a_{z_2, \ell_2} \rangle = \tilde V^\a_{z_1,\ell_1}(z_2).
\]
Evaluating this by \eqref{1jul1b} gives \eqref{30junb} for the case $\ell_1 \le \ell_2$. If $\ell_1 \ge \ell_2$, by a direct calculation using \eqref{1jul1b},
\begin{align*}
\tilde V_{z_1,\ell_1}^{\a} - \tilde V_{z_1,\ell_2}^{\a} = e^{i\ell_2(\Theta+\overline{\Theta(z_1)})} \tilde V^{\a-\eta\ell_2}_{z_1,\ell_1-\ell_2} \in e^{i\ell_2\Theta} \cH^2(\a-\eta\ell_2),
\end{align*}
so this vector is orthogonal to $V_{z_2,\ell_2}^\a$. Thus,
\[
\langle \tilde V^\a_{z_1,\ell_1} ,
V_{z_2,\ell_2}^\a
\rangle = \langle  \tilde V^\a_{z_1,\ell_2}, V_{z_2,\ell_2}^\a
\rangle,
\]
so the calculation for $\ell_1 \ge \ell_2$ reduces to the case $\ell_1 = \ell_2$ computed above.

To prove \eqref{30junc}, assume without loss of generality that $\ell_1 \le \ell_2$; the other case reduces to this by complex conjugation. Using Lemma~\ref{lemma1jul1},
\[
\langle  \tilde V^\a_{z_1,\ell_1},  \tilde V^\a_{z_2,\ell_2} \rangle =  \cW(z_2) ( \tilde V^\a_{z_1,\ell_1})_\flat(z_2)
\]
and evaluating by \eqref{1jul1b} completes the proof.
\end{proof}

\subsection{Fourier transform (general case)}
Our goal is to prove Theorem~\ref{theorem12}. We will begin constructing the Fourier transform $\cF^{\a}$ by assigning how it maps certain functions and extending by linearity and continuity.  
We assume that $\theta_r=0$ (this is only for notational convenience, see Remark~\ref{remarkthetar0nonessential}). 

We denote vector functions
\begin{align}
f_{\alpha}(z) & = \begin{pmatrix} 0 & 1 \end{pmatrix} \cT_{\a}(z)  = \begin{pmatrix} K_\sharp^\a(z) & K^\a(z) \end{pmatrix} \label{7jun1} \\
g_{\a}(z) & = \begin{pmatrix} 1 & 0 \end{pmatrix} \cT_\a(z)  = \begin{pmatrix} \tau_* \Phi_\sharp K^{\tilde\a}& \bar\tau_* \Phi K_\sharp^{\tilde\a} \end{pmatrix}
\end{align}
Let us point out that
\begin{align}
f_{\alpha}(z) \fA^{\alpha}(z,\ell) & = f_{\alpha-\eta \ell}(z) e^{i\ell \Theta(z) }, \label{18jun1} \\
g_{\alpha}(z) \fA^{\alpha}(z,\ell) & = g_{\alpha-\eta \ell}(z) e^{-i\ell\Theta(z) }, \label{18jun1v}
\end{align}
so $f_{\alpha-\eta \ell}(z) e^{i\ell \Theta(z)}$ and $g_{\alpha-\eta \ell}(z) e^{-i\ell\Theta(z)}$ are Weyl solutions at $+\infty$ and $-\infty$ for the canonical system.

We begin constructing the Fourier transform $\cF^{\a}$ by prescribing that it maps, for $z_0 \in \Omega$ and $\ell > 0$,
\begin{align}
 \cF^{\a} : & \sqrt{A^{\a}(l)} (f_{\alpha- \eta l}(z_0))^* e^{-i\overline{ \Theta(z_0)} l } \chi_{[0,\ell]}(l) \mapsto \sqrt{2}  \frac{\Phi_\sharp(z_0)}{z_0 +i} V^{\a-\b_\Phi}_{z_0,\ell} \label{29juna} \\
\cF^{\a}: & \sqrt{A^{\a}(l)}
(g_{\a-\eta l}(z_0))^*
e^{i\overline{\Theta(z_0)} l}\chi_{[0,\ell]}(l)
\mapsto  \sqrt{2} \frac{\Phi_\sharp(z_0)}{z_0 +i} \tilde V^{\a-\b_\Phi}_{z_0,\ell} \label{29junb}
\end{align}
The next lemma ensures that this preserves inner products:

\begin{lemma}\label{lemma2jul1}
For all $z, z_0 \in \Omega$ and $\ell > 0$,
\begin{equation}
\begin{split}
&  \int_0^\ell \Lambda_{\Theta(z)}(l) \cT_{\alpha - \eta l}(z) 
A^{\a}(l) 
 \cT_{\alpha - \eta l}(z_0)^* \Lambda_{\Theta(z_0)}(l)^* \,d\mu^\alpha(l)  \\
&  = 2  \frac{\Phi_\sharp(z)}{z + i} \overline{\left( \frac{\Phi_\sharp(z_0)}{z_0 + i} \right) } 
\begin{pmatrix}
 \langle  \tilde V^{\a-\b_\Phi}_{z_0,\ell}, \tilde V^{\a-\b_\Phi}_{z,\ell} \rangle & 
  \langle  V^{\a-\b_\Phi}_{z_0,\ell}, \tilde V^{\a-\b_\Phi}_{z,\ell} \rangle  \\
 \langle  \tilde V^{\a-\b_\Phi}_{z_0,\ell}, V^{\a-\b_\Phi}_{z,\ell} \rangle &
  \langle  V^{\a-\b_\Phi}_{z_0,\ell}, V^{\a-\b_\Phi}_{z,\ell} \rangle
\end{pmatrix}
\end{split}
\end{equation}
\end{lemma}

\begin{proof}
By the canonical system equation,
\[
\pd_\mu\fA^{\a}(z,\a) j=\fA^{\a}(z,\a)(izA^{\a} -B^{\a}).
\]
Since $A^{\a}$ is self-adjoint and $B^{\a}$ anti-self-adjoint, this implies
\[
j \pd_\mu\fA^{\a}(z_0,\a)^* =(i\bar z_0A^{\a}+B^{\a}) \fA^{\a}(z_0,\a)^*.
\]
Computing $\pd_\mu( \fA^{\a}(z,\ell) j \fA^{\a}(z_0,\ell)^*)$ by the product rule and integrating gives
\[
 \int_0^\ell  \fA^{\a}(z,l)
A^{\alpha}(l)
\fA^{\a}(z_0,l)^*\,d\mu^\alpha(l) = i \frac{j - \fA^{\a}(z,\ell) j \fA^{\a}(z_0,\ell)^*}{z-\bar z_0}.
\]
Multiplying by $\cT_{\a}(z)$ on the left and $\cT_{\a}(z_0)^*$ on the right and using \eqref{9mar202} gives
\begin{equation}\label{18jun23}
\begin{split}
&  \int_0^\ell \Lambda_{\Theta(z)}(l) \cT_{\alpha - \eta l}(z)
A^{\a}(l) 
 \cT_{\alpha - \eta l}(z_0)^* \Lambda_{\Theta(z_0)}(l)^* \,d\mu^\alpha(l)  \\
&  = i \frac{ \cT_\a(z) j \cT_\a(z_0)^* - \Lambda_{\Theta(z)}(\ell) \cT_{\a-\eta \ell}(z) j \cT_{\a-\eta\ell}(z_0)^* \Lambda_{\Theta(z_0)}(\ell)^*}{z-\bar z_0}.
\end{split}
\end{equation}
Using Lemma~\ref{lemma29jun3} twice for the right-hand side, and comparing entries with the inner products computed in Lemma~\ref{lemma2jul5} completes the proof.
\end{proof}

\begin{lemma}
For any $L > 0$, $\cF^{\a}$ extends by linearity and continuity to a unitary operator
\begin{equation}\label{3jul1}
 \overline{ \sqrt{A^{\a}} L^2([0,L],\bbC^2,d\mu^\alpha) } \to \cK_{e^{i L \Theta}}(\a-\b_\Phi).
\end{equation}
\end{lemma}

\begin{proof}
We begin by proving that the left sides of \eqref{29juna}, \eqref{29junb} with $\ell \in [0,L]$ have a dense span in
 $ \overline{ \sqrt{A^{\a}} L^2([0,L],\bbC^2,d\mu^\alpha) } $. Namely, let $\hat f$ be in the orthogonal complement of the span. Then for all $z \in \Omega$ and $\ell \in [0,L]$,
\begin{align*}
 \int_0^\ell e^{i\Theta(z) l } f_{\alpha- \eta l}(z) \sqrt{A^{\a}(l)} \hat f(l)^* d\mu^\a(l) & = 0, \\
 \int_0^\ell e^{-i\Theta(z) l } g_{\alpha- \eta l}(z) \sqrt{A^{\a}(l)} \hat f(l)^* d\mu^\a(l) & = 0.
 \end{align*}
Combining these equations in matrix form gives
\[
 \int_0^\ell \Lambda_{\Theta(z)}(l) \cT_{\alpha - \eta l}(z)  \sqrt{A^{\a}(l)} \hat f(l)^* d\mu^\a(l) = 0.
\]
Since $\ell \in [0,L]$ is arbitrary, this implies that for $\mu^\a$-a.e. $l \in [0,L]$,
\[
\Lambda_{\Theta(z)}(z) \cT_{\alpha-  \eta l}(z) \sqrt{A^{\a}(l)} \hat f(l)^* = 0
\]
and therefore $ \sqrt{A^{\a}} \hat f^* = 0$ $\mu^\a$-a.e.. 
Multiplying by arbitrary $\hat g \in L^2(\bbR,\bbC^2,d\mu^\a)$,
\[
 \int \hat g \sqrt{A^\a} \hat f^* \,d\mu = 0
\]
so $\hat f$ corresponds to the trivial functional on the Hilbert space $ \overline{ \sqrt{A^{\a}} L^2([0,L],\bbC^2,d\mu^\alpha) }$. Therefore, $\hat f = 0$ in $ \overline{ \sqrt{A^{\a}} L^2( [0,L],\bbC^2,d\mu^\alpha) }$.

Right sides of \eqref{29juna},  \eqref{29junb} are elements of $\cK_{e^{i L \Theta}}(\a-\b_\Phi)$. Moreover, since $V^{\a-\b_\Phi}_{z_0,L}$ are reproducing kernels of $\cK_{e^{i L \Theta}}(\a-\b_\Phi)$, orthogonality to all reproducing kernels implies that the function is trivial. Thus, the right sides of \eqref{29juna},  \eqref{29junb} have a dense span in $\cK_{e^{i L \Theta}}(\a-\b_\Phi)$.

By Lemma~\ref{lemma2jul1}, the map $\cF^{\a}$ preserves inner products between vectors in  \eqref{29juna},  \eqref{29junb}. By linearity and continuity, it extends uniquely to a unitary operator \eqref{3jul1}.
\end{proof}

\begin{lemma}\label{lemma17jul20}
$\cF^{\a}$ extends by linearity and continuity to a unitary operator
\begin{equation}\label{5jul1}
\cF^{\a} :   \overline{ \sqrt{A^{\a}} L^2([0,\infty),\bbC^2,d\mu^\a) } \to \cH^2(\a-\b_\Phi).
\end{equation}
This operator can be represented in the form \eqref{Falphadefinition} with $\ell=0$.
\end{lemma}

\begin{proof}
The union $\cup_{\ell > 0} \cK_{e^{i\ell\Theta}}(\a-\b_\Phi)$ is a dense subset of $\cH^2(\a-\b_\Phi)$, since
\begin{equation}\label{17jul205}
\bigcap_{\ell > 0} \left( \cH^2(\a-\b_\Phi) \ominus \cK_{e^{i\ell \Theta}}(\a - \b_\Phi) \right) = \bigcap_{\ell>0} e^{i\ell\Theta} \cH^2(\a - \b_\Phi - \ell\eta) = \{0\}.
\end{equation}
Thus, $\cF^{\a}$ extends by continuity to a unitary operator \eqref{5jul1}.
\end{proof}

\begin{remark}
By continuity, taking $\ell \to \infty$ on both sides of \eqref{29juna}, \eqref{29junb} shows that
\[
 \cF^{\a} :  \sqrt{A^{\a}(l)} (f_{\alpha- \eta l}(z_0))^* e^{-i\overline{ \Theta(z_0)} l } \chi_{[0,\infty)}(l) \mapsto \sqrt{2}  \frac{\Phi_\sharp(z_0)}{z_0 +i} k^{\a-\b_\Phi}_{z_0}.
 \]
Denote for $z_0 \in \Omega$ the functions
\begin{equation}\label{17jul206}
\hat k^{\a}_{z_0}(\ell)= \overline{\left(\frac{z_0+i}{\sqrt 2\Phi_\sharp(z_0)} e^{i\Theta(z_0)\ell}\right)}\sqrt{ A^{\a}(\ell)}( f_{\a-\eta \ell}(z_0))^*,\quad \ell>0.
\end{equation}
Since $\cF^{\a} \hat k^{\a}_{z_0}=k^{\a-\b_\Phi}_{z_0}
$ and the reproducing kernels are dense in $\cH^2(\a-\b_\Phi)$, the span of the set of vectors $\{ \hat k_{z_0}^{\a} \mid z_0 \in \Omega\}$ is dense in  $\overline{\sqrt {A^{\a}}L^2([0,\infty),\bbC^2,d\mu^\a)}$. Therefore, \eqref{Falphadefinition} reflects the identity
$$
f(z)=\langle  f, k^{\a-\b_{\Phi}}_{z}\rangle, \quad \forall f\in \cH^2(\a-\b_{\Phi}).
$$
\end{remark}

\begin{proof}[Proof of Theorem \ref{theorem12}]
Passing from zero in Lemma \ref{lemma17jul20} to an arbitrary $L$ in Theorem \ref{theorem12} is a matter of change of the variable
$\ell\mapsto\ell+L$. It remans to show that
$$
\text{clos}\{\cup_{L\in\bbR_-}e^{iL\Theta} \cH^2(\a-\beta_\Phi-\eta L)\}=L^2_{\partial\Omega}.
$$
By Theorem \ref{thmDCTifandonlyif} we pass to orthogonal complements and use \eqref{17jul205}.
\end{proof}

\begin{remark} \label{remarkthetar0nonessential}
Let us show that the assumption $\Re\Theta(i)=0$ is not essential. Define $\Theta_1=\Theta+i\theta_r$, $\theta_r\in\bbR$. Define now 
$(\hat k_1)_{z_0}^\a\in \overline{ \sqrt{A^{\a}} L^2([0,\infty),\bbC^2,d\mu^\a) } $ by \eqref{17jul206} with respect to $\Theta_1$. We get
$$
\|(\hat k_1)_{z_0}^\a\|^2=\|\hat k_{z_0}^\a\|^2=k^{\a-\beta_{\Phi}}(z_0,z_0)
$$
and therefore a Fourier transform with respect to a new $\Theta_1$.
\end{remark}

\subsection{Specialization: A-L condition fails}

By definition
\begin{equation}\label{28sept201}
\lim_{y\to\infty}\frac{\Im \Theta(iy)}{y}=0.
\end{equation}
By Theorem \ref{23l410} and the Krein-de Branges formula \eqref{9jul4},
$$
\int_0^L\sqrt{1-|\fa(\ell)|^2}d\mu^\a(\ell)=0,
$$
that is, $|\fa(\ell)|=1$ $\mu^\a$-a.e. Therefore $\hat f \in \sqrt{A^\a} L^2(\bbR,\bbC^2,d\mu^\a)$ can be represented as
$$
\hat f(\ell)=\frac 1{\sqrt 2}\begin{pmatrix} 1\\-\fa(\ell)
\end{pmatrix}\hat g(\ell), \quad \hat g\in L^2(\bbR, \bbC^2, d\mu^\a).
$$
This is an isometry; since $L^2(\bbR, \bbC^2, d\mu^\a)$ is closed, so is $\sqrt{A^\a} L^2(\bbR,\bbC^2,d\mu^\a)$, and \eqref{Falphadefinition} is reduced to
$$
(\cF^\a\hat g)(z)=\frac{z+i}{\sqrt{2}\Phi_\sharp(z)}\int_0^\infty\vk^{\a-\eta\ell, \overline{\fc^\a(\ell)}}(z) e^{i\Theta(z)\ell}\hat g(\ell)d\mu^\a(\ell),
$$
where, recall,  for a fixed $\a$  the following limit  is well defined  for $\mu^\a$-a.e. $\ell$
$$
\fc^\a(\ell)=\lim_{y\to\infty}s_+^{\a-\eta\ell}(iy).
$$

We point out that in the case under consideration two measures $\mu^\a$ and $\mu^\b$ for the same spectrum $\sE$ are  possibly mutually singular for certain $\a,\beta\in\pi_1(\Omega)^*$ (see Theorem \ref{th20apr1}), even though in the average all isospectral  measures form $d\ell$, see \eqref{28sept203}.

\subsection{Specialization: A-L condition holds}\label{sept23sect74}
When A-L holds, it is common to normalize the complex Martin function so that
\[
\lim_{y\to\infty} \frac{ \Im \Theta(iy)}{y} = 1.
\]
By Theorem \ref{23l410} and the Krein--de Branges formula \eqref{9jul4}, for all $L$,
\[
\int_0^L \frac{1 - \lvert \fc^\a(\ell) \rvert^2 }{1 + \lvert \fc^\a(\ell) \rvert^2 } d\mu^\a(\ell)  = \int_0^L \sqrt{1 - \lvert \fa^\a(\ell) \rvert^2 } d\mu^\a(\ell) = L.
\]
It immediately follows that:
\begin{lemma}
If A-L holds, Lebesgue measure is absolutely continuous with respect to $\mu^\a$; in particular, $\lvert \fc^\a(\ell) \rvert < 1$ for Lebesgue-a.e. $\ell$.
\end{lemma}
However, we conjecture that the converse is not automatic:
\begin{conjecture}
There exists a Dirichlet-regular Widom set $\sE$ with DCT such that A-L holds and $\mu^\a$ has a nontrivial singular component with respect to Lebesgue measure for some $\alpha$.
\end{conjecture}
We begin by working in the general A-L case, with results that hold regardless of whether $\mu^\a$ has a nontrivial singular part. The results in this section imply, in particular, Lemma~\ref{lemma13} and Theorem~\ref{th3aug201}.

\begin{lemma}
For $\mu^\a$-a.e. $\ell$ the limits \eqref{3aug201} exist. Moreover
\begin{equation}\label{1aug202}
\begin{pmatrix}L_-^\a(z,\ell)& L_+^\a(z,\ell)
\end{pmatrix}=\frac{\sqrt{1+|\fc^\a(\ell)|^2}}{2K^\b(i)}\frac{z+i}{\Phi_\sharp(z)}
f_\b(z)\sqrt{A^\a(\ell)}\begin{pmatrix}t_1&0\\ 0& t_2
\end{pmatrix},
\end{equation}
where $t_1=\overline{\Phi(i\infty)}\in\bbT$ and $t_2=\overline{v(i\infty)}\Phi(i\infty)\in\bbT$.
\end{lemma}
 
 \begin{proof}
 Recall that
$$
\fc^\a(\ell)=\lim_{y\to+\infty}\frac{K^{\a-\eta\ell}_\sharp(iy)}{K^{\a-\eta \ell}(iy)},\quad f_\a=\begin{pmatrix} K^\a_{\sharp}&K^\a
\end{pmatrix}.
$$
Also
$$
k^{\a-\b_\Phi}(z,iy)
=\frac{z+i}{\sqrt 2\Phi_\sharp(z)}\overline{\frac{i(1+y)}{\sqrt 2\Phi_\sharp(iy)}}
\frac{K^\a(z)\overline{K^\a(iy)}-K_\sharp^\a(z)\overline{K_{\sharp}^\a(iy)}}{z+iy}
$$
Thus for $\b=\a-\eta \ell$ we have
$$
L_+^\a(z,\ell)=\lim_{y\to\infty}\frac{-k^{\b-\b_\Phi}(z,iy)}{K^\b(i)\overline{K^\b(iy)}}=\frac{z+i}{ 2K^\b(i)\Phi_\sharp(z)\overline{\Phi_\sharp(i\infty)}}(K^\b(z)-\overline{\fc^\a(\ell)} K^\b_\sharp(z))
$$
For $k^{\a-\b_\Phi}(z,-iy)$ we have
$$
k^{\a-\b_\Phi}(z,-iy)
=\frac{z+i}{\sqrt 2\Phi_\sharp(z)}\overline{\frac{i(1-y)}{\sqrt 2\Phi_\sharp(-iy)}}
\frac{K^\a(z){K_\sharp^\a(iy)}-K_\sharp^\a(z){K^\a(iy)}}{z-iy}
$$
Therefore, since $\overline{K_\sharp^\b(\bar z)}=K^\b(z)$ and $\overline{\Phi_\sharp(\bar z)}=\Phi(z)$, we get
$$
L_-^\a(z,\ell)=\lim_{y\to\infty}\frac{k^{\b-\b_\Phi}(z,-iy)}{K_\sharp^\b(-i)\overline{K_\sharp^\b(-iy)}}=\frac{z+i}{ 2K^\b(i)\Phi_\sharp(z){\Phi(i\infty)}}
(-K^\a(z){\fc^\a(\ell)}+ K^\a_\sharp(z))
$$
Combining this computations we obtain
$$
\begin{pmatrix}L_-^\a(z,\ell)& L_+^\a(z,\ell)
\end{pmatrix}=\frac{z+i}{2K^\b(i)\Phi_\sharp(z)}
f_\b(z)
\begin{pmatrix}1&-\overline{\fc^\a(\ell)}\\ -\fc^\a(\ell)&1
\end{pmatrix}
\begin{pmatrix}\overline{\Phi(i\infty)}&0\\ 0&\Phi_\sharp(i\infty)
\end{pmatrix}
$$
that is, \eqref{1aug202}.
 \end{proof}

\begin{lemma}
For the matrix \eqref{3aug201b}, 
 the following limit exists
\begin{equation}\label{2aug202}
\fd^\a(\ell)=\lim_{y\to+\infty}\det\cL^\a(iy,\ell)=\frac{{1-|\fc^\a(\ell)|^2}}{2k^{\a-\eta\ell}(i,i)}.
\end{equation}
\end{lemma}

\begin{proof}
As before $\b=\a-\eta\ell$. By \eqref{1aug202} and \eqref{11jul77}  we have
\[
\cL^\a(z,\ell)=
\frac{\sqrt{1+|\fc_\a(\ell)|^2}}{2K^\b(i)}\frac{z+i}{\Phi_\sharp(z)}
\begin{pmatrix}
\frac 1{\cW(z)}& 0\\
0&1
\end{pmatrix}\cT_\b(z)
\sqrt{A^\a(\ell)}\begin{pmatrix}\frac 1{\Phi(i\infty)}&0\\ 0&\frac{\Phi(i\infty)}{v(i\infty)}
\end{pmatrix}.
\]
Using \eqref{7jul3}, we get
\begin{align*}
\det\cL^\a(z,\ell)=&\frac{{1-|\fc^\a(\ell)|^2}}{4k^\b(i,i)}\left(\frac{z+i}{\Phi_\sharp(z)}\right)^2
\frac{2}{\Theta'(z)} \frac{\Phi}{z-i} \frac{\Phi_\sharp}{z+i}\overline{v(i\infty)}
\\
=&\frac{{1-|\fc^\a(\ell)|^2}}{2k^\b(i,i)}\frac{z+i}{z-i}v(z)\overline{v(i\infty)}\frac{1}{\Theta'(z)}=
\frac{{1-|\fc^\a(\ell)|^2}}{2k^\b(i,i)\Theta'(z)}.
\end{align*}
In the normalization $\lim_{y\to\infty}\Theta'(iy)=1$ we have \eqref{2aug202}.
\end{proof}

The measure $(d\mu^\a)_{\ac} = \frac{1 + \lvert \fc^\a(\ell) \rvert^2}{1 - \lvert \fc^\a(\ell) \rvert^2} d\ell$ generates a closed subspace $L^2(\bbR, \bbC^2, (d\mu^\a)_\ac)$ of  $L^2(\bbR, \bbC^2, d\mu^\a)$. The Fourier transform restricted to this subspace can be rewritten as:

\begin{theorem} The a.c. part of the Fourier transform is unitarily equivalent to the norm-preserving map $\cF^\a_{\ac} : L^2([0,\infty),\bbC^2) \to L^2_{\partial\Omega}$ given by
$$
(\cF^\a_{a.c.} \hat g)(z)=\int_0^\infty \frac{e^{i\Theta(z)\ell}}{\sqrt{\fd^\a(\ell)}}
\begin{pmatrix}
L^\a_-(z,\ell)& L_+^\a(z,\ell)
\end{pmatrix}\hat g(\ell) d\ell,\quad \hat g\in L^2([0,\infty),\bbC^2).
$$
\end{theorem}
\begin{proof} According to \eqref{1aug202} and \eqref{2aug202}
$$
\frac{1}{\sqrt{\fd^\a(\ell)}}\begin{bmatrix}0&1
\end{bmatrix}\cL^\a(z,\ell)
=\sqrt{\frac{1+|\fc^\a(\ell)|^2}{1-|\fc^\a(\ell)|^2}}\frac{z+i}{\sqrt{2}\Phi_\sharp(z)}
f_{\a-\eta\ell}(z)\sqrt{A^\a(\ell)}
\begin{pmatrix} t_1&0\\
0&t_2
\end{pmatrix}
$$
with constants $t_j\in\bbT$. On the other hand due to \eqref{Falphadefinition}
\begin{align*}
(\cF^\a_{a.c.}\hat f)(z)=& \frac{z+i}{\sqrt 2\Phi_\sharp(z)}\int_0^\infty e^{i\Theta(z) \ell} f_{\a-\eta \ell}(z) \sqrt{ A^{\a}(\ell)} \hat f(\ell) 
\frac{1+|\fc^\a(\ell)|^2}{1-|\fc^\a(\ell)|^2}d\ell
\\
=& \int_0^\infty\frac{e^{i\Theta(z)\ell}}{\sqrt{\fd^\a(\ell)}}\begin{bmatrix}0&1
\end{bmatrix}\cL^\a(z,\ell)\hat g(\ell)d\ell
\end{align*}
with
$$
\hat g(\ell)=\sqrt{\frac{1+|\fc^\a(\ell)|^2}{1-|\fc^\a(\ell)|^2}}\begin{pmatrix} \bar t_1&0\\
0&\bar t_2
\end{pmatrix}\hat f(\ell).
$$
Note that
\[
\|\hat g\|^2_{L^2([0,\infty),\bbC^2)}=\|\hat f\|^2_{L^2([0,\infty),\bbC^2,d\mu^\a)}=\|\cF^\a\hat f\|^2_{\cH^2(\a-\b_\Phi)}. \qedhere
\]
\end{proof}

The best known sufficient A-L condition is the finite logarithmic gap length condition \eqref{21apr4}. 
In the end of this section we show that \eqref{21apr4} implies that $\mu^\a$ is absolutely continuous for an arbitrary $\a$, moreover with a uniformly bounded derivative. 
Before that, we would like to comment on its relation to the concept of Ahlfors' analytic capacity,  see e.g. \cite{Sim15}.

Recall that for an arbitrary domain $\Omega$ we say that the boundary of this domain has positive \textit{analytic capacity} if there exists
nontrivial single-valued $w(z)\in H^\infty_{\Omega}$ such that $w(z_0)=0$ for a fixed $z_0\in\Omega$.  The analytic capacity w.r.t.  $z_0$ is given by
\begin{equation}\label{23sept201}
C^A_{z_0}(\Omega)=\sup\{|w'(z_0)|: \|w\|_{H^\infty_{\Omega}}\le 1,\ w(z_0)=0\}.
\end{equation}
Strict positivity of the analytic capacity, $C^A_{z_0}(\Omega)>0$, implies strict positivity of (potential-theoretic) \emph{capacity}, but not vice versa.  It is evident that a non trivial $w\in H^\infty_{\Omega}$ such that $w(z_0)=0$ allows a factorization 
$$
w(z)=\Phi_{z_0}(z) w_1(z)
$$
where $\Phi_{z_0}(z)$ is the complex Green function in the domain. Respectively $w_1\in H^\infty_{\Omega}(-\b_{\Phi_{z_0}})$ and the extremal problem \eqref{23sept201} can be reduced to the extremal problem for bounded functions with a \textit{given character}: find
\begin{equation}\label{23sept202}
\sup\{|w_1(z_0)|: w_1\in H^\infty_{\Omega}(-\b_{\Phi_{z_0}}),\  \|w_1\|\le 1\}.
\end{equation}
We restrict the further discussion again only to the case of Denjoy domains.
Note that the analytic capacity in this case   is closely related to the Lebesgue length of its boundary $\sE$, see e.g.  \cite[\S 8.8]{Sim15}. It is natural to raise the question: how to restate the problem \eqref{23sept201} for a boundary point of the domain, say $\infty\in\sE$? Having in mind \eqref{23sept202}, this problem has the following setting.
\begin{problem}\label{pr32aug1}
Let $\Omega$ be a Denjoy domain, $\Omega=\bbC\setminus\sE$, and $\infty\in\sE$. Let $\Theta(z)$ be the symmetric complex Martin function w.r.t. $\infty$ and $\eta$ be its additive character. Does there exist a non trivial additive character automorphic function $N_1(z)$,
$N_1(\g(z))=N_1(z)-\eta(\g)$, with a positive imaginary part $\Im N_1(z)\ge 0$ in the domain? 
In other words, does there exist a \textit{single valued} function $N(z)$ such that $\Im N(z)\ge 0$, $z\in\Omega$, and
\begin{equation}\label{23sept203}
\lim_{y\to\infty}\frac{\Im N(iy)}{\Im \Theta(iy)}>0.
\end{equation}
\end{problem}

\begin{proposition}
If the condition \eqref{21apr4} holds,   there exists a single valued function $N(z)$ with positive imaginary part in the domain  such that \eqref{23sept203} is satisfied.
\end{proposition}
\begin{proof}
We define $N(z)$ in the upper half plane by its argument on the real axis
$$
\chi_N(x)=\begin{cases} 0,\  x\in\sE,\  x>0\\
1/2,\  x\in\bbR\setminus \sE\\
1,\   x\in\sE,\  x<0
\end{cases}
$$
The function (up to a positive constant multiplier) is of the form
$$
N(z)=e^{\int_{\bbR}\left(\frac{1}{x-z}-\frac{x}{1+x^2}\right)\chi_N(x)dx}
=z
e^{\frac 1 2\int_{\bbR\setminus\sE}\left(\frac{1}{x-z}-\frac{x}{1+x^2}\right)\sgn x dx}.
$$
By definition $\Im N(x+i0)=0$ for a.e. $x\in\sE$. Assuming logarithmic  finite gap length condition
we obtain a finite limit
\begin{equation}\label{25aug203}
\s_N:=\lim_{y\to\infty}\frac{\Im N(iy)}{y}= e^{-\frac 1 2 \int_{\bbR\setminus\sE}\frac{|x|dx}{1+x^2}}>0.
\end{equation}
Since $N(z)$ assumes pure imaginary values in gaps we get an extension of this function in $\Omega$ due to the symmetry principle $N(\bar z)=-\overline{N(z)}$. Thus $\Im N(z)\ge 0$ for all $z\in\Omega$. 

In other words we get an affirmative answer to the question, which was posed in Problem \ref{pr32aug1}.
Due to \eqref{25aug203} we get a function $N_1(z)=\frac 1{\sigma_N}N(z)- \Theta(z)$ such that $\Im N_1(z) \ge 0$ for $z \in \Omega$
whose additive character is $-\eta$.
\end{proof}

\begin{corollary}
If the log-finite-length condition \eqref{21apr4} holds then 
$\mu^\a$ is absolutely continuous with uniformly bounded (in $\ell$ and $\a$) derivative.
\end{corollary}

\begin{proof}
Now in addition to the fact that $e^{-\Im \Theta(i) \ell}K^{\a-\eta\ell}(i)$ is monotonically decreasing we have that the function
$e^{-\Im N_1(i) \ell}K^{\a+\eta\ell}(i)$ is also decreasing, by Corollary~\ref{repkernelinequality1} applied to $\Delta(z) = e^{i\ell N_1(z)}$. Thus the directional derivative $\pd_\eta \log K^\a(i)  = \frac{d}{d\ell} \log K^{\a+\eta\ell}(i) \rvert_{\ell=0}$ obeys
\[
-\Im\Theta(i)\le\pd_\eta \log K^\a(i)  \le \Im N_1(i).
\]
In other words
\[
0\le \frac{d\mu^\a(\ell)}{d\ell}\le \frac{\Im N(i)}{\sigma_N}=e^{\frac 1 2 \int_{\bbR\setminus\sE}\frac{|x|dx}{1+x^2}}\cos\left(\frac 1 2\int_{\bbR\setminus \sE}\frac{\sgn x\, dx}{1+x^2}\right). \qedhere
\]
\end{proof}

\section{Almost periodicity of coefficients} \label{sectionAP}

\subsection{Almost periodic measures} 
A complex measure $\nu$ on $\bbR$ is said to be translation bounded if for any compact $S \subset \bbR$,
\begin{equation}\label{26jun1}
\lVert \nu \rVert_S := \sup_{x\in\bbR} \lvert \nu \rvert (x+S) < \infty.
\end{equation}
In this notation, the measure $\lvert \nu\rvert$ is said to be uniformly continuous if $\lim_{L \downarrow 0} \lVert \nu \rVert_{[0,L]} = 0$.

Almost periodicity of translation bounded measures is usually defined by convolution with some family of test functions. 

\begin{definition} \label{definitionAPmeasure1}
Let $\nu$ be a translation bounded measure on $\bbR$ and $X$ a set of test functions on $\bbR$. We say $\nu$ is an $X$-almost periodic measure if for all  $h \in X$, the convolution
\[
(h * \nu)(\ell) := \int h(\ell - l)\,d\nu(l)
\]
is a (uniformly) almost periodic function.
\end{definition}

In particular, $C_c(\bbR)$-almost periodicity is commonly called strong almost periodicity \cite{ArgGil74,ArgGil90}, and we will consider the stronger notion of $PC_c(\bbR)$-almost periodicity, where $PC_c(\bbR)$ denotes the set of piecewise continuous compactly supported functions. With uniform continuity, these properties are equivalent.

\begin{lemma}\label{lemma26jun1}
If $\nu$ is a complex measure on $\bbR$ such that $\lvert \nu\rvert$ is uniformly continuous, the following are equivalent:
\begin{enumerate}[(i)]
\item $\nu$ is strongly almost periodic;
\item  for every $L > 0$, $\nu((\ell,\ell+L])$ is an almost periodic function of $\ell$.
\end{enumerate}
\end{lemma}

\begin{proof}
(i) $\implies$ (ii): Fix $L > 0$ and define the sequence of functions $h_n(x)  = \max( 0, 1 - n \dist(x, [0,L]))$. Then
\[
\lvert (h_n * \nu)(x) - (\chi_{[0,L]} * \nu)(x)  \rvert \le \lvert \nu \rvert ([-1/n, 0]) + \lvert \nu \rvert ([L, L+1/n]) 
\]
Thus, by uniform continuity of $\lvert \nu \rvert$, $h_n * \nu$ converges uniformly to $\chi_{[0,L]} * \nu$ as $n \to \infty$. Since the functions $h_n * \nu$ are almost periodic, their uniform limit $\chi_{[0,L]} * \nu$ is almost periodic.

Conversely, assume that (ii) holds and fix $h\in C_c(\bbR)$. Take a sequence $h_n$ of piecewise constant functions with $\supp h_n \subset \supp h$ which uniformly approximate $h$. Then $h_n * \nu$ are almost periodic, as linear combinations of almost periodic functions. Moreover, for all $x$,
\[
\lvert (h_n * \nu)(x) - (h* \nu)(x)  \rvert \le \lVert h_n - h \rVert_\infty \lvert\nu\rvert(\supp h),
\]
so $h_n * \nu$ converges to $h *\nu$ uniformly as $n\to\infty$. It follows that $h * \nu$ is almost periodic, as a uniform limit of almost periodic functions. Thus, $\nu$ is strongly almost periodic.
\end{proof}

Our proofs will require another perspective on almost periodicity in terms of linear sampling along a compact torus. We will use the following terminology:

\begin{definition} \label{definitionAPmeasure2}
For a set of test functions $X$, the measure $\nu$ is an \emph{$X$-almost periodic measure with frequency vector $\eta \in \bbR^\infty$} if it is a member of a collection $\{\nu^\a\}_{\a \in \bbT^\infty}$ of complex measures on $\bbR$ indexed by $\a \in  \bbT^\infty$, a torus of countable dimension with the product topology, with the following properties:
\begin{enumerate}[(i)]
\item (uniform local boundedness) For every compact $S \subset \bbR$, $\sup_{\a \in \bbT^\infty} \lvert \nu^\a \rvert(S) < \infty$.
\item (translation is a linear action on the torus) The  vector $\eta$ encodes translation in the sense that 
\begin{equation}\label{APtranslationinvariance}
\nu^\a((0,L]) = \nu^{\a-\eta \ell}((\ell, \ell+ L]), \qquad \forall \alpha \in \Gamma \, \forall\ell \in \bbR \, \forall L > 0.
\end{equation}
\item For any $h \in X$, $\int h \,d\nu^\a$ is a continuous function of $\a$.
\end{enumerate}
\end{definition}

Clearly, Definition~\ref{definitionAPmeasure2} implies Definition~\ref{definitionAPmeasure1}. The properties in Definition~\ref{definitionAPmeasure2} can also be reconstructed by integrating measures on intervals:

\begin{lemma}\label{lemmaAPcriterion}
Let $\{\nu^\a\}_{\a\in\bbT^\infty}$ be a collection of complex measures which is uniformly locally bounded, obeys \eqref{APtranslationinvariance}, has no point masses, and for any $L > 0$, $\nu^\a((0,L])$ is a continuous function of $\a$. Then $\{\nu^\a\}$ is a collection of $PC_c(\bbR)$-almost periodic measures with frequency vector $\eta$.
\end{lemma}

\begin{proof}
We prove that $\a \mapsto \int h d\nu^\a$ is continuous for successively larger classes of test functions $h$. By assumption, this holds for $h = \chi_{(0,L]}$. By translation, it holds for $h = \chi_{(L_1, L_2]}$ for any $L_1 < L_2$. Since $\nu^\a$ are continuous measures, it holds for the characteristic function of any bounded interval. By using linear combinations, it holds for piecewise constant compactly supported functions.

If $h$ is a piecewise continuous compactly supported function, it is uniformly approximated by piecewise constant compactly supported $h_n$ with $\supp h_n \subset \supp h$. Then $\int h_n \,d \nu^\a$ are continuous in $\a$. Moreover, for all $\a$,
\[
\left\lvert \int h_n d\nu^\a  - \int h d\nu^\a \right\rvert \le \lVert h_n - h \rVert_\infty \lvert\nu^\a\rvert(\supp h),
\]
so $ \int h_n d\nu^\a$ converges to $ \int h d\nu^\a$ uniformly as $n\to\infty$. It follows that $ \int h d\nu^\a$ is continuous in $\a$.
\end{proof}

We will also need an abstract lemma:

\begin{lemma}\label{lemma23jun1}
Let $\{\nu^\a\}$ be a collection of $PC_c(\bbR)$-almost periodic measures with frequency vector $\eta$. For any $L> 0$, if $g : \bbT^\infty \to \bbC$ is continuous, then the function
\[
\cY(\a) = \int_0^L g(\a - \eta l) d\nu^\a(l)
\]
is continuous in $\a$.
\end{lemma}

\begin{proof}
Assume that $\a_n \to \a$. Since  $g(\a_n - \eta \ell) \to g(\a - \eta \ell)$ uniformly in $\ell \in [0,L]$ and by uniform local boundedness, it follows that
\[
\int_0^L g(\a_n -  \eta \ell) d\nu^{\a_n}(\ell) - \int_0^L g(\a -  \eta \ell) d\nu^{\a_n}(\ell) \to 0, \qquad n\to\infty.
\]
Meanwhile, by applying Lemma~\ref{lemmaAPcriterion} to the function $h(\ell) = \chi_{(0,L]}(\ell) g(\a - \eta \ell)$, we conclude that
\[
\int_0^L g(\a -  \eta \ell) d\nu^{\a_n}(\ell) - \int_0^L g(\a -  \eta \ell) d\nu^{\a}(\ell) \to 0, \qquad n\to\infty.
\]
Together, these conclusions imply $\cY(\a_n) \to \cY(\a)$ as $n\to\infty$.
\end{proof}

\subsection{Almost periodicity in A-gauge}

Our next goal is to show that the coefficients of the constructed canonical systems in A-gauge are almost periodic.

\begin{lemma}\label{lemma23jun0}
The family $\{\mu^\a\}_{\a \in \pi_1(\Omega)^*}$ is a family of $PC_c(\bbR)$-almost periodic positive measures with frequency vector $\eta$.
\end{lemma}

\begin{proof}
It has already been proved that the measures $\mu^\a$ are positive, continuous measures and  \eqref{mualphadefinition} can be written as
\begin{equation}\label{26jun5}
\mu^\a((0,L]) = \mu^{\a - \eta\ell}((\ell,\ell+L]) = \cX_L(\a)
\end{equation}
where 
\begin{equation}\label{21oct1}
\cX_L(\a)  := L \theta_i - \log \frac{K^{\a - \eta L}(i)}{K^\a(i)}.
\end{equation}
By DCT, $K^\a(i)$ is continuous and positive in $\a$, so $\cX_L$ is a continuous function of $\a$.

After translation, we can assume that the compact $S$ is in $(0,L]$ for some $L$. Since $\mu^\a$ are positive measures, by continuity and compactness,
\[
\sup_{\a \in \pi_1(\Omega)^*} \mu^\a(S) \le \sup_{\a \in \pi_1(\Omega)^*} \cX_L(\a) < \infty. \qedhere
\]
\end{proof}

\begin{remark}
Similarly to almost periodic functions, almost periodic measures have an average $\bbE(\mu)$ with the property that $\bbE(h * \mu) = \bbE(\mu) \int h(l) dl$ for suitable test functions $h$. In our case, \eqref{21oct1} implies that the measures $\mu^\a$ have average $\theta_i$; compare \eqref{28sept203}.  
\end{remark}

\begin{lemma}
The family $\{\fa^\a d\mu^\a\}_{\a \in \pi_1(\Omega)^*}$ is a family of $PC_c(\bbR)$-almost periodic complex measures with frequency vector $\eta$.
\end{lemma}

\begin{proof}
The measures $d\mu^\a$ are uniformly locally bounded and continuous, and since $\lvert \fa^{\a}\rvert \le 1$, so are the measures $\fa^{\a} \,d\mu^\a$. The representation of translation follows from \eqref{aalphadefinition}. From \eqref{aalphadefinition}, we obtain  
\begin{equation}\label{22oct1}
 \int_0^{L} \fa^{\a}(l) \,d\mu^\a(l) =  \int_0^{L} s_+^{\a - \eta l}(i) \,d\mu^\a(l) - \frac 12 \left( s_+^{\a - \eta \ell}(i) - s_+^{\a}(i) \right).
\end{equation}
Since $s_+^\a(i)$ depends continuously on the character,  so does \eqref{22oct1}, by  Lemma~\ref{lemma23jun1}.
\end{proof}

In particular, this proves Theorem~\ref{theorem11}(c).

\subsection{Passing to the Dirac gauge}

\subsubsection{The gauge transform} 
Not every canonical system can be transformed to Dirac gauge \eqref{21may1}. We will describe a transformation and note the requirements along the way. Applying the Krein--de Branges formula \eqref{9jul4} to Dirac gauge shows that the variable $t$ for a canonical system in the Dirac gauge is equal to the exponential type of the transfer matrix $\fD(z,t)$. That is $t=\ell$ in our notation.

Thus, to pass from A-gauge to Dirac gauge, we first pass to derivative w.r.t. $\ell$, which we denote by $(\dot{...})$, and obtain the system in the form
\[
\dot \fA(z,\ell) j =\pd_\ell\fA(z,\ell)j=\fA(z,\ell)(iz A(\ell)-B(\ell))\dot \mu(\ell).
\]
Note that now $\det(\dot\mu(\ell)A(\ell))=1$. We use a transformation
\[
\fD(z,\ell)= \cU(0)^{-1} \fA(z,\ell)\cU(\ell),\quad \cU(\ell)\in \SU (1,1).
\]
Right-multiplication is a gauge transformation; left-multiplication ensures $\fD(z,0)=I$ and affects the Schur function. We have
\begin{align*}
\dot \fD(z,\ell)j=  \cU(0)^{-1} \dot \fA(z,\ell)\cU(\ell)j+  \cU(0)^{-1} \fA(z,\ell)\dot\cU(\ell)j
\\
= \cU(0)^{-1} \fA(z,\ell)(iz A(\ell)-B(\ell))\dot\mu(\ell)j\cU(\ell)j+
 \fD(z,\ell)\cU(\ell)^{-1}\dot\cU(\ell)j
\\
= \fD(z,\ell)\cU(\ell)^{-1}(iz A(\ell)-B(\ell))\dot\mu(\ell)(\cU(\ell)^*)^{-1}+ \fD(z,\ell)\cU(\ell)^{-1}\dot\cU(\ell)j.
\end{align*}
By choosing
\[
\cU(\ell)=\sqrt{A(\ell)\dot\mu(\ell)}=\cV(\fc(\ell)) =\frac{1}{\sqrt{1-|\fc(\ell)|^2}}\begin{pmatrix} 1&-\overline{\fc(\ell)}\\
-\fc(\ell)& 1
\end{pmatrix}
\]
we get a canonical system in a \textit{Dirac} (D) gauge
\[
\dot \fD(z,\ell)j=\fD(z,\ell)(iz I-Q(\ell)),\ Q(\ell)=\dot\mu(\ell)\cU(\ell)^{-1}B(\ell)(\cU(\ell)^*)^{-1}- \cU(\ell)^{-1}\dot\cU(\ell)j.
\]
We point out that $\fc(\ell)$ should be differentiable to this end. We automatically have the normalization $\tr Q(\ell)j=0$. In addition one of standard normalizations \cite{LevSar} requires
$\tr Q(\ell)=0$. To this end,  generally speaking, we need an extra diagonal gauge transform
\[
\fD_1(z,\ell)=\fD(z,\ell)\cU_{\psi}(\ell),\quad \cU_{\psi}(\ell)=\begin{pmatrix} e^{-i\psi(\ell)}&0\\ 0&e^{i\psi(\ell)}
\end{pmatrix}, \quad \psi(0)=0.
\]
We get a canonical system 
\[
\dot \fD_1(z,\ell)j=\fD(z,\ell)_1(iz I-Q_1(\ell)),\ Q_1(\ell)=\cU_\psi(\ell)^{-1}Q(\ell)(\cU_\psi(\ell)^*)^{-1}- i\dot \psi(\ell) I.
\]
Thus $\psi(\ell)$ should be chosen as  the integral
\[
i\psi(\ell)=\frac 1 2 \int_0^\ell\tr Q(\ell) d\ell.
\]
Note that almost periodicity of $\fc(\ell)$ and $\dot \fc(\ell)$ does not necessarily imply almost periodicity of the quantity $e^{2i\psi(\ell)}$ related to the integral \cite{Bes54}.

We arrive to the following proposition.

\begin{theorem}
Let $\fA(z,\ell)$ be the transfer matrix of a canonical system in A-gauge with parameters $\{\mu(\ell),\fa(\ell)\}$, where $\ell$ is its exponential type. If $\mu$ is absolutely continuous and $\fa(\ell)$ is differentiable, then  it can be transformed to a D-gauge,
$$
\dot\fD(z,\ell)j=\fD(z,\ell)(iz I-Q(\ell)),
$$
where
\begin{equation}\label{12jul1}
Q(\ell)=\frac{1}{1-|\fc(\ell)|^2} 
\begin{pmatrix}0&\overline{2\fc(\ell)+\dot\fc(\ell)} \\
-(2\fc(\ell)+\dot\fc(\ell))& 0
\end{pmatrix}+i\frac{\Im(\dot \fc(\ell) \overline{ \fc(\ell)})}{1-|\fc(\ell)|^2} I.
\end{equation}
With an extra (canonical form) normalization condition
$$
\dot\fD_1(z,\ell)j=\fD_1(z,\ell)(iz I-Q_1(\ell))),\quad \tr Q_1(\ell)=0,
$$
the matrix coefficient $Q_1(\ell)$ is given by
\begin{equation}\label{12jul2}
Q_1(\ell)=\frac{1}{1-|\fc(\ell)|^2} 
\begin{pmatrix}0&e^{2i\psi(\ell)}\overline{(2\fc(\ell)+\dot\fc(\ell))} \\
-e^{-2i\psi(\ell)}(2\fc(\ell)+\dot\fc(\ell))& 0
\end{pmatrix}
\end{equation}
with
\[
\psi(\ell)=\int_0^\ell\frac{\Im (\overline{\fc(\ell)}\dot\fc(\ell))d\ell}{1-|\fc(\ell)|^2}.
\]
\end{theorem}

\begin{proof}
It remains to compute $Q(\ell)$, $Q_1(\ell)$ and $\psi(\ell)$. Recall
\[
\dot \mu(\ell) B(\ell)=\frac{1+|\fc(\ell)|^2}{1-|\fc(\ell)|^2}\frac{1}{1+|\fc(\ell)|^2}
\begin{pmatrix}0&2\overline{\fc(\ell)}\\
-2\fc(\ell)&0
\end{pmatrix}.
\]
We have
\[
\frac{(1-|\fc|^2)^2}{2}\dot\mu\cU^{-1}B(\cU^*)^{-1}=
\begin{pmatrix}1&\overline{\fc}\\
\fc&1
\end{pmatrix}
\begin{pmatrix}0&\overline{\fc}\\
-\fc&0
\end{pmatrix}
\begin{pmatrix}1&\overline{\fc}\\
\fc&1
\end{pmatrix}=
(1-|\fc|^2)\begin{pmatrix}0&\overline{\fc}\\
-\fc&0
\end{pmatrix}.
\]
That is,
\[
\dot\mu \cU^{-1}B(\cU^*)^{-1}=
\frac{1}{{1-|\fc(\ell)|^2}}
\begin{pmatrix}0&2\overline{\fc(\ell)}\\
-2\fc(\ell)&0
\end{pmatrix}.
\]
Further,
\[
\dot \cU=\frac{1}{\sqrt{1-|\fc|^2}}
\begin{pmatrix}0&-\overline{\dot \fc}\\
-\dot\fc&0
\end{pmatrix}+\frac{\Re (\dot \fc\bar \fc)}{1-|\fc|^2}\cU
\]
and we obtain
\[
\cU^{-1}\dot \cU j=\frac{1}{1-|\fc|^2} 
\begin{pmatrix} 1&\overline{ \fc}\\
\fc&1
\end{pmatrix}
\begin{pmatrix}0&-\overline{\dot \fc}\\
\dot\fc&0
\end{pmatrix}+\frac{\Re (\dot \fc\bar \fc)}{1-|\fc|^2}j
=
\frac{1}{1-|\fc|^2} 
\begin{pmatrix} i\Im (\dot \fc\bar \fc)&-\overline{\dot \fc}\\
\dot \fc&-i\Im (\fc\overline{\dot \fc})
\end{pmatrix}
\]
Finally,
\[
Q=\frac{1}{1-|\fc|^2} 
\begin{pmatrix}0&\overline{2\fc+\dot\fc} \\
-(2\fc+\dot\fc)& 0
\end{pmatrix}+i\frac{\Im(\dot \fc \bar \fc)}{1-|\fc|^2} I. \qedhere
\]
\end{proof}

\subsubsection{Logarithmic gap length condition and the first term in the asymptotics}
We consider the logarithmic gap length condition \eqref{21apr4}.

\begin{theorem} \label{theorem5jula}
The limit values
\[
\lim_{y\to\infty}R^{\a,\tau}(iy)
\]
exist for all $(\a,\tau)\in\pi_1(\Omega)^*\times\bbT$ if and only if \eqref{21apr4} holds. Moreover, in this case
\begin{equation}\label{22apr3}
\cR(\a,\tau):=-i\lim_{y\to\infty}R^{\a,\tau}(iy)=|1-\tau  s^\a_+(i)|
e^{-\int_{\bbR}\frac{\xi}{1+\xi^2}\chi^{\a,\tau}(\xi)d\xi}.
\end{equation}
represents a continuous strictly positive function on $\pi_1(\Omega)^*\times\bbT$.  
\end{theorem}

\begin{proof} We have
\begin{equation}\label{22apr2}
R^{\a,\tau}(iy)=i|1-\tau  s^\a_+(i)|
e^{\int_{\bbR}\left(-\frac{(y^2-1)\xi}{(\xi^2+y^2)(1+\xi^2)}+i\frac{y}{\xi^2+y^2}\right)\chi^{\a,\tau}(\xi)d\xi}.
\end{equation}
We choose $(\a_0,\tau_0)$ such that $x^0_j=a_j$ for $b_j<0$ and $x^0_j=b_j$ for $a_j>0$. In this choice we get
$$
R^{\a_0,\tau_0}(iy)=i|1-\tau_0 s^{\a_0}_+(i)|
e^{\frac 1 2\int_{\bbR\setminus\sE}\left(\frac{(y^2-1)|\xi|}{(\xi^2+y^2)(1+\xi^2)}-i\frac{y\,\sgn\xi}{\xi^2+y^2}\right)d\xi}.
$$
Due to the Beppo Levi Theorem
$$
\lim_{y\to\infty}\int_{\bbR\setminus \sE}\frac{y^2}{x^2+y^2}\frac{|\xi|d\xi}{\xi^2+1}=
\int_{\bbR\setminus \sE}\frac{|\xi|d\xi}{\xi^2+1}.
$$
Thus existence of the limit implies \eqref{21apr4}. 

On the other hand if \eqref{21apr4} holds we get an integrable majorant for both summands in the integral \eqref{22apr2}. In particular,
for the second one we use 
$$
\frac{2y}{x^2+y^2}\le\begin{cases}
1,& |x|\le 1\\
\frac 1{|x|}, & |x|\ge 1
\end{cases} \quad \text{for} \ y\ge 2.
$$
Therefore we can pass to the limit and we get \eqref{22apr3}. Since the resulting value is continuous in $D\in \cD(\sE)$, $\cR(\a,\tau)$ is continuous. 
\end{proof}

\begin{corollary}\label{cor13} 
Let
\begin{equation}\label{22apr77}
\Xi(\a)=\frac{\cR(\a,1)-\cR(\a,-1)+ i(-\cR(\a,i)+\cR(\a,-i))}{2+\cR(\a,1)+\cR(\a,-1)}.
\end{equation}
If \eqref{21apr4} holds, then the following limits exist and represent continuous functions in $\a$
\begin{align}\label{22apr7}
\lim_{y\to\infty}s_{+}^\a(iy)=\Xi(\a),\quad
\lim_{y\to\infty}s_{-}^\a(iy)=\overline{\Xi(\a)}.
\end{align}
Moreover
\begin{equation}\label{13jul4}
\sup_{\a\in\pi_1(\Omega)^*}\frac{1}{1-|\Xi(\a)|^2}<\infty.
\end{equation}

\end{corollary}
\begin{proof} We have
$$
\cR(\a,1)+\cR(\a,-1)=2\lim_{y\to\infty}\frac{1+s_+^\a(iy)s_-^{\a}(iy)}{1-s_+^\a(iy)s_-^{\a}(iy)}
$$
Thus the limit
\begin{equation}\label{13jul3}
\lim_{y\to\infty}\frac{2}{1-s_+^\a(iy)s_-^{\a}(iy)}=1+\frac{\cR(\a,1)+\cR(\a,-1)}{2}
\end{equation}
exists.
Also
$$
\frac{\cR(\a,1)-\cR(\a,-1)} 2=\lim_{y\to\infty}\frac{s_+^\a(iy)+s_-^{\a}(iy)}{1-s_+^\a(iy)s_-^{\a}(iy)}
$$
Therefore
\begin{equation}\label{13jul201}
\lim_{y\to\infty}(s_+^\a(iy)+s_-^{\a}(iy))=\frac{2(\cR(\a,1)-\cR(\a,-1))}{2+\cR(\a,1)+\cR(\a,-1)}
\end{equation}
Similarly,
$$
\frac{\cR(\a,i)-\cR(\a,-i)} 2=i\lim_{y\to\infty}\frac{s_+^\a(iy)-s_-^{\a}(iy)}{1-s_+^\a(iy)s_-^{\a}(iy)}
$$
Therefore
\begin{equation}\label{13jul2}
\lim_{y\to\infty}(s_+^\a(iy)-s_-^{\a}(iy))=\frac{2i(-\cR(\a,i)+\cR(\a,-i))}{2+\cR(\a,i)+\cR(\a,-i)}.
\end{equation}
Note that by definition \eqref{22apr3} $\cR(\a,\tau)$ is positive. From \eqref{13jul201} and \eqref{13jul2} we get \eqref{22apr7} with \eqref{22apr77}.
By  \eqref{13jul3} we have  \eqref{13jul4}.
\end{proof}

In connection with Dirac operators the following normalization condition is natural.
\begin{definition}
We denote by $\cS_D(\sE)$ the set of $s_+ \in \cS(\sE)$ for which the following limit exists $\lim_{y\to\infty}s_+(iy)=0$.
\end{definition}

\begin{proposition}
If \eqref{21apr4} holds, then $\cS_D(\sE)$ is a compact, which allows the following parametric description
\begin{equation}\label{10nov201}
\cS_D(\sE)=\left\{(s_{\rm{D}})_+^{\a,\tau}=\tau\frac{s^\a-\Xi(\a)}{1-s_{+}^\a\overline{\Xi(\a)}},\quad (\a,\tau)\in\pi_1(\Omega)^*\times\bbT\right\}
\end{equation}
\end{proposition}

\begin{proof}
By Cor.~\ref{corSE}, any $s_+ \in \cS(\sE)$ is of the form 
$\begin{pmatrix} s_+ & 1 \end{pmatrix} \simeq \begin{pmatrix} s_+^\a & 1 \end{pmatrix} \cU$ for some $\cU \in \SU(1,1)$. If $s_+\in\cS_D(\sE)$, by Cor.~\ref{cor13}, we have
$$
\begin{pmatrix}0 & 1 \end{pmatrix} \simeq \begin{pmatrix} \Xi(\a) & 1 \end{pmatrix} \cU.
$$  
Therefore $\cU=\cV(\Xi(\a))\cU_{\tau}$ for some $\tau \in \bbT$. The inverse statement is evident.
\end{proof}

\subsubsection{Finite gap length condition and the second term in the asymptotics}

The  finite sum length gap condition with respect  to infinity is \eqref{22may7}. When \eqref{22may7} holds, we can define
\begin{equation}\label{13jun10}
\U(\a,\tau) =\sum_j\left(x_j-\frac{a_j+b_j}{2}\right), \quad (\a,\tau)=\pi(D),\ D\in \cD(\sE).
\end{equation}
Since the RHS is continuous in $\cD(\sE)$, $\U(\a,\tau)$ is continuous on $\pi_1(\Omega)^*\times\bbT$.
\begin{lemma} 
If \eqref{22may7} holds, then
\begin{equation}\label{18may3}
R^{\a,\tau}(iy)=i\cR(\a,\tau)\left(1+\frac i{y}\U(\a,\tau)+ o\left(\frac 1 y\right)\right),
\end{equation}
uniformly in $\pi_1(\Omega)^*\times\bbT$.
\end{lemma}
\begin{proof} 
\eqref{22may7} evidently implies \eqref{21apr4}, therefore the integral related to the second term in
$$
\int \left(\frac{1}{\xi-z}-\frac{\xi}{1+\xi^2}\right)\chi^{\a,\tau}(\xi)d\xi
$$
converges. For the first one we have
$$
\int\left(\frac{1}{\xi-z}+\frac 1 z\right)\chi^{\a,\tau }(\xi)d\xi-\frac 1 z\int\chi^{\a,\tau}(\xi)d\xi.
$$
Since 
$$
\left|\int\frac{\xi^2+i\xi y}{\xi^2+y^2}\chi^{\a,\tau }(\xi) d\xi\right|\le \frac 1 2\int_{\bbR\setminus \sE}\frac{\xi^2+|\xi| y}{\xi^2+y^2} d\xi
$$
and the last integrand has an integrable majorant it tends to zero as $y\to\infty$. We get 
$$
\int\frac{1}{\xi-iy}\chi^{\a,\tau}(\xi)d\xi=\frac i y\int\chi^{\a,\tau}(\xi)d\xi+ o\left(\frac 1 y\right)
$$
uniformly in $(\a,\tau)$.  Respectively, we obtain
$$
R^{\a,\tau}(iy)=i\cR(\a,\tau)e^{\frac i y\int\chi^{\a,\tau}(\xi)d\xi+ o\left(\frac 1 y\right)}
$$
which gives \eqref{18may3}.
\end{proof}

Together with Corollary \ref{cor13} we have the main conclusion on two term asymptotics for $s_\pm^\a(z)$ at infinity.

\begin{proposition}\label{prop13jul}
If \eqref{22may7} holds, then
\begin{equation}\label{13jul11}
s_{+}^\a(iy)=\Xi(\a)+\frac {\Xi_1(\a)} y +o\left(\frac 1 y\right), \qquad y \to \infty
\end{equation}
uniformly in $\a$. Moreover, $\Xi(\a)$ and $\Xi_1(\a)$ are continuous and can be given explicitly in terms of $\cR(\a,\tau)$ and $\U(\a,\tau)$.
Respectively  the Schur functions $(s_{\rm D})_+^{\a,\tau}$ defined by \eqref{10nov201}
obey
\[
(s_{\rm D})^{\a,\tau}_+(iy)=  \frac {\tau\Xi_1(\a)}{1 - \lvert \Xi(\a) \rvert^2}  \frac 1y +o\left(\frac 1 y\right), \qquad y \to \infty.
\]
\end{proposition}

\subsubsection{Almost periodicity in D-gauge} We now prove a more precise version of Theorem~\ref{th3m}:

\begin{theorem}
Let $\Omega=\bbC\setminus \sE$ be of Widom type and DCT hold. If $\sE$ obeys the gap length condition \eqref{22may7}, then for an arbitrary $(\a,\tau)$,  $(s_{\rm D})_+^{\a,\tau}\in \cS_D(\sE)$ 
is the Schur spectral function of a canonical system  \eqref{12jul1} with almost periodic $Q^{\a,\tau}(\ell)$. Moreover, the coefficients are of the form
\begin{equation}\label{13jul20}
\fc^{\a,\tau}(\ell)=\tau \Xi(\a-\eta\ell),\quad \dot\fc^{\a,\tau}(\ell)=2\tau (\Xi_1(\a-\eta\ell)-\Xi(\a-\eta \ell)).
\end{equation}
\end{theorem}

\begin{proof}
According to Proposition \ref{prop13jul}, the limit exists $\lim_{y\to\infty} s_+^{\a}(iy)=\Xi(\a)$. Therefore, the first relation \eqref{13jul20} holds.
By \eqref{13jul4} and
$$
\ell=\int_0^\ell\frac{1-|\Xi(\a-\eta l)|^2}{1+|\Xi(\a-\eta l )|^2}d\mu^{\a,\tau}(l)
$$
we  get that $\mu^{\a,\tau}$ is absolutely continuous w.r.t. $\ell$, moreover
\begin{equation}\label{23sept205}
\dot \mu^{\a,\tau}(\ell)=\frac{1+|\Xi(\a-\eta \ell) |^2}{1-|\Xi(\a-\eta \ell) |^2}.
\end{equation}
Using the Ricatti equation \eqref{RicattiEquation} in the integral form we have
$$
s_+^{\a-\eta \ell}(z)-s_+^\a(0)=
$$
$$
-iz \int_0^\ell\begin{pmatrix}s^{\a-\eta l}_+(z)&1
\end{pmatrix}
\frac{
\begin{pmatrix} 1&-\overline{\Xi(\a-\eta l)}\\ -\Xi(\a-\eta l)&1
\end{pmatrix}^2}{1-|\Xi(\a-\eta l)|^2}
\begin{pmatrix} 1\\ s^{\a-\eta l}_+(z)
\end{pmatrix}
d l
$$
$$
+
\int_0^\ell\begin{pmatrix}s^{\a-\eta l}_+(z)&1
\end{pmatrix}
\frac{
\begin{pmatrix} 0 &2\overline{\Xi(\a-\eta l)}\\ -2\Xi(\a-\eta l)&0
\end{pmatrix}}{1-|\Xi(\a-\eta l)|^2}
\begin{pmatrix} 1\\ s^{\a-\eta l}_+(z)
\end{pmatrix}
d l
$$
$$
=-2iz\int_0^\ell(s_+^{\a-\eta l}(z)-\Xi(\a-\eta l))\frac{1-s_+^{\a-\eta l}(z)\overline{\Xi(\a-\eta l)}}{1-|\Xi(\a-\eta l)|^2}d l
$$
$$
-2\int_0^\ell\frac{\Xi(\a-\eta l)-(s_+^{\a-\eta l}(z))^2 \overline{\Xi(\a-\eta l)}}{1-|\Xi(\a-\eta l)|^2}d l
$$
According to \eqref{13jul11} we can pass to the limit as $z=iy$, $y\to\infty$. We obtain
$$
\Xi(\a-\eta \ell)-\Xi(\a)=2 \int_0^\ell{\Xi_1(\a-\eta l)} dl-2 \int_0^\ell{\Xi(\a-\eta l)} dl.
$$
That is, $\fc^\a(\ell)=\Xi(\a-\eta \ell)$ is  differentiable and moreover the derivative is almost periodic, since we get the representation \eqref{13jul20}.
\end{proof}

\begin{remark}
As it was already mentioned almost periodicity of $Q^{\a,\tau}(\ell)$ does not guarantee almost periodicity of the phase function $e^{2i\psi(\ell)}$ in the representation 
\eqref{12jul2}
for $Q_1^{\a,\tau}(\ell)$ in the Dirac gauge.  It requires additional restrictions on the set $\sE$. A similar phenomenon we will discuss precisely in the next section, where we will see that the logarithmic gap length condition w.r.t the origin \eqref{22may5} has to be accompanied by a potential theory constraint \eqref{22may4}. Note also that the translations
\[
\begin{pmatrix} (s_{\rm{D}})_+^{\a(\ell),\tau(\ell)}(z) &1\end{pmatrix}\simeq\begin{pmatrix} (s_{\rm D})_+^{\a,\tau}(z) &1\end{pmatrix}\fD(z,\ell)
\ \text{and}\ 
\begin{pmatrix} (s_{\rm D})_+^{\a_1(\ell),\tau_1(\ell)}(z) &1\end{pmatrix}\simeq\begin{pmatrix} (s_{\rm D})_+^{\a,\tau}(z) &1\end{pmatrix}\fD_1(z,\ell)
\]
are respectively of the form
\[
(\a(\ell),\tau(\ell))=(\a-\eta \ell,\tau) \quad \text{and}\quad (\a_1(\ell),\tau_1(\ell))=(\a-\eta \ell,\tau e^{-2i\psi(\ell)}).
\]
That is, our choice of $Q^{\a,\tau}$ provides a conservation law $\tau=$ const.
\end{remark}

\begin{remark}
From another point of view absolute continuity of $\mu^\a$ was discussed and proved in  the end of Section \ref{sept23sect74}, cf. \eqref{23sept205}.
\end{remark}

\subsection{Passing to the Potapov-de Branges gauge}
\subsubsection{Criterion for almost periodicity}
To pass from A-gauge to PdB-gauge we make the substitution
\[
\fB^\a(z,\ell)=\fA^\a(z,\ell)\fA^\a(0,\ell)^{-1}.
\]
As a result we get the canonical system in PdB gauge,
\begin{equation}\label{canonicalPdB20}
\fB^\a(z,\ell)j=j+iz \int_0^\ell \fB^\a(z,l)H^\a(l)d\mu^\a(l).
\end{equation}
This canonical system is determined by the positive matrix measure $H^\a d\mu^\a$. We denote by $(\dots)'$ the derivative in $z$; in particular, \eqref{canonicalPdB20} implies
\begin{equation}\label{7jul201}
\fB^\a(0,\ell)'j= i \int_0^\ell H^\a(l)d\mu^\a(l).
\end{equation}

We will use Lemma \ref{lemma26jun1} to study almost periodicity of the matrix measure $H^\a d\mu^\a$. Thus, we need a relation for its integrals over intervals.

\begin{lemma}
The Hamiltonian $H^\a(\ell)$ obeys the following identity 
\begin{equation}\label{13apr2}
\int_\ell^{L+\ell} H^\a(l)d \mu^\a(l)=\fA^{\a}(0,\ell)\left\{\int_0^L H^{\a-\eta \ell}(l)d\mu^{\a-\eta\ell}(l)\right\} \fA^{\a}(0,\ell)^*.
\end{equation}
\end{lemma}

\begin{proof}
As a consequence of the chain rule \eqref{chainruleAgauge} for $\fA^\a(z,\ell)$, the transfer matrix $\fB^\a(z,\ell)$ obeys 
\begin{equation}\label{13apr1}
\fB^\a(z,\ell+L)=\fB^\a(z,\ell) \fA^{\a}(0,\ell)\fB^{\a-\eta \ell}(z,L) \fA^{\a}(0,\ell)^{-1}.
\end{equation}
Differentiating \eqref{13apr1} in $z$ and evaluating at $z=0$ gives
\[
(\fB^\a)'(0,\ell+L)=(\fB^\a)'(0,\ell)+ \fA^{\a}(0,\ell)(\fB^{\a-\eta \ell})'(0,L) \fA^{\a}(0,\ell)^{-1}.
\]
Multiplying from the right by $j$, using \eqref{7jul201} for each term, and using $\fA^\a(0,\ell)^{-1} j = j \fA^\a(0,\ell)^*$ (since $\fA^\a(0,\ell) \in \SU(1,1)$) gives \eqref{13apr2}.
\end{proof}

Since
$$
\int_0^L H^{\a}(l)d\mu^\a(l)=-i\fB^\a(0,L)'j=-i(\fA^\a(0,L))'\fA^\a(0,L)^{-1}j
$$
is a continuous function of $\a$ on the compact abelian group $\pi_1(\Omega)^*$, the internal term in \eqref{13apr2} is almost periodic in $\ell$. 
Thus the almost periodicity of the whole expression is guaranteed by almost periodicity of $\fA^\a(0,\ell)$ with respect to the variable $\ell$.

The main result of this section is the following proposition. 
\begin{theorem}\label{thdbdb} Under the assumptions of Theorem \ref{th2m},
$\fA^\a(0,\ell)$ is almost periodic in $\ell$. For a generic $\eta$, conditions \eqref{22may4}, \eqref{22may5} are also necessary. 
\end{theorem}

\begin{remark} 
Note that if $0\not\in \sE$, then $\fA^\a(0,\ell)$ is even unbounded. 
Moreover, \eqref{22may5} means that $0$  is not an end of a gap $a_j\not =0$, $b_j\not= 0$ for all $j$. Respectively, in this case each gap contains a critical point,
$(c_*)_j\in (a_j,b_j)$ for all $j$.
\end{remark}

\begin{proof}[Proof of Theorem \ref{thdbdb}]
We start with the following remark. Condition \eqref{22may4} means exactly that the function
$$
\cZ(D)=\sum_j (\omega(x_j,\sE_*)-\omega(a_j,\sE_*))\e_j \mod 1
$$
on the set $\cD(\sE)$ is continuous. Therefore $\fz(\a,\tau):=\cZ(D)$ for $(\a,\tau)=\pi(D)$ is continuous on $\pi_1(\Omega)^*\times\bbT$.

We will use the representation \eqref{9mar203} for $\fA^\a(z,\ell)$. Without loss of generality, we can assume that $\Theta(0)=0$, see remark above, that is, $\Lambda_{\Theta(0)}(\ell)=I$. Condition \eqref{22may5} implies that $R^{\a,\tau}(0)$ is continuous in $\alpha$, respectively
$s_\pm^\a(0)$ are well defined, $s_-^\a(0)=\overline{s_+^\a(0)}$, $s^\a_+(0)$ continuous and 
$$
\sup_{\a}|s_+^\a(0)|^2<1.
$$
On the other hand we do not have any control on the inner function $\iota^\a$, see \eqref{11may1} (actually, we do not know whether it is a Blaschke product or not). To overcome this problem we use the identity
$$
\iota^\a(z)=\frac{e^{i\varphi_*}  \Phi_\sharp  K^{\tilde \alpha}}{K^{\alpha}}=
\frac{e^{i\varphi_*} \Phi_\sharp  K^{\tilde \alpha} - \bar \tau e^{-i\varphi_*} \Phi K_\sharp^{\tilde \alpha}}
{K^{\alpha} - \tau K_\sharp^{\alpha}} 
\frac{1 - \tau s_+^\a}{1 - \bar \tau s_-^\a}= \Delta^{\a,\tau}(\u^{\a,\tau})^2,
$$
where
\begin{equation}\label{28apr1}
 \u^{\a,\tau}(z):=\sqrt{\frac{1 -\tau s_+^\a(z)}{1 - \bar \tau s_-^\a(z)}},\quad  \u^{\a}:= \u^{\a,1}.
\end{equation}
Recall that
$\Delta^{\a,\tau}$ here is
the  Blaschke product \eqref{28apr2b}.

 Thus, $\Pi_\a(z)$ is now a product $
\Pi_\a(z)=\Pi^\Delta_\a(z)\Pi^s_\a(z)$, 
in which 
$$
\Pi^s_\a(z)=\begin{pmatrix}{\u^\a(z)}&0\\0&{\u^\a(z)}^{-1}\end{pmatrix}
\frac{\begin{pmatrix} 1&s_-^\a(z)\\ s_+^\a(z)&1
\end{pmatrix}}{\sqrt{1-s_-^\a(z)s_+^\a(z)}},
$$
and $\Pi^s_\a(0)$ is well defined and represents continuous matrix function with values in $\SU(1,1)$. The first
factor
$$
 \Pi^\Delta_\a(z)=\begin{pmatrix}\sqrt{\Delta^\a(z)}&0\\0&\sqrt{\Delta^\a(z)}^{-1}
\end{pmatrix}
$$
is given it terms of the Blaschke product $\Delta^{\a}(z)$ with well localized zeros and poles (one zero or pole in one gap depending on the divisor 
$D$ defined by the inverse Abel map $D=\pi^{-1}(\a,1)$).

Since $\fA^\a(z,\ell)$ is entire, by 
\eqref{9mar203} we have that
$$
(\Pi^\Delta_\a(z))^{-1}\Pi^\Delta_{\a-\eta\ell}(z)=\Pi^s_\a(z)\fA^\a(z,\ell) (\Pi^s_{\a-\eta \ell}(z))^{-1}
$$
has limit value at $z=0$, and moreover $| \Delta_{\a-\eta\ell}(z)/\Delta_\a(z)|\to 1$. The limit of the argument of $\Delta_{\a-\eta\ell}(z)/\Delta_\a(z)$  can be represented  in terms of harmonic measures, see Section \ref{sectabm}, particularly \eqref{7oct5},
$$
\sum_j \left((\omega (x_j^{\a-\eta\ell},\sE_*)-\omega(a_j,\sE_*))\e_j^{\a-\eta\ell}-(\omega (x_j^{\a},\sE_*)-\omega(a_j,\sE_*))\e_j^{\a}\right) \mod 1
$$
what is $\fz(\a-\eta\ell,1)-\fz(\a,1)$.
Thus finally 
$$
\fA^\a(0,\ell)=(\Pi^s_\a(0))^{-1}\begin{pmatrix} e^{\pi i(\fz(\a-\eta\ell,1)-\fz(\a,1))}&0\\0&e^{-\pi i(\fz(\a-\eta\ell,1)-\fz(\a,1))}\end{pmatrix}
\Pi^s_{\a-\eta\ell}(0)
$$
is almost periodic in $\ell$.

Conversely, from the representation \eqref{9mar203}
 in the generic position we can conclude that almost periodicity of $\fA^\a(0,\ell)$ should imply continuity 
of $s_+^\a(0)$ and of the limit argument of the ratio $ \Delta_{\a-\eta\ell}(z)/\Delta_\a(z)$ as $z\to 0$. These both functions, being expressed in terms of $\cD(\sE)$ are continuous if and only if 
\eqref{22may4} and \eqref{22may5} hold.
\end{proof}

\begin{proof}[Proof of Theorem~\ref{th2m}]
Theorem~\ref{thdbdb} proves the case $s_+ = s_+^\a$, $\a \in \pi_1(\Omega)^*$. By Cor.~\ref{corSE}, any $s_+ \in \cS(\sE)$ is of the form $\begin{pmatrix} s_+ & 1 \end{pmatrix} \simeq \begin{pmatrix} s_+^\a & 1 \end{pmatrix} \cU$ for some $\cU \in \SU(1,1)$. The corresponding transfer matrices in PdB-gauge are obtained by the conjugation $\fB(z,\ell) = \cU^{-1} \fB^\a(z,\ell) \cU$ which preserves PdB-gauge, acts on the Hamiltonian by $H(\ell) = \cU^{-1} H^\a(\ell) (\cU^{-1})^*$, and preserves almost periodicity.
\end{proof}

\subsubsection{Symmetric canonical system in PdB gauge and counterexample (geometric progression)}

In this section we demonstrate an example of a canonical system associated to a homogeneous spectrum $\sE$ such that the corresponding Hamiltonian in PdB gauge is not almost periodic. The easiest way to violate conditions  \eqref{22may5} and simultaneously
\eqref{22may4} is to consider a set, so that the ends of gaps form geometric progressions. 
Such set is homogeneous.
 We will show that at least generically the associated Hamiltonian is not almost periodic.

Let $\sE_s$ be symmetric, i.e., $x\in\sE_s\Rightarrow(-x)\in\sE_s$ and  $0\in\sE_s$. Using the substitution $\l=z^2$ we can pass to a semi-bounded set $\sF=\bbR_+\setminus\cup_{j} (a_j,b_j)$. We say that a character $\a_s$ is  symmetric   in $\bbC\setminus \sE_s$ if it is  generated by a character $\a\in\pi_1(\bbC\setminus\sF)^*$.
First we describe certain specific properties of Hamiltonians with a symmetric spectral set \cite{Y21}. They are diagonal in the standard form for de Branges canonical systems,
see \eqref{15oct8}.

As soon as the domain 
$\Omega=\bbC\setminus\sE_s$ is of Widom type and DCT holds the coefficients of a canonical  system 
in PdB gauge corresponding to a symmetric  character $\a_s$
are of the form
\begin{equation}\label{15oct8}
\bA(z,\ell)\cJ=\cJ-z \int_0^\ell\bA(z,l)\begin{bmatrix}d\nu_1(l) &0\\
0&d\nu_2(l)
\end{bmatrix},\quad  \cJ=\begin{bmatrix}0&1\\-1&0
\end{bmatrix}.
\end{equation}
Moreover, the measures $d\nu_j$ can be given explicitly in terms of special functions (reproducing kernels and their limits), see
Theorem \ref{th20apr1} below.

Note that the normalization point $z_*=i$ corresponds to $\l_*=-1$. In this subsection we assume that the complex Martin function in $\Omega$ is normalized by  $\Theta(\l_*)=i$ and its  additive character is denoted by $\eta$.

\begin{lemma} {\cite{Y21}} Let $k^\a(\l,\l_0)$ denote the reproducing kernel in $\cH^2(\a,\bbC\setminus \sF)$. Then the limit
$$
v_\a(\l)=\lim_{\l_0\to-\infty}\frac{k^\a(\l,\l_0)}{k^\a(\l_*,\l_0)}
$$
exists and represents a continuous function in $\a$. Moreover,   the limit
$$
\fv_\a(\a-\eta \ell):=\lim_{\l\to-0}\frac{v_{\a}(\l)}{v_{\a-\eta \ell}(\l)}.
$$
exists for $\ell\in\bbR_+$ and represent a continuous function in $\ell$.
\end{lemma}

\begin{theorem}\label{th20apr1}{\cite{Y21}}
Let $\fj$ be the character generated by $\sqrt\l$ in $\bbC\setminus \sF$. Then the coefficients of the canonical system \eqref{15oct8}
are of the form
\begin{equation*}\label{20apr1000}
d\nu_1(\ell)=d\nu^{\a+\fj}(\ell),\quad d\nu_2(\ell)=d\nu^{\a}(\ell)
\end{equation*}
and
\begin{equation}\label{20apr1}
d\nu^{\a}(\ell)=\frac{-\fv_{\a}(\a-\eta \ell)^2}{\kappa(\a)}e^{2\ell}de^{-2\ell}k^{\a-\eta \ell}(\l_*,\l_*),
\end{equation}
where $\kappa(\a)=k^\a(\l_*,\l_*)+k^{\a+\fj}(\l_*,\l_*)$. Moreover, if A-L condition is violated in the symmetric domain $\Omega=\bbC\setminus \sE_s$, then the measures $\nu^{\a}$ and $\nu^{\a+\fj}$ are mutually singular. 
\end{theorem}

Based on this we get the following proposition
\begin{proposition}
Let
\[
\nu^\a_L(\ell)=\int_{\ell}^{L+\ell}d\nu^\a(l)\quad\text{and\quad} \nu_L(\a)=\nu^\a_L(0).
\] 
Then
\begin{equation}\label{20apr3}
\nu^\a_L(\ell)
= \frac{\kappa(\a-\eta \ell)}{\kappa(\a)}\fv_\a(\a-\eta \ell)^2 \nu_L(\a-\eta \ell).
\end{equation}
\end{proposition}

\begin{proof}
By definition
$$
\fv_\a(\a-\eta(\ell+l))=\lim_{\l\to-0}\frac{v_{\a}(\l)}{v_{\a-\eta \ell}(\l)}\frac{v_{\a-\eta \ell}(\l)}{v_{(\a-\eta \ell)-\eta l}(\l)}=
\fv_\a(\b)\fv_{\b}(\b-\eta l)
$$
with $\b=\a-\eta \ell$.
Therefore the same change of variable ($\ell$ is fixed) in \eqref{20apr1} provides
$$
d\nu^{\a}(\ell+l)=\frac{-\fv_{\a}(\a-\eta \ell-\eta l)^2e^{2 l}de^{-2l}k^{\a-\eta \ell-\eta l}(\l_*,\l_*)}{\kappa(\a)}
=\frac{\kappa(\b)}{\kappa(\a)}\fv_{\a}(\b)^2
d\nu^{\b}(l).
$$
Respectively, we have
$$
\int_{\ell}^{L+\ell} d\nu^\a(l)=\int_{0}^{L} d\nu^\a(l+\ell)=\frac{\kappa(\b)}{\kappa(\a)}\fv_{\a}(\b)^2
\int_0^L d\nu^{\b}(l),
$$
that is, \eqref{20apr3} with $\b=\a-\eta \ell$.
\end{proof}

\begin{remark}
The functions $\kappa(\a)$ and $\nu_{L}(\a)$ are continuous in $\pi_1(\bbC\setminus\sF)^*$. The almost periodicity for the diagonal entries of the Hamiltonian 
in \eqref{15oct8} are reduced to  the question: is it possible to extend $\fv_{\a}(\b)$
by continuity on the hull $\text{clos}\{\b=\a-\eta \ell:\ \ell\in\bbR\}$?
\end{remark}

Now we will demonstrate that for a geometric progressions set, at least generically (non algebraic numbers) the answer is \textit{no}.
Let $F=\bbR_+\setminus\cup_{n\in\bbZ}(a_n,b_n)$ be formed by a geometric progression, i.e. 
$$
a_n=\rho^na_0,\quad b_n=\rho^n b_0,\quad 0<a_0<b_0<\rho a_0, \
$$
We have an automorphism in $\bbC\setminus F$: $\l\mapsto\rho\lambda$. We identify a character $\a$ with the sequence $\{\a_k\}_{k\in\bbZ}$ of its values on the standard generators $\a_k=\a(\g_k)$. If $f(\l)$ has a character $\a$, then $f(\rho\l)$ has the character 
$S\a\simeq\{S\a\}_{k\in\bbZ}$, where
$$
(S\a)_k=\a_{k+1}.
$$

For the Martin function we have 
$$
\Theta(\rho\l)=r\Theta(\l), \quad\text{where} \ r=\frac{\eta_1}{\eta_0}.
$$
\begin{lemma}
For a set $F$ forming by a geometric progression
\begin{equation}\label{17apr1}
r k^\a(\rho\l,\rho \l_0)=k^{S\a}(\l,\l_0), \quad v^{\a}(\rho\l)=v^{\a}(\rho\l_*)v^{S\a}(\l).
\end{equation}
Respectively, if $\fv_\a(\b)$ is well defined, then
\begin{equation}\label{17apr2}
\fv_{\a}(\b)=
\frac{v_{\a}(\l_*/\rho)}{v_{\b}(\l_*/\rho)}
\fv_{S^{-1}\a}(S^{-1}\b).
\end{equation}
\end{lemma}

\begin{proof}
For $f\in \cH^2(\a)$ we have
$$
f(\rho\l_0)=\int\overline{k^\a(\l,\rho {\l_0})}f(\l) d\Theta(\l)=\int\overline{k^\a(\rho\l,\rho {\l_0})}f(\rho\l) d\Theta(\rho\l)
$$
$$
=r\int\overline{k^\a(\rho\l,\rho {\l_0})}f(\rho\l) d\Theta(\l)=\int\overline{k^{S\a}(\l, {\l_0})}f(\rho\l) d\Theta(\l).
$$
Hence we get \eqref{17apr1}.
In its turn
$$
\fv_{\a}(\b)=\lim_{\l\to-0}\frac{v_{\a}(\l/\rho)}{v_{\b}(\l/\rho)}=\frac{v_{\a}(\l_*/\rho)}{v_{\b}(\l_*/\rho)}
\lim_{\l\to-0}\frac{v_{S^{-1}\a}(\l)}{v_{S^{-1}\b}(\l)}
$$
and we have \eqref{17apr2}.
\end{proof}

\begin{proposition}
Let $r$ be a non-algebraic number and $S\a=\a$. Then the function $\fv_\a(\a-\eta \ell)$ is not almost periodic, i.e., the associated canonical system in the PdB gauge is not almost periodic.
\end{proposition}

\begin{proof}
Since frequencies $\{r^k\}$ are rationally independent clos$\{\a-\eta \ell:\ \ell\in\bbR\}=\pi_1(\bbC\setminus F)^*$.
If $S\a=\a$, $S\b_*=\b_*$ and $v_\a(\l_*/\rho)>v_{\b_*}(\l_*/\rho)$, assuming continuity $\fv_\a(\b)$, we get a contrudiction
$$
\lim_{\a-\eta \ell\to\b_*}\fv_\a(\a-\eta\ell)=\infty.
$$
In particular,  we can choose $\b_*= \a+\fj$.
\end{proof}

\appendix
\section{Nagy-Foias functional model and unitary nodes}

Let $E_1, E_2$ be complex Euclidean spaces. By $L^2(E_2)$ we denote the space of square integrable vector functions from $\bbT$ to $E_2$ with respect to Lebesgue measure on the unit circle $\bbT$. By $H^2(E_2)$ we denote the Hardy space spanned by $E_2$-vector valued linear combinations of 
$\{\z^n\}_{n\ge 0}$. Respectively, $H^2_-(E_2)$ is spanned by  $\{\z^n\}_{n\le -1}$. Note that $\overline{\z}=\z^{-1}$ on $\bbT$.

Let $S(\z)$ be an operator valued Schur class function in $\bbD=\{\z:\ |\z|<1\}$,
$$
S(\z): E_1\to E_2\ \text{for} \ \z\in\bbD,\quad
I_{E_1} -S(\z)^*S(\z)\ge 0.
$$
It has boundary values a.e. on $\bbT$. First, we assume that
\begin{equation}\label{2jul1app}
I_{E_1} -S(\z)^*S(\z)= 0, \quad \z\in\bbT.
\end{equation}
We decompose $L^2(E_2)$ into three components
$$
L^2(E_2)=H^2_-(E_2)\oplus K_S\oplus S H^2(E_1), \quad K_S:= H^2(E_2)\ominus S H^2(E_1).
$$
We consider the multiplication operator by $\overline{\z}$ in $L^2(E_2)$. In the given decomposition it acts trivially as a shift on two components, to be precise,
$$
\overline{\z}:\z H^2(E_1)\to H^2(E_1), \quad \overline{\z}:H_-^2(E_2)\to \overline{\z}H_-^2(E_2).
$$
The nontrivial part of the operator deals with
\begin{equation}\label{2jun2}
\overline{\z}: K_S\oplus  S E_1\to \overline{\z}E_2\oplus K_S.
\end{equation}

\begin{definition}
Let $K$ be a Hilbert space.
By a \textit{unitary node} we mean a unitary operator $U$ acting from $K\oplus E_1$ to $K\oplus E_2$.
$K$ is  called \textit{state space} and $E_1, E_2$ are called coefficient spaces. The operator function
\begin{align*} 
\cE(\z)=\cE(\z,U)=&P_{E_2}(I_{K\oplus E_2}-\z U P_K)^{-1}U|_{E_1}\\
=&P_{E_2}U(I_{K\oplus E_1}-\z P_K U)^{-1}|_{E_1}
\end{align*}
is called the \textit{characteristic function} of the unitary node. Here $P_K$ and $P_{E_2}$ are the orthogonal projections onto the corresponding subspaces.
\end{definition}

\begin{proposition} Let the unitary node $U: K\oplus E_1\to K\oplus E_2$ be given by  \eqref{2jun2} with $K=K_S$.
Then $S(\z)$ is its characteristic function. 
\end{proposition}
\begin{proof} 
For $\z_0\in \bbD$, we have
$$
(I_{K\oplus E_2}-\z_{0} U P_K)^{-1}\overline{\z}S(\z) u=f(\z)+\overline{\z}\cE(z_0)u, \quad \z\in\bbT,
$$
where $f\in K=K_S$ and $u\in E_1$. Therefore
$$
\overline{\z}S(\z) u=(I_{K\oplus E_2}-\z_{0} U P_K)(f(\z)+\overline{\z}\cE(z_0)u)
$$
$$
=(1-\z_0\overline{\z})f(\z)+\overline{\z}\cE(z_0)u.
$$
Multiplying by $\z$ we can extend this identity inside the disk, since $f\in H^2(E_2)$. We obtain
$$
S(\z) u=(\z-\z_0)f(\z)+\cE(z_0) u.
$$
Thus, $S(z_0)u=\cE(z_0)u$ for arbitrary $\z_0\in\bbD$ and $u\in E_1$.
\end{proof}

Let us show that the characteristic function of an arbitrary node belongs to the Schur class. In fact, we prove the so-called (generalized) 
Christoffel--Darboux identity.

\begin{theorem}
For a unitary node $U:K\oplus E_1\to K\oplus E_2$ define
\[
\cG(z)=P_K U(I_{K\oplus E_1}-\z P_K U)^{-1}|_{E_1}.
\]
Then
\begin{equation}\label{2jun4}
\frac{I_{E_1} -\cE(\z_0)^{*} \cE(\z)}{1-\overline{\z_0}\z}=\cG(\z_0)^*\cG(\z).
\end{equation}
\end{theorem}

\begin{proof}
We start with an evident identity
\[
U^*(P_K\oplus P_{E_2}) U=P_K\oplus P_{E_1},
\]
which we rewrite into the form
\[
P_{E_1}-U^*P_{E_2} U=U^*P_KU- P_{K}.
\]
Now we multiply both parts by $(I_{K\oplus E_1}-\z P_K U)^{-1}$ from the right and the conjugated to $(I_{K\oplus E_1}-\z_0 P_K U)^{-1}$ from the left. After that we restrict and project the identity to $E_1$.
Note that $P_{E_1}(I_{K\oplus E_1}-\z P_K U)^{-1}=P_{E_1}$. Therefore in the LHS we get
\begin{equation}\label{2jun5}
P_{E_1}- \cE(\z_0)^* \cE(z).
\end{equation}
To compute RHS we note that
\[
P_{K}(I_{K\oplus E_1}-\z P_K U)^{-1}P_{E_1}=(P_{K}+\z P_K U(I_{K\oplus E_1}-\z P_K U)^{-1})P_{E_1}=\z\cG(\z).
\]
Then
\begin{align*}
& P_{E_1}(I_{K\oplus E_1}-\overline{\z_0}  U^*P_K)^{-1}(U^*P_KU- P_{K})(I_{K\oplus E_1}-\z P_K U)^{-1}P_{E_1} \\
& \qquad = \cG(\z_0)^*\cG(\z)-\overline{\z_0}\cG(\z_0)^*\,\z\cG(\z).
\end{align*}
In a combination with \eqref{2jun5} we get \eqref{2jun4}.
\end{proof}

Finally we extend the concept of Nagy--Foias functional model to  the general case of Schur functions, that is we remove the condition
\eqref{2jul1app}. Let $\Sigma(\z)=\sqrt{I-S(\z)^* S(\z)}$, $\z\in\bbT$. We introduce
\[
\overline{\Sigma L^2(E_1)}=\text{clos}_{L^2(E_2)}\{g(\z)=\Sigma(\z) f(\z):\ f\in L^2(E_1)\}.
\]
We define the three term decomposition
\begin{equation}\label{2jul7}
L^2(E_2)\oplus \overline{\Sigma L^2(E_1)}=(H^2_{-}(E_2)\oplus 0)\oplus K_S\oplus (S\oplus\Sigma) H^2(E_1),
\end{equation}
where $K_s=H^2(E_2)\oplus \overline{\Sigma L^2(E_1)}\ominus (S\oplus\Sigma) H^2(E_1)$. 

For operator multiplication by $\overline{\z}$ we have again incoming and outgoing subspaces on which this operator acts as shift. The remaining part defines the unitary node
\begin{equation}\label{2jul8}
\overline{\z}: K_S\oplus (S\oplus\Sigma) E_1\to (\overline{\z} E_2\oplus 0)\oplus K_S.
\end{equation}
As a simple exercise we get $\cE(\z, U)=S(\z)$ for the unitary node \eqref{2jul8}.

Let $f=f_1\oplus f_2\in K_S$. Since the first component belongs to $H^2(E_2)$ we can define a bounded fumctional
\[
f=f_1\oplus f_2\mapsto \langle f_1(\z_0), v\rangle,\quad \z_0\in\bbD,\ v\in E_2.
\]
We denote the corresponding functional $k_{\z_0, v}(\z)\in K_S$.

\begin{lemma}\label{l1.4}
\begin{equation}\label{8jul9}
k_{\z_0,v}(\z)
=\frac{(I_{E_2}-S(\z)S(\z_0)^*)v}{1-\z\overline{\z_0}}\oplus 
\frac{-\Sigma(\z)S(\z_0)^* v}{1-\z\overline{\z_0}}
\end{equation}
In particular,
\begin{equation}\label{2jun10}
\langle k_{\z_1,v_1},k_{\z_2,v_2}\rangle=\left\langle\frac{I_{E_2}-S(\z_2)S(\z_1)^*}{1-\z_2\overline{\z_1}}v_1, v_2 \right\rangle.
\end{equation}

\end{lemma}
\begin{proof} In fact we need to project the vector
\[
\frac v{1-\z\overline{\z_0}}\oplus 0\in H^2(E_2)\oplus\overline{\Sigma L^2(E_1)}
\]
on $K_S$. The formula for this projection is given in the next lemma.
\end{proof}

\begin{lemma} Let $P_{K_S}$ be the orthogonal projection from the space $H^2(E_2)\oplus \overline{\Sigma L^2(E_1)}$, and
$P_+$ from $L^2(E_1)$ to $H^2(E_1)$. Then
\[
P_{K_S}=I -(S\oplus\Sigma)P_+(S\oplus\Sigma)^*.
\]
In particular $f=f_1\oplus f_2\in H^2(E_2)\oplus \overline{\Sigma L^2(E_1)}$ belongs to $K_S$ if and only if 
\begin{equation}\label{2jul10}
S^*f_1+\Sigma f_2\in H^2_-(E_1).
\end{equation}
\end{lemma}

\begin{proof}
We have
\[
P_{K_S}f=f-(S\oplus\Sigma) g,\quad g\in H^2(E_1),
\]
such that for an arbitrary $h\in H^2(E_1)$
\[
0=\langle f-(S\oplus\Sigma) g, (S\oplus\Sigma) h\rangle=\langle (S\oplus\Sigma)^* f- g,  h\rangle
\]
Therefore $(S\oplus\Sigma)^* f- g\in H^2_-(E_1)$ and $g=P_+(S\oplus\Sigma)^* f$.
\end{proof}

Note that
\[
S(\z)^*\frac{v}{1-\z\overline{\z_0}}=\overline{\z}\frac{(S(\z)^*-S(\z_0)^*)v}{\overline{\z}-\overline{\z_0}}+S(\z_0)^*\frac{v}{1-\z\overline{\z_0}}
\]
and the first term in RHS belongs to $H^2_-(E_1)$. Therefore
\[
P_+(S\oplus\Sigma)^*\left(\frac{v}{1-\z\overline{\z_0}}\oplus 0\right)=\frac{S(\z_0)^*v}{1-\z\overline{\z_0}}
\]
This finalize the proof of Lemma \ref{l1.4}.

Using \eqref{2jul10} to a vector $f\in K_S$ we associate $f_*\in H^2(E_1)$ by
\[
\overline{\z f_*(\z)}=S(\z)^*f_1(\z)+\Sigma(\z) f_2(\z), \quad \z\in\bbT.
\]
For $u\in E_1$, by $\overline u$ we mean certain anti--linear involution in $E_1$,
\[
\langle\overline{u_1}, \overline{u_2}\rangle={\langle u_2,u_1\rangle},\quad u_1,u_2\in E_1.
\]
Therefore the following bounded functional also has sense
\begin{equation}\label{2jul11}
f\mapsto \langle u ,f_*(\z_0)\rangle,\quad f\in K_S,\ \z_0\in\bbD,\ u\in E_1.
\end{equation}
\begin{lemma}\label{l1.6}
The functional $(k_*)_{\z_0,u}$ defined by \eqref{2jul11}, i.e.,
\[
\langle f, (k_*)_{\z_0,u} \rangle=\langle u, f_*(\z_0) \rangle,\quad f\in K_S,
\]
 is given by
$$
(k_*)_{\z_0,u}(\z)=\frac{S(\z)-S(\z_0)}{\z-\z_0}\overline{u}\oplus \frac{\Sigma(\z)}{\z-\z_0}\overline{u}.
$$
\end{lemma}
\begin{proof}
We have
$$
\langle f, (k_*)_{\z_0,u}\rangle=\left\langle \overline{\z f_*(\z)}, \frac{\overline{u}}{\z-\z_0}\right\rangle
 =\overline{\left\langle  f_{*}(\z), \frac{u}{1-\z\overline{\z_0}}\right\rangle}=\langle u, f_*(\z_0) \rangle.
$$
\end{proof}

\begin{proposition}
The following set
\[
\spann \{ k_{\z_1, v}, (k_*)_{\z_2,u}: \ \z_j\in\bbD,\quad u\in E_1, v\in E_2\}
\]
is dense in $K_S$.
\end{proposition}
\begin{proof}
If $f\in K_S$ is orthogonal to all such vectors, then $f_1=0$ (by Lemma \ref{l1.4}) and $f_*=0$ (due to Lemma \ref{l1.6}). Therefore
\[
\Sigma(\z)f_2(\z)=0, \quad \text{for a.e.}\  \z\in\bbT.
\]
Since $f_2\in\overline{\Sigma L^2(E_1)}$, we get $f_2=0$.
\end{proof}

We note that for the dual unitary node
\[
U^*:K\oplus E_2\to K\oplus E_1
\]
we have
\[
\cE(\z,U^*)=P_{E_1}U^*(I_{K\oplus E_2}-\z P_K U^*)^{-1}|_{E_2}=\cE(\overline{\z})^*,\quad \cE({\z})= \cE({\z},U).
\]
Therefore the corresponding Christoffel--Darboux identity, see \eqref{2jun4}, is of the form
\begin{equation}\label{3jul2}
\cG_*(\z_0)^*\cG_*(\z)=\frac{I_{E_2}-\cE(\overline{\z_0})\cE(\overline{\z})^*}{1-\overline{\z_0}\z},
\end{equation}
where
\[
\cG_*(\z)=P_K U^*(I_{K\oplus E_2}-\z P_K U^*)^{-1}|_{E_2}.
\]

\begin{theorem}
Let $U:K\oplus E_1\to K\oplus E_2$ be a unitary node with the characteristic function $S(\z)=\cE(\z,U)$. Then the densely defined operator
\begin{equation}\label{3jun1}
\cF:K_S\to K\quad\text{s.t.}\quad\cF k_{\z_0,v}=\cG_*(\overline{\z_0})v, \quad \cF (k_*)_{\z_0,u}=\cG(\z_0)\overline{u}
\end{equation}
is an isometry.
\end{theorem}

\begin{proof}
By
\eqref{2jun10} and \eqref{3jul2} we have
\[
\langle k_{\z_1,v_1},k_{\z_2,v_2}\rangle_{K_S}=\langle \cG_*(\overline{\z_1})v_1,\cG_*(\overline{\z_2})v_2\rangle_{K}.
\]
Since
\[
\langle (k_*)_{\z_1,u_1},(k_*)_{\z_2,u_2}\rangle_{K_S}=\langle u_2, ((k_*)_{\z_1,u_1})_*(\z_2)\rangle
\]
and
\[
\overline{((k_*)_{\z_1,u_1})_*(\z)}=\frac{I_{E_1}- S(\z)^* S(\z_1)}{1-\overline{\z}\z_1}\overline{u_1}
\]
we have
\[
\langle (k_*)_{\z_1,u_1},(k_*)_{\z_2,u_2}\rangle_{K_S}=\left\langle \frac{I_{E_1}- S(\z_2)^* S(\z_1)}{1-\overline{\z_2}\z_1}\overline{u_1},\overline{u_2}\right\rangle.
\]
By \eqref{2jun4} we get
\[
\langle (k_*)_{\z_1,u_1},(k_*)_{\z_2,u_2}\rangle_{K_S}=\langle\cG(\z_1)\overline{u_1}, \cG(\z_2)\overline{u_2}\rangle.
\]
Finally
\[
\langle (k_*)_{\z_1,u},k_{\z_2,v}\rangle_{K_S}=\left\langle\frac{S(\z_2)-S(\z_1)}{\z_2-\z_1}\overline{u}, v\right\rangle.
\]
At the same time
\begin{align*}
(\z_2-\z_1)\cG_*(\overline{\z_2})^*\cG(\z_1)
& = P_{E_2} (I_{K\oplus E_2}-\z_2 U P_K)^{-1} \z_2 UP_K U(I_{K\oplus E_1}-\z_1 P_K U)^{-1}P_{E_1} \\
& \qquad -P_{E_2} (I_{K\oplus E_2}-\z_2 U P_K)^{-1} U\,\z_1 P_K  U(I_{K\oplus E_1}-\z_1 P_K U)^{-1}P_{E_1} \\
& = -P_{E_2} U(I_{K\oplus E_2}-\z_1  P_K U)^{-1}P_{E_1}+P_{E_2} (I_{K\oplus E_1}-\z_2 U P_K)^{-1} UP_{E_1}.
\end{align*}
Therefore
\[
\langle \cG(\z_1)\overline{u},\cG_*(\overline{\z_2})v\rangle_{K}=\left\langle\frac{S(\z_2)-S(\z_1)}{\z_2-\z_1}\overline{u}, v\right\rangle=
\langle (k_*)_{\z_1,u},k_{\z_2,v}\rangle_{K_S}
\]
and the theorem is completely proved.
\end{proof}

\begin{corollary}
Let $U: K\oplus E_1\to K\oplus E_2$ be a unitary node with the characteristic function $S(\z)$. Let $\cF:K_S\to K$ be given by \eqref{3jun1}. Then $U$ acts unitary in $K\ominus\cF K_S$.
\end{corollary}

\begin{proof}
Note that $U$ acts unitary from $\cF K_S\oplus E_1$ to $\cF K_S\oplus E_2$. Therefore it acts unitary in the orthogonal complements
\[
(K\oplus E_1)\ominus (\cF K_S\oplus E_1)=K\ominus\cF K_S\quad\text{and}\quad (K\oplus E_2)\ominus (\cF K_S\oplus E_2)=K\ominus\cF K_S. \qedhere
\]
\end{proof}


\begin{thebibliography}{88}

\bibitem{AkhiezerLevin}
N.~I.~Akhiezer and B.~Ya.~Levin, \emph{Generalizations of the Bernstein inequalities for derivatives of entire functions}, Studies in the Modern Problems of the Theory of Functions of a Complex Variable (Moscow) (1960), 111--165. 
	
\bibitem{Anc79} A.~Ancona, \emph{Une propri\'et\'e de la compactification de Martin d'un domaine euclidien}, 
Ann. Inst. Fourier (Grenoble) 29 (1979), no. 4, ix, 71--90.

\bibitem{ArgGil74} L.~Argabright, J.~Gil de Lamadrid, \emph{Fourier analysis of unbounded measures on locally compact abelian groups}, Memoirs of the American Mathematical Society, No. 145. American Mathematical Society, Providence, R.I., 1974. vi+53 pp.

\bibitem{ArgGil90} J. de Lamadrid, L.~N.~Argabright, \emph{Almost periodic measures}, Mem. Amer. Math. Soc. 85 (1990), no. 428, vi+219 pp.

\bibitem{ArmGar01} D.~H.~Armitage, S.~J.~Gardiner, \emph{Classical potential theory}, 
Springer Monographs in Mathematics. Springer-Verlag London, Ltd., London, 2001. xvi+333 pp. ISBN: 1-85233-618-8

\bibitem{AGr83} D. Z. Arov and L. Z. Grossman, \emph{Scattering matrices in the theory of extensions of isometric operators}. (Russian) Dokl. Akad. Nauk SSSR 270 (1983), no. 1, 17--20.

\bibitem{AD08} D.~Z.~Arov, H.~Dym, \emph{J-contractive matrix valued functions and related topics}. 
Encyclopedia of Mathematics and its Applications, 116. Cambridge University Press, Cambridge, 2008. xii+575 pp. ISBN: 978-0-521-88300-9 

\bibitem{Ben80} M.~Benedicks, \emph{Positive harmonic functions vanishing on the boundary of certain domains in $\bbR^n$}, Ark. Mat. 18 (1980), no. 1, 53--72. 

\bibitem{Bes54} A. S. Besicovitch, \emph{Almost periodic functions},  Dover Publications, Inc., New York, 1955. xiii+180 pp.

\bibitem{BD} R.~Bessonov, S.~Denisov, \emph{A spectral Szeg\H o theorem on the real line}, Adv. Math. 359 (2020), 106851, 41 pp. 

\bibitem{BLY1} R.~Bessonov, M.~Luki\'c, P.~Yuditskii, \emph{A theory of reflectionless canonical systems, I. Arov gauge and right limits}.

\bibitem{BS01} A.~Borichev, M.~Sodin, \emph{Krein's entire functions and the Bernstein approximation problem}, Illinois J. Math. 45 (2001), no. 1, 167--185. 

\bibitem{Chul81} V.\ Chulaevskii, Perturbations of a Schr\"odinger operator with periodic potential (Russian). \textit{Uspekhi Mat.\ Nauk} \textbf{36} (1981), 203--204.

\bibitem{CG} S.~Clark, F.~Gesztesy, \emph{Weyl-Titchmarsh M-function asymptotics, local uniqueness results, trace formulas, and Borg-type theorems for Dirac operators}, Trans. Amer. Math. Soc. 354 (2002), no. 9, 3475--3534. 

\bibitem{Craig} W. Craig, \emph{The trace formula for Schr\"odinger operators on the line}, Commun. Math. Phys. 126, 379--407 (1989).

\bibitem{dB} L.\ de Branges, \textit{Hilbert Spaces of Entire Functions}, Prentice-Hall, Inc., Englewood Cliffs, N.J., 1968.

\bibitem{DEY} D.~Damanik, B.~Eichinger, P.~Yuditskii, \emph{Szeg\H o's theorem for canonical systems: the Arov gauge and a sum rule}, J. Spectr. Theory (to appear),  arXiv:1907.03267, 

\bibitem{DubMatNov} B. A. Dubrovin, V. B. Matveev, and S. P. Novikov, \emph{Non-linear equations of Korteweg-de Vries type, finite-zone linear operators, and Abelian varieties}, Russian Math. Surv. 31:1, 59--146 (1976).

\bibitem{EKT18} J.~Eckhardt, A.~Kostenko, G.~Teschl, \emph{Spectral asymptotics for canonical systems}, J. Reine Angew. Math. 736 (2018), 285--315. 

\bibitem{EP73} A.~V.~Efimov, V.~P.~Potapov, \emph{J-expanding matrix-valued functions, and their role in the analytic theory of electrical circuits}, Uspehi Mat. Nauk 28 (1973), no. 1(169), 65--130. 


\bibitem{Egorova} I. E. Egorova, \emph{On a class of almost periodic solutions of the KdV equation with a nowhere dense spectrum}, Russian Acad. Sci. Dokl. Math. 45, 290--293 (1990)

\bibitem{EY12} A.~Eremenko, P.~Yuditskii,  \emph{Comb functions}. Recent advances in orthogonal polynomials, special functions, and their applications, 99--118, Contemp. Math., 578, Amer. Math. Soc., Providence, RI, 2012. 

\bibitem{GarSjo09} S.~J.~Gardiner, T.~Sj\"odin,  \emph{Potential theory in Denjoy domains}, Analysis and mathematical physics, 143--166, Trends Math., Birkh\"auser, Basel, 2009. 

\bibitem{GarMar} J.~B.~Garnett, D.~E.~Marshall, \emph{Harmonic measure}.
New Mathematical Monographs, 2. Cambridge University Press, Cambridge, 2005. xvi+571 pp.

\bibitem{GesHol} F.~Gesztesy, H.~Holden, \emph{Soliton equations and their algebro-geometric solutions. Vol. I.
(1+1)-dimensional continuous models}, Cambridge Studies in Advanced Mathematics, 79. Cambridge University Press, Cambridge, 2003. xii+505 pp.

\bibitem{GesKriTes} F. Gesztesy, M. Krishna, and G. Teschl, \emph{On isospectral sets of Jacobi operators}, Commun.
Math. Phys. 181, 631--645 (1996).

\bibitem{GoKr}
I. C. Gohberg, M.G.  Krein \emph{Theory and Applications of Volterra operators in Hilbert space}. Transl. Math. Monographs, Vol. 24, American Mathematical Society, Providence R.I., 1970.

\bibitem{H83} M.\ Hasumi, \textit{Hardy Classes on Infinitely Connected Riemann Surfaces}, Lecture Notes in Math.\ \textbf{1027}, Springer-Verlag, Berlin, 1983.

\bibitem{JM}  R. Johnson and J. Moser, \emph{The rotation number for almost periodic potentials}, Commun.
Math. Phys. 84, 403--438 (1982).

\bibitem{KY18} A.~Kheifets, P.~Yuditskii, \emph{Martin Functions of Fuchsian Groups and Character Automorphic Subspaces of the Hardy Space in the Upper Half Plane}, 
Complex Function Theory, Operator Theory, Schur Analysis and Systems Theory (A Volume in Honor of V.E. Katsnelson), 535--581, Oper. Theory Adv. Appl., 280, Birkh\"auser, 2020. 

\bibitem{Kot82} S.~Kotani, \emph{Ljapunov indices determine absolutely continuous spectra of stationary random one-dimensional Schr\"odinger operators}, Stochastic Analysis (Katata/Kyoto, 1982), pp. 225--247, North--Holland Math. Library 32, North--Holland, Amsterdam, 1984.

\bibitem{Koosis} P.~Koosis, \emph{The logarithmic integral. I.}, Corrected reprint of the 1988 original. Cambridge Studies in Advanced Mathematics, 12. Cambridge University Press, Cambridge, 1998. xviii+606 pp.

\bibitem{Krein97} M.~G.~Krein, \emph{On the theory of entire matrix-functions of exponential type}, Topics in interpolation theory (Leipzig, 1994), 361--371, Oper. Theory Adv. Appl., 95, Birkh\"auser, Basel, 1997. 

\bibitem{KLN} 
M. G. Krein, B. Ya. Levin, A. A. Nudelman, \emph{On a special representation of polynomials that are positive on a system of closed intervals and some applications}, Functional analysis, optimization and mathematical economics, Ed. L. Leifman, Oxford Univ. Press, 1990, 56-114.

\bibitem{Lev} B. M. Levitan, \emph{Approximation of infinite-zone potentials by finite-zone potentials}, Math
USSR Izvestija 20, 55--87 (1983).

\bibitem{LevSar} B. M. Levitan, I. S.  Sargsjan, \emph{Sturm-Liouville and Dirac operators}, 
Translated from the Russian. Mathematics and its Applications (Soviet Series), 59. Kluwer Academic Publishers Group, Dordrecht, 1991. xii+350 pp.

\bibitem{LevSav}  B. M. Levitan and A. V. Savin, \emph{Examples of Schr\"odinger operators with almost periodic
potentials and nowhere dense absolutely continuous spectrum}, Sov. Math. Dokl. 29, 541--544
(1984).

\bibitem{Mar41} R.~S.~Martin, \emph{Minimal positive harmonic functions}, Trans. Amer. Math. Soc. 49 (1941), 137--172. 

 \bibitem{MarOst75} V. A. Marchenko, I. V. Ostrovskii, \emph{A characterization of the spectrum of Hill's operator},
Mat. Sb. (N.S.), 97(139) (1975),  540--606; English transl. in Math USSR sb. 26 (1975).

\bibitem{MarOst87} V. A. Marchenko, I. V. Ostrovskii, \emph{Approximation of periodic by finite-zone potentials},
Selecta Mathematica Sovetica 6 (1987), 103--136.

\bibitem{McKvanM75} H. P. McKean, P. van Moerbeke, \emph{The Spectrum of Hill's Equation}, Invent. Math. 30 (1975),
217--274.

\bibitem{McKTru76} H. P. McKean, E.~Trubowitz, \emph{Hill's operator and hyperelliptic function theory in the presence
of infinitely many branch points}, Commun. Pure Appl. Math. 29 (1976), 143--226.

\bibitem{NF} B. Sz.-Nagy, C.  Foias, \textit{Analyse harmonique des op\'erateurs de l'espace de Hilbert}. (French) Masson et Cie, Paris, 1967 xi+373 pp.

\bibitem{SzN} B. Sz.-Nagy,  \textit{Unitary dilations of Hilbert space operators and related topics}. 
 Expository Lectures from the CBMS Regional Conference held at the University of New Hampshire, Durham, N.H., June 7-11, 1971. Conference Board of the Mathematical Sciences Regional Conference Series in Mathematics, No. 19. American Mathematical Society, Providence, R.I., 1974. viii+54 pp. 
 
\bibitem{NIK}
N.~K. Nikolskii, \emph{Treatise on the shift operator}, A Series of
  Comprehensive Studies in Mathematics, Spriger-Verlag, Berlin Heidelberg New
  York Tokyo, 1986.

\bibitem{PT1} L.\ Pastur, V.\ A.\ Tkachenko, On the spectral theory of the one-dimensional Schr\"odinger operator with limit-periodic potential (Russian), \textit{Dokl.\ Akad.\ Nauk SSSR} \textbf{279} (1984) 1050--1053.

\bibitem{PT2}  L.\ Pastur, V.\ A.\ Tkachenko, Spectral theory of a class of one-dimensional Schr\"odinger operators with limit-periodic potentials (Russian), \textit{Trudy Moskov.\ Mat.\ Obshch.}\ \textbf{51} (1988) 114--168.

\bibitem{PY06} F.~Peherstorfer, P.~Yuditskii, \emph{Almost periodic Verblunsky coefficients and reproducing kernels on Riemann surfaces}, J. Approx. Theory 139 (2006), no. 1-2, 91--106.

\bibitem{Pom92} Ch.~Pommerenke, \emph{Boundary behaviour of conformal maps}, Grundlehren der Mathematischen Wissenschaften [Fundamental Principles of Mathematical Sciences], 299. Springer-Verlag, Berlin, 1992. x+300 pp.

\bibitem{P60}
V.~P.~Potapov, \emph{The multiplicative structure of J-contractive matrix functions},
Amer. Math. Soc. Transl. (2) 15 1960 131--243.

  \bibitem{Rem11}
C.~Remling, \emph{The absolutely continuous spectrum of {J}acobi matrices},
  Ann. of Math. \textbf{174} (2011), no.~1, 125--171. 
 

\bibitem{Remling}
C.~Remling, \emph{Spectral theory of canonical systems}, 
De Gruyter Studies in Mathematics, 70. De Gruyter, Berlin, 2018. x+194 pp. 

\bibitem{Romanov} R.~Romanov, \emph{Canonical systems and de Branges spaces}, preprint. arXiv:1408.6022 (2014).

\bibitem{Sim15}
 B.~ Simon,
\emph{Basic complex analysis},
 A Comprehensive Course in Analysis, Part 2A,
Amer. Math. Soc.,
Providence, RI, 2015, XVIII+641.


\bibitem{SY97} M.~Sodin, P.~Yuditskii, 
\textit{Almost periodic Jacobi matrices with homogeneous spectrum, infinite-dimensional Jacobi inversion, and Hardy spaces of character-automorphic functions}. J. Geom. Anal. 7 (1997), no. 3, 387--435.

\bibitem{SY95} M.~Sodin, P.~Yuditskii, \emph{Almost periodic Sturm-Liouville operators with Cantor homogeneous spectrum},
Comment. Math. Helv. 70 (1995), no. 4, 639--658.

\bibitem{VYInv} A.~Volberg, P.~Yuditskii, \emph{Kotani--Last problem and Hardy spaces on surfaces of Widom type}, Invent. Math. 197 (2014), no. 3, 683--740.

\bibitem{Weyl}
Hermann Weyl, \emph{\"{U}ber beschra\"ankte quadratische {F}ormen, deren
  {D}ifferenz vollstetig ist}, Rend. Circ. Mat. Palermo {27} (1909),
  373--392.

\bibitem{W} H.\ Widom, \emph{Extremal polynomials associated with a system of curves in the complex plane}, Adv. in Math.\ {2} (1969), 127--232.

\bibitem{WActa} H. Widom,  \emph{The maximum principle for multiple-valued analytic functions}, Acta Math. 126,
63--82 (1971).

\bibitem{WAnn}
H. Widom,  \emph{$H^p$ sections of vector bundles over Riemann surfaces}, Ann. Math. 94, 304--324
(1971).

\bibitem{Y21} P.~Yuditskii, \emph{Direct Cauchy Theorem and Fourier integral in Widom domains}, J. Anal. Math. (to appear), arXiv:1812.00612.



\end{thebibliography}
\end{document}